\documentclass[reqno,english]{amsart}
\usepackage{amsfonts,amsmath,latexsym,verbatim,amscd,mathrsfs,color,array}
\usepackage[colorlinks=true]{hyperref}
\usepackage{amssymb,amsthm,graphicx,color}
\usepackage[hmargin=3cm, vmargin=3cm]{geometry}
\usepackage{bbm}
\usepackage{amssymb}
\usepackage{pdfsync}
\usepackage{epstopdf}
\usepackage{cite}
\usepackage{graphicx}
\usepackage{multirow}
\usepackage[font=bf,aboveskip=15pt]{caption}
\usepackage[toc,page]{appendix}
\usepackage[colorlinks=true]{hyperref}
\usepackage{bookmark}
\DeclareFontShape{OT1}{cmr}{bx}{sc}{<-> cmbcsc10}{}
\usepackage{float}
\usepackage{ulem}
\usepackage{syntonly}
\usepackage{mathtools}
\usepackage{bm}
\usepackage{amsfonts,amsmath,latexsym,verbatim,amscd,mathrsfs,color,array}
\usepackage[colorlinks=true]{hyperref}

\usepackage{amsmath,amssymb,amsthm,amsfonts,graphicx,color}
\usepackage{amssymb}
\usepackage{pdfsync}
\usepackage{epstopdf}
\usepackage{upgreek}

\newcommand\abs[1]{\lvert#1\rvert} %[1]应该是newcommand 的可选参数，填进去啥，就是后面作用的东西

\newcommand{\ass}{\quad\mbox{as}\quad}

\newcommand{\diff}{\,\mathrm{d}}

\newcommand{\inn}{{\quad\hbox{in } }}

\newcommand{\nn}{ {\nabla}  }
\newcommand{\pp}{ {\partial} }

\newcommand{\RR}{{{\mathbb R}}}

\newcommand{\R} {\mathbb R}
\newcommand{\Z} {\mathbb Z}
\newcommand{\cuad}{{\sqcap\kern-.68em\sqcup}}

\newcommand{\foral}{\quad\mbox{for all}\quad}

\newcommand{\be}{\begin{equation}}
\newcommand{\ee}{\end{equation}}

\newcommand{\la}{\lambda}

\newcommand{\equ}[1]{(\ref{#1})}

\newtheorem{lemma}{Lemma}[section]
\newtheorem{prop}{Proposition}[section]
\newtheorem{theorem}{Theorem}

\newtheorem{remark}{Remark}[section]
\newcommand{\bremark}{\begin{remark} \em}
	\newcommand{\eremark}{\end{remark} }

\numberwithin{equation}{section}

\begin{document}
\title[Blow-up for critical heat equation]{Type II Finite time blow-up for the three dimensional energy critical heat equation}

\author[M. del Pino]{Manuel del Pino}
\address{\noindent
Department of Mathematical Sciences,
University of Bath, Bath BA2 7AY, United Kingdom}
\email{m.delpino@bath.ac.uk}

\author[M. Musso]{Monica Musso}
\address{\noindent
Department of Mathematical Sciences,
University of Bath, Bath BA2 7AY, United Kingdom}
\email{m.musso@bath.ac.uk}

\author[J. Wei]{Juncheng Wei}
\address{\noindent
Department of Mathematics,
University of British Columbia, Vancouver, B.C., V6T 1Z2, Canada}
\email{jcwei@math.ubc.ca}

\author[Q. Zhang]{Qidi Zhang}
\address{\noindent
Department of Mathematics,
University of British Columbia, Vancouver, B.C., V6T 1Z2, Canada}
\email{qidi@math.ubc.ca}

\author[Y. Zhou]{Yifu Zhou}
\address{\noindent
Department of Mathematics,
University of British Columbia, Vancouver, B.C., V6T 1Z2, Canada}
\email{yfzhou@math.ubc.ca}

\begin{abstract}
We consider the following Cauchy problem for three dimensional energy critical heat equation
\begin{equation*}
\begin{cases}
u_t=\Delta u+u^{5},~&\mbox{ in } \R^3 \times (0,T),\\
u(x,0)=u_0(x),~&\mbox{ in } \R^3.
\end{cases}
\end{equation*}
 We construct type II finite time blow-up solution $u(x,t)$ with the blow-up rates $ \| u\|_{L^\infty} \sim (T-t)^{-k}$, where $ k=1,2,... $. This gives a rigorous proof of the  formal computations by Filippas, Herrero and Vel´azquez \cite{fhv}. This is the first instance of type II finite time blow-up for three dimensional energy critical heat equation.

\end{abstract}
\maketitle

%{
%\hypersetup{linkcolor=black}
%\tableofcontents
%}

\section{Introduction}
The study of blow-up phenomena for the Fujita type nonlinear heat equation
\be\left\{
\begin{aligned}
u_t   & = \Delta u+ |u|^{p-1}u   \inn\ \R^n \times (0, T),  \\
u(\cdot,0) &  =u_0 \inn \R^n
\end{aligned}\right.
\label{F}\ee
is a classical topic with important applications in mathematical modelling and geometry.  Here $p>1$ and $ n\geq 1$.  A smooth solution of \equ{F} {\em blows-up at time $T$} if
$$ \lim_{t\to T} \| u(\cdot, t)\|_{L^\infty (\R^n) }    = + \infty . $$
There are two types of blow-ups:  the blow-up of a solution $u(x,t)$ is of type I if it happens at most at the ODE rate:
$$
\limsup_{t\to T} (T-t)^{\frac 1{p-1}} \| u(\cdot, t)\|_{L^\infty (\R^n) }    <  + \infty
$$
while the blow-up is said of type II if
$$
\limsup_{t\to T} (T-t)^{\frac 1{p-1}} \| u(\cdot, t)\|_{L^\infty (\R^n) }    =  + \infty.
$$

 It is known after a series of works, including \cite{gk1,gk2,gk3}, that type I is the only way possible if $p<p_S$ where $p_S$ is the critical Sobolev exponent,
$$p_S :=\begin{cases} \frac{n+2}{n-2}  & \hbox{ if } n\ge 3 \\ +\infty  &\hbox{ if } n=1,2.    \end{cases}
$$
Stability and genericity of type I blow-up have been considered for instance in \cite{collot2,mz,mm2, PY1, PY2}.
Solutions with type II blow-up are in fact much harder to detect. The first example was discovered in \cite{hv,hv1}, for $p>p_{JL}$ where    $p_{JL}$ is the Joseph-Lundgreen exponent,
$$p_{JL}=\begin{cases} 1+{4\over n-4-2\,\sqrt{n-1}} & \text{if $n\ge11$}\\ + \infty, & \text{if $ n\le 10$}.\\  \end{cases} $$
See the book \cite{qs} for a survey of related results.
In fact, no type II blow-up is present for radial solutions if $p_S< p< p_{JL}$, while for radial positive solutions this is not possible
if $p= \frac{n+2}{n-2}$ \cite{fhv}. Examples of nonradial positive blow-up solutions for $p> p_{JL}$  have been found in
\cite{collot4,collot1}. Various different scenarios have been discovered or discarded in the supercritical case. See for instance \cite{collot4,collot1,dmw2,hv,hv1,mm1, mm2, MS} and the book \cite{qs}.

In \cite{fhv} Filippas, Herrero and Velazquez formally obtained sign-changing solution with type II blow-up for $p=p_S$ in \cite{fhv} in lower dimensions $ n=3, 4, 5, 6$. The blowup of radial  blowup solutions found in \cite{fhv}  are given by
\begin{equation}
\label{FHVformal}
\| u \|_{L^\infty (\R^n)} \sim \left\{\begin{array}{l}
(T-t)^{-k},  \  n=3,\\
(T-t)^{-k} | \log (T-t)|^{\frac{2k}{k-1}},  \ n=4,\\
(T-t)^{-k},  \  n=5,\\
(T-t)^{-\frac{1}{4}} |\log (T-t)|^{-\frac{15}{8}}, \  n=6
\end{array}
\right.
\end{equation}
and this is corrected  to 
\begin{equation}
\label{Haformal}
\| u \|_{L^\infty (\R^n)} \sim \left\{\begin{array}{l}
(T-t)^{-k},  \  n=3,\\
(T-t)^{-k} | \log (T-t)|^{\frac{2k}{k-1}},  \ n=4,\\
(T-t)^{-3k},  \  n=5,\\
(T-t)^{-\frac{5}{2}} |\log (T-t)|^{-\frac{15}{4}}, \  n=6
\end{array}
\right.
\end{equation}
by Harada \cite{Ha1}. 

\medskip

  The first  rigorous proof of  radial example  was constructed by Schweyer in  \cite{schweyer} for $n=4$ and $ k=1$.  (See \cite{dmwz} for multiple blow-ups and nonradial case in $n=4, k=1$.) There is a deep connection between dimension four energy critical and two-dimensional harmonic map flows (\cite{ddw}\cite{RS1}).
Slower blowup rates for harmonic map flows are established by Raphael and Schweyer \cite{RS2}, which suggests that slower blowup rates are also possible for energy critical heat equation in the case $ n=4, k\geq 2$.

\medskip

In  dimension $n=5$ and  critical case      $p=p_S = \frac 73 $, the first three authors gave a rigorous construction of type II blow-up in radial and nonradial cases (confirming the case of $ k=1, n=5$).  Higher speed blow-up solutions for $n=5, p= p_S=\frac 73 , k\geq 2$ are carried out recently by Harada \cite{Ha1}. 

\medskip

Very recently Harada succeeded in establishing the type II blow-up in dimension $n=6, p=p_S=2$ \cite{Ha2}.

\medskip

  The purpose of this paper is to fill the gap in the last remaining dimension $n=3, p=p_S=5$.

\begin{theorem} \label{teo1} Let $n=3$ and $ k\in \Z_+$.  For each $T>0$ sufficiently small there exists
an initial condition $u_0$ such that the solution of Problem \equ{F} blows up at time $T$ which looks like at main order
\begin{equation*}
\begin{aligned}
u(x,t)=&~\eta\left(\frac{x}{\sqrt{T-t}}\right)\left[\mu^{-\frac{1}{2}}(t)w \left(\frac{x}{\mu(t)}\right)
+ 2\mu'(t)\,\mu^{\frac 1 2}(t) J\left(\frac{x}{\mu(t)}\right)\right] \\
&~+ \left(1-\eta\left(\frac{x}{\sqrt{T-t}}\right)\right)\eta\left(\frac{x}{\sqrt{T-t}}\right) (T-t)^k\frac{1}{|x|}C_k H_{2k}\left(\frac{|x|}{2\sqrt{T-t}}\right) + \theta(x,t)
\end{aligned}
\end{equation*}
where
\begin{equation*}
w(y) = 3^{\frac 14}(1+|y|^2)^{-\frac{1}{2}}, \quad C_k=\frac{(-1)^kk!\sqrt{3}}{(2k)!},
\end{equation*}
$H_{2k}$ is the Hermite polynomial defined in \eqref{def-Hermite} and $J$ is defined in \eqref{def-J}.
%\[
%J(y) = Z_0(y) \int_0^{y} \frac{1}{Z_0^2(s)} s^{-2}
%\int_0^s Z_0(u) u^2 \frac{Z_0(u)}{2} \diff u \diff s.
%\]
Moreover, the blow-up rate $\mu(t)$ satisfies
$$
\mu(t)  \sim \mu_0(t)= 3^{\frac 1 2} A (T-t)^{2k}, \quad k\in\Z_+
$$
and
 $\|\theta \|_{L^\infty} \le  T^a $ for some $a>0$.
\end{theorem}

 \medskip

\medskip
The method of this paper is close in spirit to the analysis in the works
\cite{cdm,ddw,dmw,dmw1,dmw2}, where the {\em inner-outer gluing method} is employed. This method reduces the original problem to solving a basically uncoupled system, which depends in subtle ways on the parameter choices (which are governed by relatively simple ODE systems). The main obstacle in proving finite time blow-up in dimension three is the very slow decaying behavior of the kernel $Z_0 (y) \sim \frac{1}{|y|}$ which is not even in $L^2$. To overcome this difficulty we  follow the ideas in \cite{dmw} in which the first three authors constructed {\em infinite} time blow-up for \equ{F} with fast decaying initial condition. First we use the  matched asymptotics expansion of \cite{fhv}
 to construct a good inner and outer expansions. But this is not good enough as the error still carries slow decaying. As in \cite{ddw} and \cite{dmw}, we use a {\em global} term to correct the slow decaying error. This global term carries all the information needed to solve the scaling parameter. We then use the {\em inner-outer} gluing procedure to find a true solution. Interestingly the remainder of the scaling parameter solves a nonlocal ODE of the following form
 \begin{equation*}
 \int_t^T \frac{\alpha (s)}{ \sqrt{T-s}} ds = h (t), t<T, \alpha (T)=0
 \end{equation*}
 which is Caputo derivative of $\frac{1}{2}$ for $ \beta=\int_0^t \alpha (t)$. This is the {\em nonlocal feature } for three dimensional problem. As far as we know this seems to be the first construction of finite time blow-up for critical heat equation in $\R^3$.

\medskip

We remark that the phenomena of blow-up by bubbling (time dependent, energy invariant, asymptotically singular scalings of steady states) arises in various problems of parabolic, geometric and dispersive nature. This has been an intensively studied topic in the harmonic map flow \cite{ddw, LW},  critical heat equations \cite{fhv, cmr, dmw1, mm1, mm2, qs}, critical Schr\"{o}dinger maps \cite{MR, MRR}, critical wave maps and Yang-Mills \cite{KS1,RR}, and critical wave equations \cite{DK,DKM2, kenigmerle, KNS,KS2, KST}.

\medskip
The rest of this paper will be devoted to the proof of Theorem \ref{teo1}.

%%%%%%%%%%%%%%%%%%%%%%%%%%%%%%%%%%%%%%%%%%%%%%%%%%%%%%%%%%

\medskip
	
\section{First approximation and matching}

\medskip

In this section, following the matched asymptotic expansions first developed in \cite{fhv}, we obtain an initial blow-up rate of finite blow-up solutions of the following nonlinear heat equation with critical exponent in $\RR^3$,
\begin{equation}\label{eq-u}
u_t=\Delta u+ u^5, \ \ u(x,0)=u_0(x),\ \ x\in\RR^3,\ \ t>0.
\end{equation}
where the initial value $u_0$ will be determined later. Throughout the paper, we shall use the symbol  $``\,\lesssim\,"$ to denote $``\,\leq\, C\,"$ for a positive constant $C$ independent of $t$ and $T$, where $C$ might be different from line to line.

\subsection{Approximate solutions}

\medskip

In the self-similar variable
\begin{equation*}
z=\frac{x}{\sqrt{T-t}},\quad \tau=-\log(T-t), \quad
 \Phi(z,\tau)=(T-t)^{\frac1{4}}u(x,t),
\end{equation*}
problem \eqref{eq-u} reads as
\begin{equation*}
\Phi_{\tau}=\Delta\Phi-\frac1{2} (z\cdot\nn\Phi)-\frac{\Phi}{4}+\Phi^5.
\end{equation*}
We next find good approximate solution for the above equation  in both inner and outer regimes.

\medskip

For the outer part, the first profile will be chosen as the solution to the linear problem
\begin{equation}\label{phiout}
\partial_{\tau}\Phi_{out} = \Delta\Phi_{out} - \frac1{2} (z\cdot\nn\Phi_{out})-\frac{\Phi_{out}}{4}.
\end{equation}
Writing $\Phi_{out}=e^{\gamma\tau}m(z)$, we look for radially symmetric solutions to the following ODE
\[
m''+(\frac{2}{z}-\frac{1}{2}z)m'-(\gamma + \frac{1}{4})m=0,
\]
which turns out to be an eigenvalue problem (\cite{fhv}). In order to get even solutions with polynomial growth rate in the outer regime, we take $\gamma=\frac{1}{4}-k,\ k\in \Z_+$,
and then
$$m(z)=A^{\frac{1}{2}}C_k\frac{1}{z}H_{2k}(\frac{z}{2}),$$
where $H_{2k}$ is Hermite polynomial
\begin{equation}\label{def-Hermite}
H_{2k}(\frac{z}{2})=\frac{(-1)^k(2k)!}{k!}(1-k\frac{z^2}{2}+\dots+a_kz^{2k}),\ C_k=\frac{(-1)^kk!\sqrt{3}}{(2k)!},
\end{equation}
$a_k$ is constant depending on $k$, and $A$ is any positive constant. Therefore, we get special solutions of $\equ{phiout}$
\[
\Phi_{out}(z,\tau)=A^{\frac{1}{2}}e^{(\frac{1}{4}-k)\tau}C_k\frac{1}{z}H_{2k}(\frac{z}{2}),
\]
which are sufficient for the matching to be carried out below, and we take the outer approximate solution of $\equ{eq-u}$ to be
\begin{equation}
\label{uout}
u_{out}(x,t)=(T-t)^{-\frac{1}{4}}\Phi_{out}(z,\tau)=A^{\frac{1}{2}}(T-t)^k\frac{1}{x}C_kH_{2k}(\frac{1}{2}\frac{x}{\sqrt{T-t}}).
\end{equation}
It can be directly  checked that in the original variable $ u_{out}$ satisfies
\[
\pp_t u_{out} = \Delta u_{out}.
\]

\medskip

For the inner part, we choose the approximate solution to be
\[
\Phi_{in}(z,\tau) = \epsilon^{-1/2} \big( w (y) + \sigma(\tau) J(y) \big),
\]
where
\[
w(y) = 3^{\frac 14}(1+|y|^2)^{-\frac{1}{2}},\
y = \frac{z}{\epsilon},\
\sigma(\tau) = 2\epsilon(\tau) \pp_{\tau} \epsilon(\tau) - \epsilon^2(\tau) ,
\]
and $\epsilon(\tau)$ is a positive function to be determined later. Here $J$ is the radial solution of
\begin{equation}\label{eq-J}
\Delta J +
5 w ^4 J + \frac 12 y \cdot\nn w + \frac{w}{4} = 0,
\ J(0)=0,\ J'(0)=0.
\end{equation}
Set $\mu(t) : =\sqrt{T-t} \epsilon(\tau)$ and $y=\frac{x}{\mu}$. Then one has
\begin{equation*}
\bar{u}_{in}(x,t) = (T-t)^{-\frac 14} \Phi_{in}(z,\tau)
= \mu^{-\frac{1}{2}}\big[w (\frac{x}{\mu}) + 2\mu \mu' J(\frac{x}{\mu})\big]
\end{equation*}
since
\[
\begin{split}
\sigma &= 2\epsilon \pp_{\tau} \epsilon-\epsilon^2
= 2 \epsilon \pp_t \epsilon (T-t)-\epsilon^2 \\
&= 2\mu (T-t)^{-\frac 12} \big[ \mu' (T-t)^{-\frac 12}
+ \mu \frac 1 2 (T-t)^{-\frac 32} \big](T-t) - \mu^2 (T-t)^{-1} \\
&= 2 \mu \mu' .
\end{split}
\]
Denote
\begin{equation}
\label{Z0}
Z_0(y) : = -\big[ y \cdot \nn w + \frac w 2 \big]
= \frac{3^{\frac 1 4}}{2} \frac{\abs{y}^2-1}{(1+\abs{y}^2)^{\frac 3 2}},
\end{equation}
which is a radial kernel of the homogeneous part of $\equ{eq-J}$. Then it is easy to see that

\begin{equation}\label{def-J}
J(y) = Z_0(y) \int_0^{y} \frac{1}{Z_0^2(s)} s^{-2}
 \int_0^s Z_0(u) u^2 \frac{Z_0(u)}{2} \diff u \diff s ,
\end{equation}
and thus
\[
J(y) \sim \frac{3^{\frac 1 4}}{8} y \ass y \to \infty .
\]

\medskip

\subsection{Matching inner and outer solutions}

In the region $\mu(t) \ll |x| \ll \sqrt{T-t}$, the inner and outer solutions have the following asymptotic behaviors respectively
\begin{equation}\label{matching-in}
\bar{u}_{in} \sim 3^{\frac 1 4} \mu^{\frac 1 2} |x|^{-1}
+\frac{3^{\frac 1 4}}{4} \mu^{-\frac 1 2} \mu' \, |x| ,
\end{equation}
\begin{equation}\label{matching-out}
u_{out}\sim A^{\frac{1}{2}}\sqrt{3}\frac{(T-t)^k}{|x|}-A^{\frac{1}{2}}\sqrt{3}\frac{k}{2}|x|(T-t)^{k-1}.
\end{equation}
Matching the inner and outer solutions, we get
\begin{equation*}
\mu \sim 3^{\frac 1 2} A (T-t)^{2k}, \qquad k \in \Z_+
\end{equation*}
so it is natural to choose
\[\mu_0 :=3^{\frac 1 2} A (T-t)^{2k} \]
as the leading order of the scaling parameter $\mu(t)$.

\medskip

We next choose
\begin{equation}\label{uin}
      u_{in}(x,t)
	= \mu^{-\frac{1}{2}}w (\frac{x}{\mu})
	 + 2\mu_0'\,\mu^{\frac 1 2} J(\frac{x}{\mu})
\end{equation}
and take the first approximate solution as follows
\begin{equation}
\label{U1}
U_1(x,t) := \eta\big(\frac{|x|}{r\sqrt{T-t}}\big) u_{in}
+ \big(1-\eta(\frac{|x|}{r_1 (T-t)^{\zeta_1}   })\big)
\eta(\frac{|x|}{r_2 (T-t)^{\zeta_2}  })u_{out},
\end{equation}
where $\eta$ is the smooth cut-off function satisfying $\eta(t)=1$ for $ t\in [0,1]$ and $\eta(t) = 0$ for $ t\in [2,\infty)$, and $r,~r_1,~r_2(>0)$ are sufficiently small constants to be determined later. For simplicity, we write
\[
\eta_1(x,t) := \eta\big(\frac{|x|}{r\sqrt{T-t}}\big),
\]

\[
\eta_{o1} (x,t) = \eta(\frac{|x|}{r_1 (T-t)^{\zeta_1}}),
\]

\[
\eta_{o2} (x,t) = \eta(\frac{|x|}{r_2 (T-t)^{\zeta_2}}).
\]

\medskip

\subsection{Error of the  first approximation}

We define the error function
$$
S(u):= - \pp_t u + \Delta_x u + u^5,
$$
and compute
\[
\begin{split}
S(U_1)
=& -\pp_t U_1 + \Delta_x U_1 + U_1^5 \\
=& -\pp_t u_{in} \eta_1 -u_{in}\pp_t \eta_1
+\Delta_x u_{in} \eta_1 + 2 \nn_x u_{in}\cdot \nn_x \eta_1 + u_{in}\Delta_x \eta_1 \\
& -\pp_t u_{out} (1-\eta_{o1} )\eta_{o2}
- u_{out} \pp_t ((1-\eta_{o1})\eta_{o2} ) \\
& +(1-\eta_{o1} ) \eta_{o2} \Delta_x u_{out}
+ 2 \nn_x ((1-\eta_{o1})\eta_{o2} ) \cdot \nn_x u_{out}
+u_{out}\Delta_x ((1-\eta_{o1})\eta_{o2}) \\
& + \big[u_{in}\eta_1 + (1 - \eta_{o1})\eta_{o2} u_{out}\big]^5 \\
=& \ \eta_1 ( - \pp_t u_{in} + \Delta_x u_{in} + u_{in}^5)
+ (1-\eta_{o1} )\eta_{o2}
( - \pp_t u_{out} + \Delta_x u_{out} + u_{out}^5) \\
& - \pp_t \eta_1 u_{in}
+ \Delta_x \eta_1 u_{in}
+ 2 \nn_x \eta_1 \nn_x u_{in}  \\
& -\pp_t [(1-\eta_{o1} )\eta_{o2}] u_{out}
+ \Delta_x [(1-\eta_{o1} )\eta_{o2}] u_{out}
+ 2 \nabla_x [(1-\eta_{o1} )\eta_{o2}] \nabla_x u_{out}
\\
& + \big[ \eta_1 u_{in} + (1-\eta_{o1} )\eta_{o2} u_{out}\big]^5
- \eta_1 u_{in}^5 - (1-\eta_{o1} )\eta_{o2} u_{out}^5 .
\end{split}
\]

Let
$$S_{in} := - \pp_t u_{in} + \Delta_x u_{in} + u_{in}^5,$$
$$S_{out} := - \pp_t u_{out} + \Delta_x u_{out} + u_{out}^5,$$ where we compute
\[
\begin{split}
\pp_t u_{in} =& -\mu^{-\frac 3 2} \mu' \big(\frac{w}{2}+\nn_y
w \cdot y\big)
+ 2 \mu_0'' \mu^{\frac 1 2} J(\frac{x}{\mu}) \\
&+ \mu_0' \mu^{-\frac 1 2} \mu' J(\frac{x}{\mu})
 -2 \mu_0' \mu^{-\frac 1 2} \mu' \nn_y J(\frac{x}{\mu})\cdot \frac{x}{\mu},
\end{split}
\]

\[
\pp_{\mu} \big( \mu^{\frac 1 2} J(\frac x{\mu}) \big)
= \mu^{-\frac 1 2} \big[ \frac 1 2 J(\frac x{\mu})
- \nn_y J(\frac x{\mu}) \cdot \frac x{\mu}
\big].
\]
Therefore, we have
\[
\begin{split}
S_{in}
=\ & \mu^{-\frac 3 2} \mu' \big(\frac{w}{2}
 + \nn_y w \cdot \frac{x}{\mu} \big)
 - 2 \mu_0'' \mu^{\frac 1 2} J(\frac{x}{\mu}) \\
&- \mu_0' \mu^{-\frac 1 2} \mu' J(\frac{x}{\mu})
 +2\mu_0' \mu^{-\frac 1 2} \mu' \nn_y J(\frac{x}{\mu})\cdot \frac{x}{\mu} \\
&+ \mu^{-\frac 5 2} \Delta_y w
 + 2 \mu_0' \mu^{-\frac 3 2} \Delta_y J(\frac{x}{\mu}) \\
&+\big[\mu^{-\frac 1 2}w(\frac{x}{\mu})
 + 2\mu_0'\,\mu^{\frac 1 2} J(\frac{x}{\mu})]^5 \\
=\ & \mu^{-\frac 3 2} (\mu-\mu_0)' \big(\frac{w}{2}
 + \nn_y w \cdot \frac{x}{\mu}\big)
 - 2 \mu_0'' \mu^{\frac 1 2} J(\frac{x}{\mu}) \\
&- \mu_0' \mu^{-\frac 1 2} \mu' J(\frac{x}{\mu})
 +2\mu_0' \mu^{-\frac 1 2} \mu' \nn_y J(\frac{x}{\mu})\cdot \frac{x}{\mu} \\
&+\big[\mu^{-\frac 1 2}w(\frac{x}{\mu})
 + 2\mu_0'\,\mu^{\frac 1 2} J(\frac{x}{\mu})]^5
 - \big( \mu^{-\frac 1 2} w(\frac{x}{\mu}) \big)^5
 -5 \big( \mu^{-\frac 1 2} w(\frac{x}{\mu}) \big)^4 2\mu_0'\,\mu^{\frac 1 2} J(\frac{x}{\mu}),
\end{split}
\]
where the first term can be written as
\[
\begin{split}
  \mu^{-\frac 3 2} (\mu-\mu_0)' \big(\frac{w}{2}
 + \nn_y w \cdot \frac{x}{\mu}\big) = \mu^{-\frac 3 2} (\mu-\mu_0)'
  \frac{3^{\frac 1 4}}{2} \frac{1 - \abs{y}^2}{(1+\abs{y}^2)^{\frac 3 2}}.
\end{split}
\]
We notice that
\[
(\mu' - \mu_0')\mu^{- \frac 1 2}
= 2(\mu^{\frac 1 2} - \mu_0^{\frac 1 2})'
 + (\mu^{\frac 1 2} - \mu_0^{\frac 1 2}) \mu_0^{-1} \mu_0'
 - (\mu^{\frac 1 2} - \mu_0^{\frac 1 2})^2 \mu^{- \frac 1 2}
  \mu_0^{-1} \mu_0',
\]
and
\[
\begin{split}
& \frac{3^{\frac 1 4}}{2} \frac{1 - \abs{y}^2}{(1+\abs{y}^2)^{\frac 3 2}}
= \frac{3^{\frac 1 4}}{2} \mu \frac{\mu^2 - |x|^2}{( \mu^2 + |x|^2)^{\frac 3 2}}
= 3^{\frac 1 4}
\frac{ \mu^3}{(\mu^2 + |x|^2)^{\frac 3 2}}
- \frac{ 3^{\frac 1 4}}{2} \frac {\mu}{(\mu^2 + |x|^2 )^{\frac 1 2}}.
\end{split}
\]
So the leading term is
\[
\begin{split}
& \eta_1 \big[ 2(\mu^{\frac 1 2} - \mu_0^{\frac 1 2})'
+ (\mu^{\frac 1 2} - \mu_0^{\frac 1 2}) \mu_0^{-1} \mu_0' \big]
\mu^{-1} ( - \frac{3^{\frac 1 4}}{2} )
\frac{\mu}{ (\mu^2 + |x|^2 )^{\frac 1 2} } \\
=\ & \eta_1 \big[ 2(\mu_0^{\frac 1 2} \Lambda(t) )'
+ \mu_0^{\frac 1 2} \Lambda(t) \mu_0^{-1} \mu_0' \big]
( - \frac{3^{\frac 1 4}}{2} ) \frac 1 { (\mu^2 + |x|^2 )^{\frac 1 2} } \\
=\ & \eta_1 2 \mu_0^{-\frac 1 2} \big[ \mu_0^{\frac 1 2}(\mu_0^{\frac 1 2}    \Lambda(t))'
+ (\mu_0^{\frac 1 2 } \Lambda(t))(\mu_0^{\frac 1 2})'  \big]
( - \frac{3^{\frac 1 4}}{2} ) \frac 1 { (\mu^2 + |x|^2 )^{\frac 1 2} } \\
=\ & \eta_1 \frac {\alpha(t)} { (\mu^2 + |x|^2 )^{\frac 1 2} }.
\end{split}
\]
Here we define
\[
\begin{split}
\mu(t) :=\ & \mu_0(t) (1+\Lambda(t))^2,\\
\alpha(t) :=\ & (-3^{\frac 1 4}) \mu_0^{-\frac 1 2}(t) (\mu_0(t) \Lambda(t))'.
\end{split}
\]

For the last term in $S_{in}$, we compare the size of
$\mu^{-\frac 1 2}w(\frac{x}{\mu}) $ and
$ 2\mu_0'\,\mu^{\frac 1 2} J(\frac{x}{\mu})$ in the regime $\frac{|x|}{r\sqrt{T-t}} \le 2$ thanks to the cut-off $\eta_1$: if $\frac{|x|}{\mu} \ll 1$, we have
\[
\mu^{-\frac 1 2} \gg \mu_0' \mu^{\frac 1 2} \iff \mu^{-1} \gg  \mu_0'
\iff 3^{-\frac 1 2} A^{-1} (T-t)^{-2k} \gg 3^{\frac 1 2} A 2k (T-t)^{2k-1}.
\]
Thus we get $\mu^{-\frac 1 2}w(\frac{x}{\mu}) \gg 2\mu_0'\,\mu^{\frac 1 2} J(\frac{x}{\mu})$. If $\frac{|x|}{\mu} \gg 1$, we only need to check
\[
\mu^{-\frac 1 2} \frac{\mu}{|x|} \gg \mu^{\frac 1 2} \mu_0' \frac{|x|}{\mu}
\iff \mu(\mu_0')^{-1} \gg |x|^2
\iff (T-t) \gg |x|^2
\]
since $\frac{|x|}{r\sqrt{T-t}} \le 2$ and $r\ll 1$. So $\mu^{-\frac 1 2}w(\frac{x}{\mu}) \gg 2 \mu_0'\,\mu^{\frac 1 2} J(\frac{x}{\mu})$ is also satisfied. Therefore, the last term in $S_{in}$ can be expanded as
\[
\begin{split}
  & \big[\mu^{-\frac 1 2}w(\frac{x}{\mu})
    + 2\mu_0'\,\mu^{\frac 1 2} J(\frac{x}{\mu})]^5
    - \big( \mu^{-\frac 1 2} w(\frac{x}{\mu}) \big)^5
    -5 \big( \mu^{-\frac 1 2} w(\frac{x}{\mu}) \big)^4
     2\mu_0'\,\mu^{\frac 1 2} J(\frac{x}{\mu}) \\
=\  & 5 \big[\mu^{-\frac 1 2}w(\frac{x}{\mu})
    + \theta_1 2\mu_0'\,\mu^{\frac 1 2} J(\frac{x}{\mu})]^4
      2\mu_0'\,\mu^{\frac 1 2} J(\frac{x}{\mu})
    -5 \big( \mu^{-\frac 1 2} w(\frac{x}{\mu}) \big)^4
      2\mu_0'\,\mu^{\frac 1 2} J(\frac{x}{\mu}) \\
=\  & 20 \big[\mu^{-\frac 1 2}w(\frac{x}{\mu})
      + \theta_2 \theta_1 2\mu_0'\,\mu^{\frac 1 2} J(\frac{x}{\mu})]^3
       \theta_1 \big[2\mu_0'\,\mu^{\frac 1 2} J(\frac{x}{\mu})\big]^2 \\
\lesssim\
    & \big[\mu^{-\frac 1 2}w(\frac{x}{\mu}) \big]^3\
      \big[2\mu_0'\,\mu^{\frac 1 2} J(\frac{x}{\mu})\big]^2,
\end{split}
\]
where $\theta_1,\theta_2\in[0,1]$ and \begin{equation}\label{def-hhh}
|J(y)| \le h(y)
:=
\begin{cases}
y^2  & \text{if} \ \ y \rightarrow 0;  \\
y    & \text{if} \ \ y \rightarrow \infty.
\end{cases}
\end{equation}
We define that $\chi(x) = 1 $ if $\abs{x} \le 1$ and $\chi(x) = 0 $ otherwise. Therefore, we have the following estimate
\begin{equation}\label{est-g4main}
\begin{aligned}
& \eta_1 S_{in} - \chi(\frac{x}{c_0(T-t)^{\frac12}})
\frac {\alpha(t)} { (\mu^2 + |x|^2 )^{\frac 1 2}  } \\
=\ & \mu^{-\frac 3 2} (\mu-\mu_0)' \big(\frac{w}{2}
+ \nn_y w \cdot \frac{x}{\mu}\big) \eta_1
- \chi(\frac{x}{c_0(T-t)^{\frac 1 2}})
\frac {\alpha(t)} { (\mu^2 + |x|^2 )^{\frac 1 2}  } \\
& - 2 \mu_0'' \mu^{\frac 1 2} J(\frac{x}{\mu}) \eta_1
  - \mu_0' \mu^{-\frac 1 2} \mu' J(\frac{x}{\mu}) \eta_1
  +2\mu_0' \mu^{-\frac 1 2} \mu' \nn_y J(\frac{x}{\mu})\cdot \frac{x}{\mu} \eta_1 \\
&+\left[ \big[\mu^{-\frac 1 2}w(\frac{x}{\mu})
+ 2\mu_0'\,\mu^{\frac 1 2} J(\frac{x}{\mu})]^5
- \big( \mu^{-\frac 1 2} w(\frac{x}{\mu}) \big)^5
-5 \big( \mu^{-\frac 1 2} w(\frac{x}{\mu}) \big)^4 2\mu_0'\,\mu^{\frac 1 2} J(\frac{x}{\mu}) \right] \eta_1 \\
%=\ & (-\frac 2{3^{\frac 1 4} } \alpha(t) - \Lambda^2 \mu^{-\frac 1 2 } \mu_0')
%( 3^{\frac 1 4} \frac{ \mu^2 }{(\mu^2 + |x|^2)^{\frac 3 2}}
%- \frac { 3^{\frac 1 4}} {2} \frac 1 {(\mu^2 + |x|^2 )^{\frac 1 2}} ) \eta_1\\
%& - \chi(\frac{x}{c_0(T-t)^{\frac 1 2}})
%\frac {\alpha(t)} { (\mu^2 + |x|^2 )^{\frac 1 2}  } \\
%& - 2 \mu_0'' \mu^{\frac 1 2} J(\frac{x}{\mu}) \eta_1
%- \mu_0' \mu^{-\frac 1 2} \mu' J(\frac{x}{\mu}) \eta_1
%+2\mu_0' \mu^{-\frac 1 2} \mu' \nn_y J(\frac{x}{\mu})\cdot \frac{x}{\mu} \eta_1 \\
%& + \left[  \big[\mu^{-\frac 1 2}w(\frac{x}{\mu})
%+ 2\mu_0'\,\mu^{\frac 1 2} J(\frac{x}{\mu})]^5
%- \big( \mu^{-\frac 1 2} w(\frac{x}{\mu}) \big)^5
%-5 \big( \mu^{-\frac 1 2} w(\frac{x}{\mu}) \big)^4 2\mu_0'\,\mu^{\frac 1 2} J(\frac{x}{\mu}) \right] \eta_1  \\
=\ & -2 \alpha(t) \frac{ \mu^2 }{(\mu^2 + |x|^2)^{\frac 3 2}} \eta_1
+ \frac {\alpha(t) } {(\mu^2 + |x|^2 )^{\frac 1 2} }
(\eta_1 - \chi(\frac{x}{c_0(T-t)^{\frac 1 2} } ) )
\\
& - \Lambda^2 \mu^{-\frac 1 2 } \mu_0'
( 3^{\frac 1 4} \frac{ \mu^2 }{(\mu^2 + |x|^2)^{\frac 3 2}}
- \frac { 3^{\frac 1 4}} {2} \frac 1 {(\mu^2 + |x|^2 )^{\frac 1 2}} ) \eta_1
\\
& - 2 \mu_0'' \mu^{\frac 1 2} J(\frac{x}{\mu}) \eta_1
- \mu_0' \mu^{-\frac 1 2} \mu' J(\frac{x}{\mu}) \eta_1
+2\mu_0' \mu^{-\frac 1 2} \mu' \nn_y J(\frac{x}{\mu})\cdot \frac{x}{\mu} \eta_1 \\
& + \left[  \big[\mu^{-\frac 1 2}w(\frac{x}{\mu})
+ 2\mu_0'\,\mu^{\frac 1 2} J(\frac{x}{\mu})]^5
- \big( \mu^{-\frac 1 2} w(\frac{x}{\mu}) \big)^5
-5 \big( \mu^{-\frac 1 2} w(\frac{x}{\mu}) \big)^4 2\mu_0'\,\mu^{\frac 1 2} J(\frac{x}{\mu}) \right] \eta_1\\
\lesssim \
&|\alpha(t)| \frac{ \mu^2 }{(\mu^2 + |x|^2)^{\frac 3 2}} \eta_1
 + \frac { | \alpha(t) | } {(\mu^2 + |x|^2 )^{\frac 1 2} }
\chi (r\le \frac{ |x| }{\sqrt{T-t}} \le c_0)
 \\
 & + \Lambda^2 \mu^{-\frac 1 2 } \mu_0'
 \left| 3^{\frac 1 4} \frac{ \mu^2 }{(\mu^2 + |x|^2)^{\frac 3 2}}
 - \frac { 3^{\frac 1 4}} {2} \frac 1 {(\mu^2 + |x|^2 )^{\frac 1 2}} \right| \eta_1
 \\
 & + 2 \mu_0'' \mu^{\frac 1 2} h(\frac{x}{\mu}) \eta_1
 + \mu_0' \mu^{-\frac 1 2} \mu' h(\frac{x}{\mu}) \eta_1
 +2\mu_0' \mu^{-\frac 1 2} \mu' h(\frac{x}{\mu} ) \eta_1 \\
 & + \frac{\mu^{\frac 5 2}}{(\mu^2 + |x|^2)^{\frac 3 2}} (\mu_0')^2 h^2(\frac x{\mu_0}) \eta_1.
 \\
\end{aligned}
\end{equation}

\medskip
%%%%%%%%%%%

We decompose $u = U_1 + \Phi_1 + \Phi_2$ and compute
\[
\begin{split}
 & S(U_1 + \Phi_1 +\Phi_2) \\
=\ & -\pp_t U_1 -\pp_t \Phi_1 -\pp_t \Phi_2
     +\Delta_x U_1 + \Delta_x \Phi_1 + \Delta_x \Phi_2
     +U_1^5 + (U_1 + \Phi_1+\Phi_2)^5 - U_1^5 \\
=\ & S(U_1)  -\pp_t \Phi_1 -\pp_t \Phi_2 + \Delta_x \Phi_1 + \Delta_x \Phi_2
     + (U_1 + \Phi_1 + \Phi_2)^5 - U_1^5 \\
=\ & S(U_1) - \chi(\frac{x}{c_0(T-t)^{\frac12}})\frac {\alpha(t)} { (\mu^2 + |x|^2 )^{\frac 1 2} }
     +\chi(\frac{x}{c_0(T-t)^{\frac12}})
     \frac {\alpha(t)} { (\mu^2 + |x|^2 )^{\frac 1 2} }
     - \pp_t \Phi_1  \\
   & - \pp_t \Phi_2 + \Delta_x \Phi_1 + \Delta_x \Phi_2
     +(U_1 + \Phi_1 +\Phi_2 )^5 - U_1^5
\end{split}
\]
where $\chi(x) = 1 $ if $\abs{x} \le 1$ and $\chi(x) = 0 $ otherwise, $c_0$ is some constant, $\Phi_1$ and $\Phi_2$ are perturbations to be determined later.

\medskip

\subsection{Nonlocal correction: second approximation}

Observe that in the error there is a slow decaying term $ \frac {\alpha(t)} { (\mu^2(t) + \abs{x}^2)^{1/2} }$. Following the idea in \cite{dmw}, we now introduce a nonlocal correction $\Phi_1$ solving
\begin{equation}
	\pp_t \Phi_1 = \Delta_x \Phi_1
	     + \chi(\frac{x}{c_0(T-t)^{\frac 1 2}   }) \frac {\alpha(t)} {(\mu^2(t) + \abs{x}^2)^{1/2} }.
\end{equation}
Choosing the initial data, we get from Duhamel's formula
\begin{equation}
 \Phi_1(x,t) =
 \sum\limits_{j=1}^{k} c_j \mathcal B^{(j)}(x,t)
 + \int_0^t \int\limits_{\RR^3}
     (\frac 1{2\sqrt{\pi}})^3(t-s)^{-\frac 3 2}
     e^{-\frac{\abs{x-\xi}^2}{4(t-s)}}
     \chi(\frac{\xi}{c_0(T-s)^{\frac 1 2}  })\frac {\alpha(s)} {(\mu^2(s) + \abs{\xi}^2)^{1/2} } \diff \xi \diff s
\end{equation}
where $c_j$ are constants and $\mathcal B^{(j)}$ satisfy heat equation
\[
	\pp_t \mathcal B^{(j)} (x,t) = \Delta_x \mathcal B^{(j)} (x,t)
\]
which will be determined later when solving $\alpha(t)$ from the reduced equation in Section \ref{sec-redu}.

Then the new error becomes
\[
S(U_1 + \Phi_1) =
S(U_1) - \chi(\frac{x}{c_0(T-t)^{{\frac 1 2}} } )\frac {\alpha(t)} { (\mu^2 + |x|^2 )^{\frac 1 2}  }
+(U_1 + \Phi_1)^5 - U_1^5
\]

Let $ \Phi=\Phi_1+\Phi_2$. We have
\[
\begin{split}
& S(U_1 + \Phi_1 + \Phi_2) \\
=\ &
S(U_1) - \chi(\frac{x}{c_0(T-t)^{\frac 1 2} } )
\frac {\alpha(t)} { (\mu^2 + |x|^2 )^{\frac 1 2}  }
- \pp_t \Phi_2 + \Delta_x \Phi_2
+(U_1 + \Phi_1 +\Phi_2)^5 - U_1^5 \\
=\ & S(U_1) - \chi(\frac{x}{c_0(T-t)^{\frac 1 2} })
\frac {\alpha(t)} {(\mu^2 + |x|^2 )^{\frac 1 2}  }
+(U_1 + \Phi_1)^5 - U_1^5 \\
&- \pp_t \Phi_2 + \Delta_x \Phi_2
+(U_1 + \Phi_1 +\Phi_2)^5 - (U_1 + \Phi_1)^5 \\
=\ & S(U_1) - \chi(\frac{x}{c_0(T-t)^{\frac 1 2}  })
\frac {\alpha(t)} {(\mu^2 + |x|^2 )^{\frac 1 2}  }
+(U_1 + \Phi_1)^5 - U_1^5 \\
&- \pp_t \Phi_2 + \Delta_x \Phi_2 + 5(U_1 + \Phi_1)^4 \Phi_2
+(U_1 + \Phi_1 +\Phi_2)^5 - (U_1 + \Phi_1)^5 -5(U_1 + \Phi_1)^4 \Phi_2.
\end{split}
\]

\medskip

%%%%%%%%%%%%%%%%%%%%%%%%%%%%%%%%%%%%%%%%%%%%%%%%%%%%%%%%

\section{Inner--outer gluing system}

\medskip

In this section, we set up the inner--outer gluing scheme for the nonlinear problem. We look for perturbation of the following form
\[
\begin{split}
\Phi_2(x,t)
=\ & \psi(x,t) + \phi^{in}(x,t)
   = \psi(x,t) + \eta(\frac{x}{R\mu_0}) \hat{\phi}(x,t) \\
=\ & \psi(x,t) + \eta(\frac x {R\mu_0})
     \mu_0^{-\frac 1 2}\phi(\frac x {\mu_0},t) \\
=\ & \psi(x,t) + \eta_R \mu_0^{-\frac 1 2}\phi(\frac x {\mu_0},t)
\end{split}
\]
where $\hat{\phi}(x,t) := \mu_0^{-\frac 1 2}\phi(\frac x {\mu_0},t)$,
$\eta_R := \eta(\frac x {R\mu_0})$, and $R(t)=\mu_0^{-\beta}(t)$ with $\beta\in(0,1/2)$. We next compute
\[
\pp_t \Phi_2(x,t) = \pp_t \psi(x,t)
  + \pp_t \eta_R \mu_0^{-\frac 1 2}\phi(\frac x {\mu_0},t)
  + \eta_R \pp_t (\mu_0^{-\frac 1 2}\phi(\frac x {\mu_0},t))
\]
where
\[
\begin{split}
 \pp_t (\mu_0^{-\frac 1 2}\phi(\frac x {\mu_0},t))
=\ & -\frac 1 2 \mu_0^{-\frac 3 2} \mu_0' \phi(\frac x {\mu_0},t)
  -\mu_0^{-\frac 3 2} \nn_y \phi(\frac{x}{\mu_0}, t) \cdot \frac{x}{\mu_0} \mu_0'
  + \mu_0^{- \frac 1 2} \pp_t \phi(\frac{x}{\mu_0}, t) \\
=\ & -\mu_0^{-\frac 3 2}\mu_0'
     \big[ \frac 1 2 \phi(\frac{x}{\mu_0}, t) + \nn_y \phi(\frac{x}{\mu_0}, t) \cdot \frac{x}{\mu_0}  \big]
     + \mu_0^{- \frac 1 2} \pp_t \phi(\frac{x}{\mu_0}, t)
\end{split}
\]
and
\[
\Delta_x \Phi_2(x,t) = \Delta_x \psi(x,t)
   + \Delta_x \eta_R \mu_0^{-\frac 1 2}\phi(\frac x {\mu_0},t)
   + 2\nn_x \eta_R \mu_0^{-\frac 3 2} \nn_y \phi(\frac x{\mu_0},t)
   + \eta_R \mu_0^{-\frac 5 2} \Delta_y \phi(\frac x{\mu_0}, t).
\]
Therefore, we obtain
\[
\begin{split}
   & S(U_1 + \Phi_1 + \Phi_2) \\
=\ & S(U_1) - \chi(\frac{x}{c_0(T-t)^{\frac 1 2}})
\frac {\alpha(t)} { (\mu^2 + |x|^2 )^{\frac 1 2}  }
     + (U_1 + \Phi_1)^5 - U_1^5 \\ % =S(U_1 + $\Phi_1$)
   & - \pp_t \psi(x,t)
     - \pp_t \eta_R \mu_0^{-\frac 1 2}\phi(\frac x {\mu_0},t) \\
   & + \eta_R \mu_0^{-\frac 3 2}\mu_0'
       \big[ \frac 1 2 \phi(\frac{x}{\mu_0}, t)
     + \nn_y \phi(\frac{x}{\mu_0}, t) \cdot \frac{x}{\mu_0}  \big]
     - \eta_R \mu_0^{- \frac 1 2} \pp_t \phi(\frac{x}{\mu_0}, t) \\
   & + \Delta_x \psi(x,t)
     + \Delta_x \eta_R \mu_0^{-\frac 1 2}\phi(\frac x {\mu_0},t)
     + 2\nn_x \eta_R \cdot \mu_0^{-\frac 3 2} \nn_y \phi(\frac x{\mu_0},t)
     + \eta_R \mu_0^{-\frac 5 2} \Delta_y \phi(\frac x{\mu_0}, t) \\
   & + 5(U_1 + \Phi_1)^4
       \big(\psi(x,t) + \eta_R \mu_0^{-\frac 1 2}\phi(\frac x {\mu_0},t) \big) \\
   & + (U_1 + \Phi_1 +\Phi_2)^5 - (U_1 + \Phi_1)^5 -5(U_1 + \Phi_1)^4 \Phi_2.
\end{split}
\]
It can be directly  checked that $S(U_1 + \Phi_1 + \Phi_2)=0$ if $\psi$ solves
\begin{equation}\label{eqn-outer'''}
\begin{split}
& - \pp_t \psi(x,t)
+ \Delta_x \psi(x,t)  - \pp_t \eta_R \mu_0^{-\frac 1 2}\phi(\frac x {\mu_0},t)   + \Delta_x \eta_R \mu_0^{-\frac 1 2}\phi(\frac x {\mu_0},t)  + 2\nn_x \eta_R \cdot \mu_0^{-\frac 3 2} \nn_y \phi(\frac x{\mu_0},t) \\
& + 5(U_1 + \Phi_1)^4 (1-\eta_R)\psi(x,t)  + 5\big[ (U_1 + \Phi_1)^4
  - \big( \mu^{-\frac 1 2} w(\frac x{\mu}) \big)^4 \big]
    \psi(x,t) \eta_R \\
& + (U_1 + \Phi_1 +\Phi_2)^5 - (U_1 + \Phi_1)^5 -5(U_1 + \Phi_1)^4 \Phi_2 \\
& + \left[S(U_1) - \chi(\frac{x}{c_0(T-t)^{\frac12}})
\frac {\alpha(t)} { (\mu^2 + |x|^2 )^{\frac 1 2}  }\right](1-\eta_R)
  + \big[ (U_1 + \Phi_1)^5 - U_1^5 \big](1-\eta_R) = 0
\end{split}
\end{equation}
and $\phi$ solves
\begin{equation}\label{eqn-inner'''}
\begin{split}
& \eta_R \mu_0^{-\frac 3 2}\mu_0'
  \big[ \frac 1 2 \phi(\frac{x}{\mu_0}, t) + \nn_y \phi(\frac{x}{\mu_0}, t) \cdot \frac{x}{\mu_0}  \big]
  - \eta_R \mu_0^{- \frac 1 2} \pp_t \phi(\frac{x}{\mu_0}, t) \\
& + \eta_R \mu_0^{-\frac 5 2} \Delta_y \phi(\frac x{\mu_0}, t)  + 5\big( \mu^{-\frac 1 2} w(\frac x{\mu}) \big)^4 \psi(x,t) \eta_R + 5(U_1 + \Phi_1)^4 \mu_0^{-\frac 1 2} \phi(\frac x{\mu_0},t) \eta_R \\
& +\eta_R\left[S(U_1) - \chi(\frac{x}{c_0(T-t)^{\frac12}})
\frac {\alpha(t)} { (\mu^2 + |x|^2 )^{\frac 1 2}  }\right]+ \big[ (U_1 + \Phi_1)^5 - U_1^5 \big]\eta_R = 0.
\end{split}
\end{equation}
Recall that in the inner regime one has
\[
U_1 = u_{in} = \mu^{-\frac{1}{2}}w (\frac{x}{\mu})
+ 2\mu_0'\,\mu^{\frac 1 2} J(\frac{x}{\mu}).
\]
Writing $y = \frac x{\mu_0}$, we rearrange problems \eqref{eqn-outer'''} and \eqref{eqn-inner'''} and get
\begin{itemize}
\item The outer problem:
\begin{equation}\label{eqn-outer}
\begin{split}
\pp_t \psi(x,t)=\Delta_x \psi(x,t) +\mathcal G(\phi,\psi,\alpha)
\end{split}
\end{equation}
where
\begin{equation}\label{def-mathcalG}
\begin{aligned}
\mathcal G(\phi,\psi,\alpha)=& - \pp_t \eta_R \mu_0^{-\frac 1 2}\phi(\frac x {\mu_0},t)
  + \Delta_x \eta_R \mu_0^{-\frac 1 2}\phi(\frac x {\mu_0},t)
  + 2\nn_x \eta_R \cdot \mu_0^{-\frac 3 2} \nn_y \phi(\frac x{\mu_0},t) \\
& +\left(5(U_1 + \Phi_1)^4 (1-\eta_R) + 5\big[ (U_1 + \Phi_1)^4
        - \big( \mu^{-\frac 1 2} w(\frac x{\mu}) \big)^4 \big] \eta_R\right)\psi\\
& + (U_1 + \Phi_1 +\Phi_2)^5 - (U_1 + \Phi_1)^5 -5(U_1 + \Phi_1)^4 \Phi_2 \\
& + \left[S(U_1) - \chi(\frac{x}{c_0(T-t)^{\frac 1 2}  })
\frac {\alpha(t)} { (\mu^2 + |x|^2 )^{\frac 1 2}  }\right](1-\eta_R)
+ \big[ (U_1 + \Phi_1)^5 - U_1^5 \big](1-\eta_R)
\end{aligned}
\end{equation}

\item The inner problem:
\begin{equation}\label{eqn-inner}
\begin{split}
\mu_0^2 \pp_t \phi(y, t)
  =\Delta_y \phi(y, t)
  + 5 w^4(y) \phi(y,t) +\mathcal H(\phi,\psi,\alpha) \quad\mbox{ in }~B_{2R}\times (0,T),
\end{split}
\end{equation}
where
\begin{equation}\label{def-mathcalH}
\begin{aligned}
\mathcal H(\phi,\psi,\alpha)=
&~ 5\big[ (u_{in}(\mu_0 y, t) + \Phi_1(\mu_0 y, t))^4
  - \mu_0^{-2}w^4(y)\big] \mu_0^2 \phi(y,t) \\
&~ + 5 \mu_0^{\frac 1 2}(1 + \Lambda)^{-4} w^4(\frac y{(1+\Lambda)^2})
    \psi(\mu_0 y,t)
  + \mu_0 \mu_0'
    \big[ \frac 1 2 \phi(y, t) + \nn_y \phi(y,t) \cdot y \big]\\
&~ + \mu_0^{\frac52}\left[S(U_1) - \chi(\frac{\mu_0 y}{c_0(T-t)^{\frac12}})
\frac {\alpha(t)} { (\mu^2 + |\mu_0 y|^2 )^{\frac 1 2}  }\right]\\
&~+ \mu_0^{\frac 5 2}\big[ (u_{in}(\mu_0 y, t) + \Phi_1(\mu_0 y, t))^5
  - u_{in}^5(\mu_0 y, t) \big]
\end{aligned}
\end{equation}
\end{itemize}
We choose $\zeta_1=\zeta_2=\frac12$ so that $ \big(1-\eta(\frac{|x|}{r_1 (T-t)^{\zeta_1}   })\big)
\eta(\frac{|x|}{r_2 (T-t)^{\zeta_2}  })u_{out}\equiv 0$ in the inner region
$\{ x: |x|\le 2R\mu_0 \}$
\[
r_1 (T-t)^{\zeta_1}  \gg 2R\mu_0
\]
for $T\ll 1$ since $R(t)=\mu_0^{-\beta}$ with $\beta\in(0,1/2)$.

Performing the change of variables
\[
\frac{\diff \tau}{\diff t} = \mu_0^{-2},\quad y=\frac{x}{\mu_0}
%t = T - [3A^2(4k-1)]^{-\frac 1{4k-1}} \tau^{-\frac 1{4k-1}}
\]
so that
$$
\tau = \frac 1{3A^2 (4k-1)}(T-t)^{1-4k},\quad \mu_0 = (3^{\frac 1 2} A)^{\frac {-1}{4k-1} } (4k-1)^{-\frac{2k}{4k-1} }
        \tau^{-\frac{2k}{4k-1} }
$$
and writing $\phi(y, \tau) = \phi(y, t(\tau))$, we obtain
\begin{equation*}
\begin{split}
& - \pp_{\tau} \phi(y, \tau)
  + \Delta_y \phi(y, \tau)
  + 5 w^4(y) \phi(y,\tau) \\
& + 5\big[ (u_{in}(\mu_0 y, t(\tau) ) + \Phi_1(\mu_0 y, t(\tau)) )^4
  - \mu_0^{-2}w^4(y)\big] \mu_0^2 \phi(y,\tau) \\
& + 5 \mu_0^{\frac 1 2}(1 + \Lambda)^{-4} w^4(\frac y{(1+\Lambda)^2})
    \psi(\mu_0 y,t(\tau) )
  + \mu_0 \mu_0'
    \big[ \frac 1 2 \phi(y, \tau) + \nn_y \phi(y,\tau) \cdot y \big]\\
&+ \mu_0^{\frac52}\left[S(U_1) - \chi(\frac{\mu_0 y}{c_0[3A^2(4k-1)\tau]^{\frac{1}{2(1-4k)}}})
\frac {\alpha(t(\tau))} { (\mu^2 + |\mu_0 y|^2 )^{\frac 1 2}  }\right]\\
& + \mu_0^{\frac 5 2}\big[ (u_{in}(\mu_0 y, t(\tau))
  + \Phi_1(\mu_0 y, t(\tau)))^5
  - u_{in}^5(\mu_0 y, t(\tau)) \big]  = 0  \quad\mbox{ in }~B_{2R}\times (\tau_0,+\infty).
\end{split}
\end{equation*}

In the sequel, we shall solve the inner--outer gluing system \eqref{eqn-outer} and \eqref{eqn-inner} by developing suitable linear theories and the fixed point argument.

\medskip

%%%%%%%%%%%%%%%%%%%%%%%%%%%%%%%%%%%%%%%%%%%%%%%%%%%%%%

\section{Linear theories}\label{sec-lt}

\medskip

\subsection{Linear theory for the inner problem}\label{sec-inner}

\medskip

By using the Fourier decomposition and delicate analysis, the authors in \cite{dmw} developed the linear theory for the inner problem \eqref{eqn-inner} in dimension $3$. Denote $Z_-$ by the positive radial bounded eigenfunction associated to the only negative eigenvalue $\la_-$ to
$$
\Delta \phi + 5w^4 \phi + \la_- \phi=0,\quad \phi\in L^{\infty}(\R^3).
$$
It is further known that $\la_-$ is simple and $Z_-$ satisfies
$$
Z_-(y)\sim |y|^{-1} e^{-\sqrt{|\la_0|}|y|}~~\mbox{ as }~~|y|\to \infty.
$$
Consider an orthonormal basis $\{\Theta_i\}_{i=0}^{\infty}$ made up of spherical harmonics in $L^2(\mathbb{S}^{2})$, i.e.
$$\Delta_{\mathbb{S}^{2}}\Theta_i+\la_i \Theta_i=0 ~\mbox{ in } ~\mathbb{S}^{2}$$
with $0=\la_0<\la_1=\la_2=\la_3=2<\la_4\leq\cdots.$ More precisely, $\Theta_0(y)=a_0,~\Theta_i(y)=a_1y_i,~i=1,\cdots,3$ for two constants $a_0$, $a_1$.
For $h\in L^2(B_{2R})$, we decompose
\begin{equation*}
h(y,t)=\sum\limits_{j=0}^{\infty} h_j(r,t)\Theta_j(y/r),~r=|y|,~h_j(r,t)=\int_{\mathbb{S}^{2}} h(r\theta,t)\Theta_j(\theta)d\theta
\end{equation*}
and write $h=h^0+h^1+h^{\perp}$ with
$$h^0=h_0(r,t),~h^1=\sum\limits_{j=1}^3 h_j(r,t)\Theta_j,~h^{\perp}=\sum\limits_{j=4}^{\infty} h_j(r,t)\Theta_j.$$
Also, we decompose $\phi=\phi^0+\phi^1+\phi^{\perp}$ in a similar form. Define
\begin{equation}\label{def-normnua}
\|h\|_{\nu,2+\sigma}:=\sup_{(y,t)\in B_{2R}\times(0,T)} \mu_0^{-\nu}(t)(1+|y|^{2+\sigma})|h(y,t)|.
\end{equation}

\begin{prop}[\cite{dmw}]\label{propmode}
Let $\nu$, $\sigma$ be given positive numbers with $0<\sigma<2$. Then, for all sufficiently large $R>0$ and any $h = h(y,t)$ with $\|h\|_{\nu,2+\sigma} < +\infty$ that satisfies for all $j=0,1,\dots 3$
\[
\int\limits_{B_{2R}} h(y,t) Z_j(y) \diff y = 0 \ \ \text{for all} \ \
t \in (0 , T)
\]
there exist $\phi = \phi[h]$ and $e_0 = e_0[h]$ which solve
\begin{equation*}
	\mu_0^2\phi_{t} = \Delta \phi + 5w^4 \phi +h(y,t) \ \text{in} \ \
	B_{2R} \times (0, T),\ \  \phi(y,0) = e_0 Z_-(y)
	\ \ \text{in} \ \ B_{2R}.
\end{equation*}
Moreover, they define linear operators of h that satisfies the estimates
\begin{equation*}
	|\phi(y,t)| \lesssim \mu_0^{\nu}(t)
	\left[ \frac{R^{4-\sigma}}{1+|y|^3 } \|h^0\|_{\nu, 2+\sigma}
	+ \frac{R^{4-\sigma}}{1+|y|^4 } \|h^1\|_{\nu, 2+\sigma}
	+\frac{\|h\|_{\nu, 2+\sigma}}{1+|y|^{\sigma}}
	\right]
\end{equation*}

\begin{equation*}
|\nabla_y\phi(y,t)| \lesssim \mu_0^{\nu}(t)
\left[ \frac{R^{4-\sigma}}{1+|y|^4 } \|h^0\|_{\nu, 2+\sigma}
+ \frac{R^{4-\sigma}}{1+|y|^5 } \|h^1\|_{\nu, 2+\sigma}
+\frac{\|h\|_{\nu, 2+\sigma}}{1+|y|^{\sigma+1}}
\right]
\end{equation*}
and
\begin{equation*}
	|e_0[h]| \lesssim \|h\|_{\nu, 2+\sigma}.
\end{equation*}
\end{prop}

\begin{remark}
\noindent

\medskip

\begin{itemize}
\item[(1)] Since the blow-up profile constructed in the paper is radially symmetric, we only consider the case of mode $0$ in the Fourier decomposition. The construction of nonradial solutions can be carried out in a similar manner as in \cite[Section 10]{dmw} with nonradial perturbation.

\medskip

\item[(2)] Using the method of supersolution as in \cite[Lemma 7.3]{ddw}, one can improve the linear theory at mode 0 in the interior region:
$$
|\phi^0(y,t)|+(1+|y|)|\nabla \phi^0(y,t)|\lesssim \frac{\mu_0^{\nu}(t)R^{\frac{4-\sigma}{3}}}{1+|y|}\|h^0\|_{\nu,2+\sigma}\min\left\{1,(1+|y|)^{-2}R^{\frac{2(4-\sigma)}{3}}\right\}.
$$
If we define
\begin{equation}\label{def-normm0}
\|\phi\|_{0,\nu,\sigma}:=\sup_{(y,t)\in B_{2R}\times(0,T)} \frac{1+|y|}{\mu_0^{\nu}(t)R^{\frac{4-\sigma}{3}}(t)}\left[|\phi(y,t)|+(1+|y|)|\nabla \phi(y,t)|\right],
\end{equation}
then under the assumptions of Proposition \ref{propmode}, we have
\begin{equation*}
\|\phi^0\|_{0,\nu,\sigma}\lesssim \|h^0\|_{\nu,2+\sigma}.
\end{equation*}

\medskip

\item[(3)] Under the assumptions in Proposition \ref{propmode}, if the right hand side $h(y,t)$ further satisfies
$$h(y,t)\in C^{2k-2+2\epsilon,k-1+\epsilon}_{y,t}(B_{4R}\times(0,T))$$
for some $0<\epsilon<1$, then
$$\|\phi\|_{C^{2k,k}_{y,t}(B_{2R}\times(0,T))}\lesssim T^{\epsilon}.$$
This is a consequence of scaling argument and the parabolic Schauder estimate.
\end{itemize}
\end{remark}

\medskip

\subsection{Linear theory for the outer problem}\label{sec-outer}

\medskip

In this section, we develop the linear theory for the outer problem. The model problem is
\begin{equation}\label{outer-linear}
\left\{
\begin{aligned}
\psi_t =\Delta \psi + f, &~\mbox{ in }~\R^3\times(0,T)\\
\psi(\cdot,0)=\psi_0, &~\mbox{ in }~\R^3\\
\end{aligned}
\right.
\end{equation}
where the non-homogeneous term $f$ in \eqref{outer-linear} is assumed to be bounded with respect to the weights appearing in the outer problem \eqref{eqn-outer}. Define the weights
\begin{equation}\label{def-weights}
\begin{cases}
\varrho_1:=\mu_0^{\nu-\frac52}(t)R^{-2-a}(t)\chi_{\{|x|\leq 2\mu_0 R\}}\\
\varrho_2:=\frac{\mu_0 ^{\nu_2}}{|x|^{a_2}}\chi_{\{|x|\geq \mu_0 R\}}\\
\varrho_3:=1\\
\end{cases}
\end{equation}
where $\nu>0$, $0<a<1$, $1\leq a_2\leq 2$, $\nu_2>0$, and we choose $R(t)=\mu_0 ^{-\beta}(t)$ for $\beta\in(0,1/2)$ throughout the paper. We define the norms
\begin{equation}\label{def-norm**}
\|f\|_{**}:=\sup_{(x,t)\in \R^3\times(0,T)} \left(\sum\limits_{i=1}^3 \varrho_i(x,t)\right)^{-1}|f(x,t)|
\end{equation}

\begin{equation}\label{def-norm*}
\begin{aligned}
\|\psi\|_{*}:=&~\mu_0 ^{\frac12-\nu}(0)R^{a}(0)\|\psi\|_{L^{\infty}(\R^3\times(0,T))}+\mu_0 ^{\frac32-\nu}(0)R^{1+a}(0)\|\nabla\psi\|_{L^{\infty}(\R^3\times(0,T))}\\
&~+\sup_{(x,t)\in \R^3\times(0,T)}\left[\mu_0 ^{\frac12-\nu}(t)R^{a}(t)\left|\psi(x,t)-\psi(x,T)\right|\right]\\
&~+\sup_{(x,t)\in \R^3\times(0,T)}\left[\mu_0 ^{\frac32-\nu}(t)R^{1+a}(t)|\left|\nabla\psi(x,t)-\nabla\psi(x,T)\right|\right]\\
&~+\sup_{\R^3\times I_T} \frac{\mu_0 ^{2\gamma+\frac12-\nu}(t_2)R^{2\gamma+a}(t_2)}{(t_2-t_1)^{\gamma}} |\psi(x,t_2)-\psi(x,t_1)|,
\end{aligned}
\end{equation}
where $\nu>0$, $0<a,\gamma<1$, and the last supremum is taken over
$$\R^3\times I_T=\left\{(x,t_1,t_2): ~x\in\R^3, ~0\leq t_1\leq t_2\leq T,~t_2-t_1\leq \frac{1}{10}(T-t_2)\right\}.$$
For problem \eqref{outer-linear}, we have the following estimates.
\begin{prop}\label{outer-apriori}
Let $\psi$ be the solution to problem \eqref{outer-linear} with $\|f\|_{**}<+\infty$. Then it holds that
\begin{equation*}\label{est-outer-apriori}
\|\psi\|_*\lesssim \|f\|_{**}.
\end{equation*}
\end{prop}
Proposition \ref{outer-apriori} is established by the following three lemmas with different right hand sides.

\begin{lemma}\label{lemma-rhs1}
Let $\psi$ solve problem \eqref{outer-linear} with right hand side
$$|f(x,t)|\lesssim \mu_0 ^{\nu-\frac52}(t)R^{-2-a}(t)\chi_{\{|x|\leq 2\mu_0  R\}}$$
with $0<a<1$ and $\nu>0$. Then it holds that
\begin{equation}\label{outer-rhs1}
|\psi(x,t)|\lesssim\mu_0 ^{\nu-\frac12}(0)R^{-a}(0),
\end{equation}
\begin{equation}\label{outerT-rhs1}
|\psi(x,t)-\psi(x,T)|\lesssim\mu_0 ^{\nu-\frac12}(t)R^{-a}(t),
\end{equation}
\begin{equation}\label{outergradient-rhs1}
|\nabla \psi(x,t)|\lesssim \mu_0 ^{\nu-\frac32}(0)R^{-1-a}(0),
\end{equation}
\begin{equation}\label{outergradientT-rhs1}
|\nabla \psi(x,t)-\nabla \psi(x,T)|\lesssim \mu_0 ^{\nu-\frac32}(t)R^{-1-a}(t),
\end{equation}
and
\begin{equation}\label{outerholder-rhs1}
|\psi(x,t_2)-\psi(x,t_1)|\lesssim\mu_0 ^{\nu+\gamma_1-\frac52}(t_2)R^{\gamma_1-2-a}(t_2) (t_2-t_1)^{1-\gamma_1/2},
\end{equation}
where $0\leq t_1\leq t_2\leq T$ with $t_2-t_1\leq \frac{1}{10}(T-t_2)$ and $\gamma_1\in(0,1)$.
\end{lemma}

\begin{lemma}\label{lemma-rhs2}
Let $\psi$ solve problem \eqref{outer-linear} with right hand side
$$|f(x,t)|\lesssim \frac{\mu_0 ^{\nu_2}}{|x|^{a_2}}\chi_{\{|x|\geq \mu_0  R\}},$$
where $\nu_2>0$ and $1\leq a_2\leq 2$.
Then it holds that
\begin{equation}\label{outer-rhs2}
|\psi(x,t)|\lesssim \mu_0^{\nu_2+\frac{2-a_2}{4k}}(0),
\end{equation}
\begin{equation}\label{outerT-rhs2}
|\psi(x,t)-\psi(x,T)|\lesssim \mu_0^{\nu_2+\frac{2-a_2}{4k}}(t),
\end{equation}
\begin{equation}\label{outergradient-rhs2}
|\nabla \psi(x,t)|\lesssim \mu_0^{\nu_2+\frac{1-a_2}{4k}}(0),
\end{equation}
\begin{equation}\label{outergradientT-rhs2}
|\nabla \psi(x,t)-\nabla \psi(x,T)|\lesssim \mu_0^{\nu_2+\frac{1-a_2}{4k}}(t),
\end{equation}
and
\begin{equation}\label{outerholder-rhs2}
|\psi(x,t_2)-\psi(x,t_1)|\lesssim \frac{\mu_0 ^{\nu_2}(t)}{(\mu_0 (t)R(t))^{2\gamma}}(t_2-t_1)^{\gamma},
\end{equation}
where $0\leq t_1\leq t_2\leq T$ with $t_2-t_1\leq \frac{1}{10}(T-t_2)$ and $\gamma\in(0,1)$.
\end{lemma}

\begin{lemma}\label{lemma-rhs3}
Let $\psi$ solve problem \eqref{outer-linear} with right hand side
$$|f(x,t)|\lesssim 1.$$
Then it holds that
\begin{equation*}\label{outer-rhs3}
|\psi(x,t)|\lesssim t,
\end{equation*}
\begin{equation*}\label{outerT-rhs3}
|\psi(x,t)-\psi(x,T)|\lesssim (T-t),
\end{equation*}
\begin{equation*}\label{outergradient-rhs3}
|\nabla\psi(x,t)|\lesssim T^{1/2},
\end{equation*}
\begin{equation*}\label{outergradientT-rhs3}
|\nabla\psi(x,t)-\nabla\psi(x,T)|\lesssim (T-t)^{1/2},
\end{equation*}
and
\begin{equation*}\label{outerholder-rhs3}
|\psi(x,t_2)-\psi(x,t_1)|\lesssim (t_2-t_1),
\end{equation*}
where $0\leq t_1\leq t_2\leq T$ with $t_2-t_1\leq \frac{1}{10}(T-t_2)$.
\end{lemma}

Proposition \ref{outer-apriori} is a direct consequence of Lemma \ref{lemma-rhs1}, Lemma \ref{lemma-rhs2} and Lemma \ref{lemma-rhs3}, and the proofs of Lemma \ref{lemma-rhs1}, Lemma \ref{lemma-rhs2} and Lemma \ref{lemma-rhs3} are achieved by using Duhamel's formula similarly as in \cite{ddw}. Here we only give proof of Lemma \ref{lemma-rhs1} and Lemma \ref{lemma-rhs2} in the Appendix.

\medskip

\begin{remark}\label{remark5.2}
\noindent

\medskip

\begin{itemize}
\item[(1)] Let us point out the reason why we use the $\|\cdot\|_{*}$-norm of $\psi$ (\eqref{def-norm*}) only involving $\nu$ but not $\nu_2$ appearing in Lemma \ref{lemma-rhs2}. Lemma \ref{lemma-rhs2} is needed to deal with the right hand side of outer problem with cut-off $1-\eta_R$ in front. For convenience, when we carry out the inner--outer gluing procedure to bound right hand sides in the chosen topology, we will adjust $\nu_2$ such that the control of $\psi$ in Lemma \ref{lemma-rhs2} is better than that of Lemma \ref{lemma-rhs1}. This will result in a constraint for the parameters
$$
\nu_2+\frac{2-a_2}{4k}>\nu-\frac12+a\beta.
$$
In fact, the above constraint will be satisfied by the choices of parameters in Section \ref{subsec-choices}.

\medskip

\item[(2)] Under the assumptions of Proposition \ref{outer-apriori}, if we further assume that
$$f(x,t)\in C^{2k-2+2\epsilon,k-1+\epsilon}_{x,t}(\R^3\times(0,T)),$$
then
$$\|\psi\|_{C^{2k,k}_{x,t}(B_{\mu_0(0)R(0)}\times (0,T))}\lesssim T^{\epsilon}$$
for some $0<\epsilon<1$.
\end{itemize}
\end{remark}

\medskip

\medskip

%%%%%%%%%%%%%%%%%%%%%%%%%%%%%%%%%%%%%%%%%%%%%%%%%%%%%%%%

\section{The reduced equation for $\alpha(t)$}\label{sec-redu}

\medskip

From the linear theory for the inner problem \eqref{eqn-inner} in Section \ref{sec-inner}, orthogonality condition is required to guarantee the existence of solutions with sufficient space-time decay. In this section, we will adjust the scaling parameter $\mu(t)$ by such orthogonality condition.

\medskip

Recall that the slow decaying kernel for the linearized operator
\[
Z_0(y) = -\big[ y \cdot \nn w + \frac w 2 \big]
= \frac{3^{\frac 1 4}}{2} \frac{\abs{y}^2-1}{(1+\abs{y}^2)^{\frac 3 2}}.
\]
and
\[
\begin{split}
\mu(t) =\ & \mu_0(t) (1+\Lambda(t))^2,\\
\alpha(t) =\ & (-3^{\frac 1 4}) \mu_0^{-\frac 1 2}(t) (\mu_0(t) \Lambda(t))'.
\end{split}
\]
Since $\mu_0(t)\sim (T-t)^{2k}$, our aim is to look for
$$
\alpha(t)\sim (T-t)^{k-1}\Lambda(t), \quad k\in\Z_+.
$$
where $\Lambda(t)\to 0$ as $t\to T$.

From the linear theory in Section \ref{sec-inner}, the inner solution can be found in suitable topology if the following  orthogonality condition is satisfied
\begin{equation}\label{orthogonality}
\int_{B_{2R}} \mathcal H(\phi,\psi,\alpha) Z_0(y) dy =0 ~\mbox{  for all  }~ t\in(0,T),
\end{equation}
where $\mathcal H(\phi,\psi,\alpha)$ is defined in \eqref{def-mathcalH}. Define
\begin{equation}
\|h\|_{\delta}:=\sup_{t\in(0,T)} |(T-t)^{-\delta}h(t)|.
\end{equation}
It turns out that the reduced problem {orthogonality} is a problem involving the following nonlocal operator
$$
\int_0^t \frac{\alpha(s)}{(t-s)^{1/2}}ds,
$$
which is the nonlocal feature inherited from the slow decaying kernel $Z_0(y)$.

%%%%%%%%%%%%%%%%%%%%%%%%%%%%%%%%%%%%%%%%%%%%%%%%%%%%%%%%%%%%%%%%%%%%%%%%

\medskip

\subsection{A linear theory for the reduced equation}

\medskip

Before we consider the reduced problem \eqref{orthogonality}, we first develop a key linear theory for the following problem
\begin{equation}\label{linear-redu}
\int_0^t \frac{\alpha(s)}{(t-s)^{1/2}}ds=\sum\limits_{j=1}^{k} c_j \mathcal B^{(j)}(0,t)+h(t),
\end{equation}
where $h(t)\in C^k_t(0,T)$,  and for $j=1,\dots,k$, $c_j$ are constants and $\mathcal B^{(j)}$ are smooth functions to be determined. We have the following lemma concerning the solvability of problem \eqref{linear-redu}, which enables us to solve $\alpha(t)$ with sufficiently fast decay in problem \eqref{orthogonality}.

\medskip

\begin{lemma}\label{lemma-linear-redu}
For problem \eqref{linear-redu}, if $h(t)\in C^k_t(0,T)$, then there exist constants $c_j$ and smooth functions $\mathcal B^{(j)}$ such that problem \eqref{linear-redu} has a solution satisfying
$$
\alpha(t)\sim (T-t)^{k-1}\Lambda(t), \quad k\in\Z_+,
$$
where $\Lambda(t)\to 0$ as $t\to T$.
\end{lemma}

\medskip

\begin{proof}
In order to find $\alpha(t)$ with the above vanishing order, it suffices to show that
$$
\alpha(T)=\alpha'(T)=\alpha''(T)=\cdots=\alpha^{(k-1)}(T)=0.
$$
We shall choose $c_j$ and $\mathcal B^{(j)}$ using solutions for heat equations as building blocks.

We first find solutions $B_{j}(x,t)$ to the following heat equation
\begin{equation}\label{eqn-initial}
\left\{
\begin{aligned}
\partial_t B_{j} =\Delta B_{j}~&\mbox{ in }~\R^3\times(0,T)\\
B_{j}(\cdot,0)=B_{0,j}~&\mbox{ in }~\R^3\\
\end{aligned}
\right.
\end{equation}
where the decaying initial data $B_{0,j}$ will be chosen. Using Duhamel's formula in problem \eqref{eqn-initial}, we write
\begin{equation*}
B_j(0,t)=\int_{\R^3} e^{-\frac{|\xi|^2}{4t}} \frac{B_{0,j}(\xi)}{(4\pi t)^{3/2}} d\xi.
\end{equation*}
%Then
%\begin{equation}
%\Phi(t)=\sum\limits_{j=1}^{k} c_j \int_0^{t^{1/2}} B_j(0,t-u^2) du=\sum\limits_{j=1}^{k} c_j \pi^{-3/2} \int_0^{t^{1/2}} \int_{\R^3} e^{-|z|^2} B_{0,j}(2\sqrt{t-u^2}z) dz du.
%\end{equation}
%It is easy to adjust $c_j$ and $B_{0,j}$ such that \eqref{vanishingcondition} is satisfied.
Let us choose the initial condition
$$
B_{0,j}(|x|)=e^{-\kappa_j |x|^2}.
$$
Then
\begin{equation*}
\begin{aligned}
B_j(0,t)=&~\frac{1}{(4\pi)^{3/2}}\int_{\R^3} t^{-3/2} e^{-\frac{|\xi|^2}{4t}} e^{-\kappa_j |\xi|^2}d\xi\\
=&~\frac{1}{(4\pi)^{3/2}}\int_{\R^3} t^{-3/2} e^{-|x|^2} \left(\kappa_j t+\frac14\right)^{-3/2} dx\\
=&~\frac18 \left(\kappa_j t+\frac14\right)^{-3/2}
\end{aligned}
\end{equation*}
and
\begin{equation*}
\begin{aligned}
\int_0^t \frac{B_j(0,s)}{(t-s)^{1/2}}ds=&~2\int_0^{t^{1/2}} B_j(0,t-u^2) du\\
=&~\frac14 \int_0^{t^{1/2}} \left(\kappa_j (t-u^2)+\frac14\right)^{-3/2}du\\
=&~2 \int_0^{t^{1/2}} \left(4\kappa_j (t-u^2)+1\right)^{-3/2}du\\
=&~\frac{2t^{1/2}}{4\kappa_j t+1}.\\
\end{aligned}
\end{equation*}
We consider the linear combination of the initial data
\begin{equation*}
\mathcal B_0^{(i)}(x)=\sum\limits_{j=1}^k \ell_j^{(i)} B_{0,j}(|x|)
\end{equation*}
so that the corresponding solution at $x=0$ is
\begin{equation*}
\mathcal B^{(i)}(0,t)=\sum\limits_{j=1}^k \ell_j^{(i)} \frac18 \left(\kappa_j t+\frac14\right)^{-3/2}
\end{equation*}
and
\begin{equation}\label{int-mathcalB}
\int_0^t \frac{\mathcal B^{(i)}(0,s)}{(t-s)^{1/2}}ds=\sum\limits_{j=1}^k \ell_j^{(i)} \frac{2t^{1/2}}{4\kappa_j t+1}.
\end{equation}
Rearranging the constants $\tilde \ell_j^{(i)}=\frac{\ell_j^{(i)}}{2a_j}$ and $\tilde \kappa_j=\frac{1}{4\kappa_j}$, we denote
%\begin{equation}
%\int_0^t \frac{B_j(0,s)}{(t-s)^{1/2}}ds=\sum\limits_{j=1}^{k} \tilde c_j \frac{t^{1/2}}{t+\tilde a_j}.
%\end{equation}
%Denote
\begin{equation}\label{def-Upsilon}
\Upsilon_j(t)=\frac{t^{1/2}}{t+\tilde \kappa_j},\quad \Upsilon^{(i)}(t)=\sum\limits_{j=1}^k \tilde \ell_j^{(i)} \frac{t^{1/2}}{t+\tilde \kappa_j}.
\end{equation}
By adjusting free parameters $\tilde \ell_j^{(i)}$ and $\tilde \kappa_j$, we can find solutions $\Upsilon^{(i)}(t)$ with vanishing order $(T-t)^i$ near $T$. Indeed, writing $\widehat{\Upsilon}^{(i)}(t)=\sum\limits_{j=1}^k \tilde \ell_j^{(i)} \frac{1}{t+\tilde \kappa_j}$, $\Upsilon^{(i)}(t)\sim (T-t)^i$ is equivalent to showing the invertibility of the following system
\begin{equation}\label{invertmatrix}
M v_{\ell}=e_i,
\end{equation}
where
\begin{equation*}
M=\begin{pmatrix}
\frac{1}{T+\tilde \kappa_1} & \frac{1}{T+\tilde \kappa_1} & \cdots & \frac{1}{T+\tilde \kappa_k}\\
-\frac{1}{(T+\tilde \kappa_1)^2} & -\frac{1}{(T+\tilde \kappa_2)^2}  & \cdots & -\frac{1}{(T+\tilde \kappa_k)^2}\\
\vdots & \vdots & \ddots & \vdots \\
\frac{(-1)^k}{(T+\tilde \kappa_1)^k} & \frac{(-1)^k }{(T+\tilde \kappa_2)^k}  & \cdots & \frac{(-1)^k}{(T+\tilde \kappa_k)^k}\\
\end{pmatrix},
\end{equation*}
\begin{equation*}
v_{\ell}=(\tilde \ell_1^{(i)},\tilde \ell_2^{(i)},\dots,\tilde \ell_k^{(i)})^T,\quad e_i=(0,\cdots,0,{\underbrace{1}_{i{\rm-th~ entry}}},0,\cdots,0)^T.
\end{equation*}
Letting $\tilde \kappa_1,\dots,\tilde \kappa_k$ be different from each other, it is obvious that system \eqref{invertmatrix} is invertible. So we construct the solutions $\Upsilon^{(i)}(t)$ with vanishing order $(T-t)^i$ near $T$. Choosing a linear combination of $k$ such functions and plugging them into the reduced equation, we obtain
\begin{equation*}
\int_0^t \frac{\alpha(s)}{(t-s)^{1/2}} ds=\sum\limits_{j=1}^{k} c_j \mathcal B^{(j)}(0,t)+h(t),
\end{equation*}
where
%$h(t)$ comes from the right hand side of the inner problem, and
$c_j$ are free parameters in the initial data to be adjusted below.
%Expanding
%$$
%h(t)=\sum\limits_{j=0}^{k-1} d_j (T-t)^j + {\rm h.o.t.}
%$$
Then from Lemma \ref{lemma-RL}, we have
%$$
%\int_0^t \frac{\alpha(s)}{(t-s)^{1/2}} ds=\sum\limits_{j=1}^{k} c_j B_j(0,t)+ h(t)
%$$
%and thus
\begin{equation}\label{eqn-alpha}
\int_0^t \alpha(s) ds=\frac{1}{(\Gamma(\frac12))^2}\left(\sum\limits_{j=1}^{k} c_j \int_0^t \frac{\mathcal B^{(j)}(0,s)}{(t-s)^{1/2}} ds+ \int_0^t \frac{h(s)}{(t-s)^{1/2}} ds\right).
\end{equation}
We write
\begin{equation*}
\int_0^t \frac{h(s)}{(t-s)^{1/2}} ds=d_0(T)+\sum\limits_{j=1}^{k} d_j (T-t)^j+{\rm h.o.t.}
\end{equation*}
We can choose $c_j$ such that problem \eqref{eqn-alpha} has solution $\alpha(t)=o((T-t)^{k-1})$. In other words, our aim is to show
\begin{equation}\label{vanishingcondition}
\alpha(T)=\alpha'(T)=\alpha''(T)=\dots=\alpha^{(k-1)}(T)=0.
\end{equation}
From \eqref{eqn-alpha}, \eqref{int-mathcalB} and \eqref{def-Upsilon}, we obtain
\begin{equation*}
\int_0^t \alpha(s) ds=\frac{1}{(\Gamma(\frac12))^2}\left( \sum\limits_{j=1}^{k} c_j \Upsilon^{(j)}(t)+\sum\limits_{j=1}^{k} d_j (T-t)^j\right)+{\rm h.o.t.}
\end{equation*}
By the vanishing order of $\Upsilon^{(j)}(t)$ and choosing $c_j=d_j$, we conclude the validity of \eqref{vanishingcondition}. The proof is complete.
\end{proof}

\medskip

%%%%%%%%%%%%%%%%%%%%%%%%%%%%%%%%%%%%%%%%%%%%%%%%%%%%%%%%%%%%%%%%%%%%%%%%

\medskip

\subsection{Reduced equation for $\alpha(t)$}

\medskip

By the orthogonality condition \eqref{orthogonality}, we directly compute
\begin{equation*}
\begin{split}
& \int\limits_{B_{2R}} 5\big[ \big(u_{in}(\mu_0 y, t(\tau) )
+ \Phi_1(\mu_0 y, t(\tau)) \big)^4
- \mu_0^{-2}w^4(y)\big] \mu_0^2 \phi(y,\tau) Z_0(y) \diff y \\
+\ & \int\limits_{B_{2R}} 5 \mu_0^{\frac 1 2}(1 + \Lambda)^{-4}
w^4(\frac y{(1+\Lambda)^2})
\psi(\mu_0 y,t(\tau) )  Z_0(y) \diff y \\
+\ & \int\limits_{B_{2R}} \mu_0 \mu_0'
\big[ \frac 1 2 \phi(y, \tau) + \nn_y \phi(y,\tau) \cdot y \big]
Z_0(y) \diff y \\
+\ & \int\limits_{B_{2R}}
\mu_0^{\frac52} \left[S(U_1) - \chi(\frac{\mu_0 y}{c_0(T-t)^{\frac12}})
\frac {\alpha(t)} { (\mu^2 + |\mu_0 y|^2 )^{\frac 1 2}  }\right]
Z_0(y) \diff y
\\
+\ & \int\limits_{B_{2R}} \mu_0^{\frac 5 2}\big[ \big( u_{in}(\mu_0 y, t(\tau))
+ \Phi_1(\mu_0 y, t(\tau)) \big)^5
- u_{in}^5(\mu_0 y, t(\tau)) \big]  Z_0(y) \diff y = 0.
\end{split}
\end{equation*}
Since
\[
u_{in}(x,t)
= \mu^{-\frac{1}{2}}w (\frac{x}{\mu})
+ 2\mu_0'\,\mu^{\frac 1 2} J(\frac{x}{\mu})
= \mu^{-\frac{1}{2}}w (\frac y {(1+\Lambda)^2})
+ 2\mu_0'\,\mu^{\frac 1 2} J(\frac y {(1+\Lambda)^2}),
\]
we obtain
\begin{equation*}
\begin{split}
& \int\limits_{B_{2R}} 5\big[ \big( \mu^{-\frac{1}{2}}
w (\frac y {(1+\Lambda)^2})
+ 2\mu_0'\,\mu^{\frac 1 2} J(\frac y {(1+\Lambda)^2} )
+ \Phi_1(\mu_0 y, t(\tau)) \big)^4
- \mu_0^{-2}w^4(y)\big] \mu_0^2 \phi(y,\tau) Z_0(y) \diff y \\
+\ & \int\limits_{B_{2R}} 5 \mu_0^{\frac 1 2}(1 + \Lambda)^{-4}
w^4(\frac y{(1+\Lambda)^2})
\psi(\mu_0 y,t(\tau) )  Z_0(y) \diff y \\
+\ & \int\limits_{B_{2R}} \mu_0 \mu_0'
\big[ \frac 1 2 \phi(y, \tau) + \nn_y \phi(y,\tau) \cdot y \big]
Z_0(y) \diff y \\
+\ & \int\limits_{B_{2R}}
\mu_0^{\frac52} \left[S(U_1) - \chi(\frac{\mu_0 y}{c_0(T-t)^{\frac12}})
\frac {\alpha(t)} { (\mu^2 + |\mu_0 y|^2 )^{\frac 1 2}  }\right]
Z_0(y) \diff y
\\
+\ & \int\limits_{B_{2R}} \mu_0^{\frac 5 2}
\big[ \big( u_{in} + \Phi_1(\mu_0 y, t(\tau)) \big)^5
-  u_{in}^5 \big]  Z_0(y) \diff y = 0,
\end{split}
\end{equation*}
where we expand
\[
\begin{split}
& \big( u_{in} + \Phi_1 ( \mu_0 y, t(\tau) ) \big)^5  -  u_{in}^5 \\
= \ & 5 u_{in}^4 \Phi_1 (\mu_0 y, t(\tau)) + \Phi_1^2 (\mu_0 y, t(\tau))      \int_0^1 20 \big( u_{in} + s \Phi_1 (\mu_0 y, t(\tau)) \big)^3 (1-s)
\diff s \\
= \ & 5(\mu_0^{ - \frac 1 2} w(y))^4 \Phi_1 (\mu_0 y, t(\tau))
+ 5 \big[ u_{in}^4 - (\mu_0^{ - \frac 1 2} w(y))^4 \big]
\Phi_1 (\mu_0 y, t(\tau)) \\
\ \ &       + \Phi_1^2 (\mu_0 y, t(\tau))
\int_0^1 20 \big( u_{in} + s \Phi_1 (\mu_0 y, t(\tau)) \big)^3 (1-s)
\diff s  \\
= \ &  5(\mu_0^{ - \frac 1 2} w(y))^4 \Phi_1 (0, t(\tau)) +
5(\mu_0^{ - \frac 1 2} w(y))^4
\big( \Phi_1 (\mu_0 y, t(\tau)) - \Phi_1 (0, t(\tau)) \big) \\
& + 5 \big[ u_{in}^4 - (\mu_0^{ - \frac 1 2} w(y))^4 \big]
\Phi_1 (\mu_0 y, t(\tau))
+ \Phi_1^2 (\mu_0 y, t(\tau))
\int_0^1 20 \big( u_{in} + s \Phi_1 (\mu_0 y, t(\tau)) \big)^3 (1-s)
\diff s
\end{split}
\]
and we shall prove the leading term is
\[
5(\mu_0^{ - \frac 1 2} w(y))^4 \Phi_1 (0, t(\tau)) ,
\]
and all other terms have sufficiently fast decay. Indeed, we simplify the above equation and evaluate
\begin{equation*}
\begin{split}
& \int\limits_{B_{2R}} 5\big[ \big( \mu^{-\frac{1}{2}}
w (\frac y {(1+\Lambda)^2})
+ 2\mu_0'\,\mu^{\frac 1 2} J(\frac y {(1+\Lambda)^2} )
+ \Phi_1(\mu_0 y, t(\tau)) \big)^4
- \mu_0^{-2}w^4(y)\big] \mu_0^{\frac 3 2} \phi(y,\tau) Z_0(y) \diff y \\
+\ & \int\limits_{B_{2R}} 5 w^4(y)
\psi(\mu_0 y,t(\tau) )  Z_0(y) \diff y
+\  \int\limits_{B_{2R}} 5 [ (1 + \Lambda)^{-4}
w^4(\frac y{(1+\Lambda)^2}) -w^4(y) ]
\psi(\mu_0 y,t(\tau) )  Z_0(y) \diff y \\
+\ & \int\limits_{B_{2R}} \mu_0^{\frac 1 2}  \mu_0'
\big[ \frac 1 2 \phi(y, \tau) + \nn_y \phi(y,\tau) \cdot y \big]
Z_0(y) \diff y \\
+\ & \int\limits_{B_{2R}}
\mu_0^2 \left[S(U_1) - \chi(\frac{\mu_0 y}{c_0(T-t)^{\frac12}})
\frac {\alpha(t)} { (\mu^2 + |\mu_0 y|^2 )^{\frac 1 2}  }\right]
Z_0(y) \diff y
\\
+\ &  \int\limits_{B_{2R}} 5 w^4(y)
\Phi_1 (0, t(\tau))  Z_0(y) \diff y\\
+\ & \int\limits_{B_{2R}} \mu_0^2
\big[ \big( u_{in} + \Phi_1(\mu_0 y, t(\tau)) \big)^5
-  u_{in}^5 - 5(\mu_0^{ - \frac 1 2} w(y))^4 \Phi_1 (0, t(\tau)) \big]  Z_0(y) \diff y = 0.
\end{split}
\end{equation*}
Recall that
\[
Z_0(y) : = -\big[ y \cdot \nn w + \frac w 2 \big]
= \frac{3^{\frac 1 4}}{2} \frac{\abs{y}^2-1}{(1+\abs{y}^2)^{\frac 3 2}},
\]
\[
w(y) = 3^{\frac 14}(1+|y|^2)^{-\frac{1}{2}}.
\]
So
\[
\begin{split}
\int\limits_{B_{2R}} 5 w^4(y) Z_0(y) \diff y =\ & \frac 5 2 3^{\frac 5 4} \int_0^{2R}
(1 + r^2)^{-3.5}(r^2 - 1) 4\pi r^2 \diff r \\
=\ & 10\pi 3^{\frac 5 4} \frac{r^3(r^2-5)}{15(r^2+1)^{\frac 5 2}} \Big|_0^{2R} \\
=\ &  10\pi 3^{\frac 5 4} (\frac 1{15} + O(R^{-2})).
\end{split}
\]
Next we consider the nonlocal term
\begin{equation*}
\begin{split}
\Phi_1(0,t)
= \ & \sum\limits_{j=1}^{k} c_j \mathcal B^{(j)}(0,t)
+ \int_0^t \int\limits_{\RR^3}
(\frac 1{2\sqrt{\pi}})^3(t-s)^{-\frac 3 2}
e^{-\frac{\abs{\xi}^2}{4(t-s)}}
\chi(\frac{\xi}{c_0(T-s)^{\frac 1 2}   })\frac {\alpha(s)} { \mu(s) + \abs{\xi} } \diff \xi \diff s \\
= \ & \sum\limits_{j=1}^{k} c_j \mathcal B^{(j)}(0,t)
+ \int_0^t \int\limits_{\RR^3}
(\frac 1{2\sqrt{\pi}})^3(t-s)^{-\frac 3 2}
e^{-\frac{\abs{\xi}^2}{4(t-s)}}
\chi(\frac{\xi}{c_0(T-s)^{\frac 1 2}    })\frac {\alpha(s)} { \abs{\xi} } \diff \xi \diff s \\
& + \int_0^t \int\limits_{\RR^3}
(\frac 1{2\sqrt{\pi}})^3(t-s)^{-\frac 3 2}
e^{ -\frac{\abs{\xi}^2}{4(t-s)} }
\chi(\frac{\xi}{c_0(T-s)^{\frac 1 2}    }) \alpha(s)
\big( \frac 1 { \mu(s) + \abs{\xi} } - \frac 1 { \abs{\xi} } \big)
\diff \xi \diff s.\\
\end{split}
\end{equation*}
Notice
\[
\begin{split}
& \int_0^t \int\limits_{\RR^3}
(\frac 1{2\sqrt{\pi}})^3(t-s)^{-\frac 3 2}
e^{-\frac{\abs{\xi}^2}{4(t-s)}}
\chi(\frac{\xi}{c_0(T-s)^{\frac 1 2} } )\frac {\alpha(s)} { \abs{\xi} } \diff \xi \diff s \\
= \ & \int_0^t \int_0^{c_0(T-s)^{\frac 1 2}}
(\frac 1{2\sqrt{\pi}})^3 (t-s)^{-\frac 3 2}
e^{-\frac{r^2} {4(t-s)} }
\frac {\alpha(s)} {r}
4\pi r^2
\diff r \diff s \\
= \ & 4^{-1}\pi^{-\frac 1 2}
\int_0^t \int_0^{c_0(T-s)^{\frac 1 2}  }
(t-s)^{-\frac 3 2} e^{-\frac{r^2} {4(t-s)} } \alpha(s)
\diff (r^2) \diff s \\
= \ & 4^{-1} \pi^{-\frac 1 2}
\int_0^t  (t-s)^{-\frac 3 2} \alpha(s)
(-4)(t - s ) e^{ - \frac r{4(t-s )} } \Big|_{r=0}^{c_0^2 (T-s) }
\diff s \\
= \ & \pi^{-\frac 1 2}
\int_0^t  (t-s)^{-\frac 1 2} \alpha(s)
[ 1 - e^{ - \frac {c_0^2 (T-s)  }{4(t-s)} }  ] \diff s.
\end{split}
\]
Then
\begin{equation*}
\begin{split}
\Phi_1(0,t)
= \ & \sum\limits_{j=1}^{k} c_j \mathcal B^{(j)}(0,t) + \pi^{-\frac 1 2}
\int_0^t  (t-s)^{-\frac 1 2} \alpha(s)
[ 1 - e^{ - \frac {c_0^2 (T-s)  }{4(t-s)} }  ] \diff s \\
& + \int_0^t \int\limits_{\RR^3}
(\frac 1{2\sqrt{\pi}})^3(t-s)^{-\frac 3 2}
e^{-\frac{\abs{\xi}^2}{4(t-s)}}
\chi(\frac{\xi}{c_0(T-s)^{\frac 1 2}    })  \alpha(s)
\big( \frac 1 { \mu(s) + \abs{\xi} } - \frac 1 { \abs{\xi} } \big)
\diff \xi \diff s \\
\end{split}
\end{equation*}
and thus
\begin{equation*}
\begin{split}
& \int\limits_{B_{2R}} 5\big[ \big( \mu^{-\frac{1}{2}}
w (\frac y {(1+\Lambda)^2})
+ 2\mu_0'\,\mu^{\frac 1 2} J(\frac y {(1+\Lambda)^2} )
+ \Phi_1(\mu_0 y, t(\tau)) \big)^4
- \mu_0^{-2}w^4(y)\big] \mu_0^{\frac 3 2} \phi(y,\tau) Z_0(y) \diff y \\
+\ & \int\limits_{B_{2R}} 5 w^4(y)
\psi(\mu_0 y,t(\tau) )  Z_0(y) \diff y
+\  \int\limits_{B_{2R}} 5 [ (1 + \Lambda)^{-4}
w^4(\frac y{(1+\Lambda)^2}) -w^4(y) ]
\psi(\mu_0 y,t(\tau) )  Z_0(y) \diff y \\
+\ & \int\limits_{B_{2R}} \mu_0^{\frac 1 2} \mu_0'
\big[ \frac 1 2 \phi(y, \tau) + \nn_y \phi(y,\tau) \cdot y \big]
Z_0(y) \diff y \\
+\ & \int\limits_{B_{2R}}
\mu_0^2 \left[S(U_1) - \chi(\frac{\mu_0 y}{c_0(T-t)^{\frac12}})
\frac {\alpha(t)} { (\mu^2 + |\mu_0 y|^2 )^{\frac 1 2}  }\right]
Z_0(y) \diff y
\\
+\ & 10\pi 3^{\frac 5 4} (\frac 1{15} + O(R^{-2}))
(\sum\limits_{j=1}^{k} c_j \mathcal B^{(j)}(0,t)  ) \\
+\ &  10\pi 3^{\frac 5 4} (\frac 1{15} + O(R^{-2}))
\pi^{-\frac 1 2}
\int_0^t  (t-s)^{-\frac 1 2} \alpha(s)
[ 1 - e^{ - \frac {c_0^2 (T-s)  }{4(t-s)} }   ] \diff s \\
+\ &  10\pi 3^{\frac 5 4} (\frac 1{15} + O(R^{-2}))
\int_0^t \int\limits_{\RR^3}
(\frac 1{2\sqrt{\pi}})^3(t-s)^{-\frac 3 2}
e^{-\frac{\abs{\xi}^2}{4(t-s)}}
\chi(\frac{\xi}{c_0(T-s)^{\frac 1 2}    })
\big( \frac 1 { \mu(s) + \abs{\xi} } - \frac 1 { \abs{\xi} } \big)
\diff \xi \diff s \\
+\ & \int\limits_{B_{2R}} \mu_0^2
\big[ \big( u_{in} + \Phi_1(\mu_0 y, t(\tau)) \big)^5
-  u_{in}^5 - 5(\mu_0^{ - \frac 1 2} w(y))^4 \Phi_1 (0, t(\tau)) \big]  Z_0(y) \diff y = 0.
\end{split}
\end{equation*}
%$ 10\pi 3^{\frac 5 4} \frac 1{15} \pi^{-\frac 1 2} = 2 \pi^{\frac 1 2} 3^{\frac 1 4} $,
%divide $2 \pi^{\frac 1 2} 3^{\frac 1 4}$ for the equation.

The ansatz for the parameter function is
\[
\alpha(t) := (-3^{\frac 1 4}) \mu_0^{-\frac 1 2} (\mu_0 \Lambda(t))'
\rightarrow 0 \quad \mbox{as} \quad t \rightarrow T
\]
which is possible provided
\[
\alpha(t) \sim (T-t)^{-k} (T-t)^{2k} \Lambda(t) (T-t)^{-1}
\sim (T-t)^{k-1} \Lambda(t).
\]

Next we compute
\begin{equation}\label{orthogonalityeqn}
\begin{split}
&
\int_0^t  (t-s)^{-\frac 1 2} \alpha(s) \diff s
\\
= \ &
\pi^{\frac 1 2}
\sum\limits_{j=1}^{k} c_j \mathcal B^{(j)}(0,t)
+ O(R^{-2}) \pi^{\frac 1 2} \sum\limits_{j=1}^{k} c_j \mathcal B^{(j)}(0,t)
 + \int_0^t  (t-s)^{-\frac 1 2} \alpha(s)
e^{ - \frac {c_0^2 (T-s)  }{4(t-s)} }   \diff s \\
& - \int\limits_{B_{2R}}
\frac 5{2 \pi^{\frac 1 2} 3^{\frac 1 4}}
\big[ \big( \mu^{-\frac{1}{2}}
w (\frac y {(1+\Lambda)^2})
+ 2\mu_0'\,\mu^{\frac 1 2} J(\frac y {(1+\Lambda)^2} )
+ \Phi_1(\mu_0 y, t(\tau)) \big)^4
- \mu_0^{-2}w^4(y)\big] \\
& \ \ \ \ \  \times \mu_0^{\frac 3 2} \phi(y,\tau) Z_0(y) \diff y  - \int\limits_{B_{2R}}
\frac 5{2 \pi^{\frac 1 2} 3^{\frac 1 4}}
w^4(y)
\psi(\mu_0 y,t(\tau) )  Z_0(y) \diff y \\
& - \int\limits_{B_{2R}}
\frac 5{2 \pi^{\frac 1 2} 3^{\frac 1 4}}
[ (1 + \Lambda)^{-4}
w^4(\frac y{(1+\Lambda)^2}) -w^4(y) ]
\psi(\mu_0 y,t(\tau) )  Z_0(y) \diff y \\
& -  \int\limits_{B_{2R}}
\frac 1{2 \pi^{\frac 1 2} 3^{\frac 1 4}}
\mu_0^{\frac 1 2} \mu_0'
\big[ \frac 1 2 \phi(y, \tau) + \nn_y \phi(y,\tau) \cdot y \big]
Z_0(y) \diff y \\
& + \int\limits_{B_{2R}}
\frac 1{2 \pi^{\frac 1 2} 3^{\frac 1 4}}
\mu_0^2 \left[S(U_1) - \chi(\frac{\mu_0 y}{c_0(T-t)^{\frac12}})
\frac {\alpha(t)} { (\mu^2 + |\mu_0 y|^2 )^{\frac 1 2}  }\right]
Z_0(y) \diff y
\\
& -  O(R^{-2})
\int_0^t  (t-s)^{-\frac 1 2} \alpha(s)
[ 1 - e^{ - \frac {c_0^2 (T-s)  }{4(t-s)} }   ] \diff s \\
& -  ( 1 + O(R^{-2})) \pi^{\frac 1 2}
\int_0^t \int\limits_{\RR^3}
(\frac 1{2\sqrt{\pi}})^3(t-s)^{-\frac 3 2}
e^{-\frac{\abs{\xi}^2}{4(t-s)}}
\chi(\frac{\xi}{c_0(T-s)^{\frac 1 2}    })\alpha(s)
\big( \frac 1 { \mu(s) + \abs{\xi} } - \frac 1 { \abs{\xi} } \big)
\diff \xi \diff s \\
& - \int\limits_{B_{2R}}
\frac{1}{2 \pi^{\frac 1 2} 3^{\frac 1 4} }
\mu_0^2 \big[ \big( u_{in} + \Phi_1(\mu_0 y, t(\tau)) \big)^5
-  u_{in}^5 - 5(\mu_0^{ - \frac 1 2} w(y))^4 \Phi_1 (0, t(\tau)) \big]  Z_0(y) \diff y.
\end{split}
\end{equation}
From the theory on Riemann-Liouville fractional differential operator (see \cite{bookRL} for instance), we have
\begin{equation}\label{lemma-RL}
\begin{split}
& \frac{1}{\Gamma(\frac 1 2)}
\int_0^t (t-y )^{ - \frac 1 2}
\frac{1}{\Gamma(\frac 1 2)}
\int_0^y (y-s)^{- \frac 1 2} \alpha(s) \diff s \diff y \\
=\ &  \frac{1}{\Gamma(\frac 1 2)}
\int_0^{t} (t-y)^{ - \frac 1 2}
\frac{1}{\Gamma(\frac 1 2)}
\int_0^t \chi[s \le y] (y-s)^{- \frac 1 2} \alpha(s) \diff s \diff y \\
=\ & \frac{1}{(\Gamma(\frac 1 2) )^2}
\int_0^t \int_s^t
(t - s)^{-1} (\frac{t - y}{t - s})^{-\frac 1 2}
(\frac{y-s}{t - s})^{-\frac 1 2} \alpha(s)  \diff y \diff s \\
=\ & \int_0^t \alpha(s) \diff s.
\end{split}
\end{equation}
Therefore, we obtain
\begin{equation*}
\begin{split}
&\int_0^t \alpha(s) \diff s \\
=\ &
\frac{1}{\Gamma^2 (\frac 1 2)} \int_0^t (t-a)^{-\frac 1 2}
\int_0^a  (a-s)^{-\frac 1 2} \alpha(s)  \diff s \diff a \\
= \
&
\frac{1}{\Gamma (\frac 1 2)}
\sum\limits_{j=1}^{k} c_j
\int_0^t (t-a)^{-\frac 1 2}
\mathcal B^{(j)}(0,a) \diff a
 + \frac{1}{\Gamma (\frac 1 2)}
\sum\limits_{j=1}^{k} c_j
\int_0^t (t-a)^{-\frac 1 2}
O(R^{-2}(a)  ) \mathcal B^{(j)}(0,a) \diff a
\\
&
+ \frac{1}{\Gamma^2 (\frac 1 2)} \int_0^t (t-a)^{-\frac 1 2}
\int_0^a  (a-s)^{-\frac 1 2} \alpha(s)
e^{ - \frac {c_0^2 (T-s)  }{4(a-s)} }   \diff s \diff a \\
&
- \frac 5 {2 \pi^{\frac 3 2} 3^{\frac 1 4} } \int_0^t (t-a)^{-\frac 1 2}
\int\limits_{B_{2R(a)  }}
\Big[
\big( \mu(a)^{-\frac{1}{2}}
w (\frac y {(1+\Lambda (a) )^2})
+ 2\mu'_0 (a) \mu^{\frac 1 2} (a) J(\frac y {(1+\Lambda (a) )^2} )\\
&
\qquad\qquad\qquad\qquad\qquad+ \Phi_1(\mu_0 (a) y, a) \big)^4
- \mu_0^{-2} (a) w^4(y)\Big]  \times \mu_0^{\frac 3 2}(a) \phi(y,a) Z_0(y) \diff y \diff a \\
&
- \frac 5 {2 \pi^{\frac 3 2} 3^{\frac 1 4} }
\int_0^t (t-a)^{-\frac 1 2}
\int\limits_{B_{2R(a) } }  w^4(y)
\psi(\mu_0(a) y, a )  Z_0(y) \diff y \diff a \\
&
-  \frac 5 {2 \pi^{\frac 3 2} 3^{\frac 1 4} } \int_0^t (t-a)^{-\frac 1 2}
\int\limits_{B_{2R(a)  }} [ (1 + \Lambda(a) )^{-4}
w^4(\frac y{(1+\Lambda(a)  )^2}) -w^4(y) ]
\psi(\mu_0(a) y,a )  Z_0(y) \diff y \diff a  \\
& -  \frac 1 {2 \pi^{\frac 3 2} 3^{\frac 1 4} }
\int_0^t (t-a)^{-\frac 1 2}
\int\limits_{B_{2R(a)   }} \mu_0^{\frac 1 2}(a) \mu'_0(a)
\big[ \frac 1 2 \phi(y, a) + \nn_y \phi(y,a) \cdot y \big]
Z_0(y) \diff y \diff a \\
& +
\frac 1 {2 \pi^{\frac 3 2} 3^{\frac 1 4} }
\int_0^t (t-a)^{-\frac 1 2}
\int\limits_{B_{2R(a) }  }
\mu_0^2 (a) \Bigg[S(U_1)(\mu_0(a)y , a ) \\
&\qquad\qquad\qquad\qquad\qquad \qquad
- \chi(\frac{\mu_0 (a) y}{c_0(T-a)^{\frac12}})
\frac {\alpha(a)} { (\mu^2 (a) + |\mu_0 (a) y|^2 )^{\frac 1 2}  } \Bigg]
Z_0(y) \diff y \diff a
\\
&
- \frac{1}{\pi } \int_0^t (t-a)^{-\frac 1 2}
O(R^{-2} (a))
\int_0^a  (a-s)^{-\frac 1 2} \alpha(s)
[ 1 - e^{ - \frac {c_0^2 (T-s)  }{4(a-s)} }   ] \diff s \diff a
\\
&
-   \frac{1}{\pi^{\frac 1 2} } \int_0^t (t-a)^{-\frac 1 2}
( 1 + O(R^{-2}(a)  ))
\int_0^a \int\limits_{\RR^3}
(\frac 1 {2\sqrt{\pi}})^3  (a-s)^{-\frac 3 2}
e^{-\frac{ |\xi|^2 } {4(a-s)}}
\chi(\frac{\xi}{c_0(T-s)^{\frac 1 2}    })\\
&\qquad\qquad\qquad\qquad\qquad\qquad\qquad\times
\alpha(s)
\big( \frac 1 { \mu(s) + \abs{\xi} } - \frac 1 { \abs{\xi} } \big)
\diff \xi \diff s \diff a
\\
&
- \frac 1 {2 \pi^{\frac 3 2} 3^{\frac 1 4} }
\int_0^t (t-a)^{-\frac 1 2}
\int\limits_{B_{2R(a)  }}
\mu_0^2 (a) \Big[ \big( u_{in}(y,a) + \Phi_1(\mu_0(a) y, a) \big)^5
-  u_{in}^5(y,a)\\
&\qquad\qquad\qquad\qquad\qquad\qquad\qquad
 - 5(\mu_0^{ - \frac 1 2}(a) w(y))^4 \Phi_1 (0, a) \Big]  Z_0(y) \diff y \diff a.
\end{split}
\end{equation*}

Changing the variable $(t-s)^{1/2}=u$, we have
\begin{equation}\label{alphaalphaalpha}
\begin{split}
&
\int_0^t \alpha(s) \diff s \\
= \
&
\frac{1}{\Gamma (\frac 1 2)}
\sum\limits_{j=1}^{k} c_j
\int_0^t (t-a)^{-\frac 1 2}
\mathcal B^{(j)}(0,a) \diff a
 + \frac{2}{\Gamma (\frac 1 2)}
\sum\limits_{j=1}^{k} c_j
\int_0^{t^{\frac 1 2}}
O(R^{-2}(t-u^2)  ) \mathcal B^{(j)}(0,t-u^2) \diff u
\\
&
+ \frac{2}{\Gamma^2 (\frac 1 2)}
\int_0^{t^{\frac 1 2} }
\int_0^{t-u^2}  (t-u^2 -s)^{-\frac 1 2} \alpha(s)
e^{ - \frac {c_0^2 (T-s)  }{4(t-u^2-s)} }   \diff s \diff u \\
&
- \frac {5} { \pi^{\frac 3 2} 3^{\frac 1 4} }
\int_0^{t^{\frac 1 2} }
\int\limits_{B_{2R(t-u^2 )  }}
\Big[
\big( \mu ^{-\frac{1}{2}} (t-u^2)
w (\frac y {(1+\Lambda (t-u^2) )^2})\\
&
+ 2\mu'_0 (t-u^2) \mu^{\frac 1 2} (t-u^2) J(\frac y {(1+\Lambda (t-u^2) )^2} )
 + \Phi_1(\mu_0 (t-u^2) \  y, t-u^2) \big)^4
- \mu_0^{-2} (t-u^2) w^4(y) \Big] \\
&
\qquad\qquad\qquad\qquad
\times \mu_0^{\frac 3 2}(t-u^2) \phi (y,t-u^2) Z_0(y) \diff y \diff u \\
&
- \frac {5} { \pi^{\frac 3 2} 3^{\frac 1 4} }
\int_0^{t^{\frac 1 2}}
\int\limits_{B_{2R(t-u^2) } }  w^4(y)
\psi(\mu_0(t-u^2) y, t-u^2 )  Z_0(y) \diff y \diff u \\
&
-  \frac 5 { \pi^{\frac 3 2} 3^{\frac 1 4} } \int_0^{t^{\frac 1 2}}
\int\limits_{B_{2R(t-u^2)  }} \Big[ (1 + \Lambda(t-u^2) )^{-4}
w^4(\frac y{(1+\Lambda(t-u^2)  )^2})
 -w^4(y) \Big]\\
 &\qquad \qquad\qquad\qquad
 \times
\psi(\mu_0(t-u^2) y,t-u^2 )  Z_0(y) \diff y \diff u  \\
& -  \frac 1 { \pi^{\frac 3 2} 3^{\frac 1 4} }
\int_0^{t^{\frac 1 2}}
\int\limits_{B_{2R(t-u^2)   }} \mu_0^{\frac 1 2}(t-u^2) \mu'_0(t-u^2)
\big[ \frac 1 2 \phi(y, t-u^2) + \nn_y \phi(y,t-u^2) \cdot y \big]
Z_0(y) \diff y \diff u \\
& +
\frac 1 { \pi^{\frac 3 2} 3^{\frac 1 4} }
\int_0^{t^{\frac 1 2}}
\int\limits_{B_{2R(t-u^2) }  }
\mu_0^2 (t-u^2) \times
\Bigg[ S(U_1)(\mu_0(t-u^2) \ y , t-u^2 ) \\
&\qquad\qquad\qquad\qquad
- \chi(\frac{\mu_0 (t-u^2) y}{c_0(T-t + u^2)^{\frac12}})
\frac {\alpha(t-u^2)} { (\mu^2 (t-u^2) + |\mu_0 (t-u^2) y|^2 )^{\frac 1 2}  } \Bigg]
Z_0(y) \diff y \diff u
\\
&
- \frac{2}{\pi } \int_0^{t^{\frac 1 2}}
O(R^{-2} (t-u^2))
\int_0^{t-u^2}  (t-u^2-s)^{-\frac 1 2} \alpha(s)
[ 1 - e^{ - \frac {c_0^2 (T-s)  }{4(t-u^2-s)} }   ] \diff s \diff u
\\
&
-  \frac{2}{\pi^{\frac 1 2} } \int_0^{t^{\frac 1 2}}
( 1 + O(R^{-2}(t-u^2)  ))  \\
& \times
\int_0^{t-u^2} \int\limits_{\RR^3}
(\frac 1 {2\sqrt{\pi}})^3  (t-u^2-s)^{-\frac 3 2}
e^{-\frac{ |\xi|^2 } {4(t-u^2-s)}}
\chi(\frac{\xi}{c_0(T-s)^{\frac 1 2}    })\alpha(s)
\big( \frac 1 { \mu(s) + \abs{\xi} } - \frac 1 { \abs{\xi} } \big)
\diff \xi \diff s \diff u
\\
&
- \frac 1 { \pi^{\frac 3 2} 3^{\frac 1 4} }
\int_0^{t^{\frac 1 2}}
\int\limits_{B_{2R(t-u^2)  }}
\mu_0^2 (t-u^2)
\times
\Big[ \big( u_{in}(y,t-u^2) + \Phi_1(\mu_0(t-u^2) y, t-u^2) \big)^5\\
&\qquad\qquad\qquad\qquad\qquad
-  u_{in}^5(y,t-u^2) - 5(\mu_0^{ - \frac 1 2}(t-u^2) w(y))^4 \Phi_1 (0, t-u^2) \Big]  Z_0(y) \diff y \diff u.
\end{split}
\end{equation}

Now we check the differentiability of right hand side in \eqref{alphaalphaalpha} term by term. First, we consider
\[
\int_0^{t^{\frac 1 2}}
O(R^{-2}(t-u^2)  ) \mathcal B^{(j)}(0,t-u^2) \diff u.
\]
This term is smooth about $t$ near $T$ and we have the following estimate:
\[
\pp_t^{(i)}\int_0^{t^{\frac 1 2}}
O(R^{-2}(t-u^2)  ) \mathcal B^{(j)}(0,t-u^2) \diff u \Big|_{t=T}
\sim O(T^{\frac12 - i}).
\]
Next, we have
\[
\begin{aligned}
& \pp_t \left(\int_0^{t^{\frac 1 2} }
\int_0^{t-u^2}  (t-u^2 -s)^{-\frac 1 2} \alpha(s)
e^{ - \frac {c_0^2 (T-s)  }{4(t-u^2-s)} }   \diff s \diff u \right) \\
= \ &
\int_0^{t^{\frac 1 2} }
\int_0^{t-u^2}  \pp_t
\left( (t-u^2 -s)^{-\frac 1 2} \alpha(s)
e^{ - \frac {c_0^2 (T-s)  }{4(t-u^2-s)} }  \right)  \diff s \diff u.
\end{aligned}
\]
Similarly, we can check that this term is smooth about $t$ near $T$ and have the following estimate:
\[
\pp_t^{(i) } \left(\int_0^{t^{\frac 1 2} }
\int_0^{t-u^2}  (t-u^2 -s)^{-\frac 1 2} \alpha(s)
e^{ - \frac {c_0^2 (T-s)  }{4(t-u^2-s)} }   \diff s \diff u \right) \Big|_{t=T} \sim o(T^{\frac12 - i}).
\]
Then, the term
\[
\begin{aligned}
& \int_0^{t^{\frac 1 2} }
\int\limits_{B_{2R(t-u^2 )  }}
\Big[
\big( \mu ^{-\frac{1}{2}} (t-u^2)
w (\frac y {(1+\Lambda (t-u^2) )^2})
+ 2\mu'_0 (t-u^2) \mu^{\frac 1 2} (t-u^2) J(\frac y {(1+\Lambda (t-u^2) )^2} )
\\
& + \Phi_1(\mu_0 (t-u^2) \  y, t-u^2) \big)^4
- \mu_0^{-2} (t-u^2) w^4(y) \Big]
\times \mu_0^{\frac 3 2}(t-u^2) \phi (y,t-u^2) Z_0(y) \diff y \diff u
\end{aligned}
\]
is $C^{k}_t$ for $t$ near $T$ if we have $\phi \in C^{2k,k}_{x,t}$, and we have the following estimate

\[
\begin{aligned}
& \pp_t^{(i)} \int_0^{t^{\frac 1 2} }
\int\limits_{B_{2R(t-u^2 )  }}
\Big[
\big( \mu ^{-\frac{1}{2}} (t-u^2)
w (\frac y {(1+\Lambda (t-u^2) )^2})
+ 2\mu'_0 (t-u^2) \mu^{\frac 1 2} (t-u^2) J(\frac y {(1+\Lambda (t-u^2) )^2} )
\\
& + \Phi_1(\mu_0 (t-u^2) \  y, t-u^2) \big)^4
- \mu_0^{-2} (t-u^2) w^4(y) \Big] \\
&\qquad\qquad\qquad
\times \mu_0^{\frac 3 2}(t-u^2) \phi (y,t-u^2) Z_0(y) \diff y \diff u
\Big|_{t=T} = o(T^{\frac12 - i}).
\end{aligned}
\]
For the term
\[
\int_0^{t^{\frac 1 2}}
\int\limits_{B_{2R(t-u^2) } }  w^4(y)
\psi(\mu_0(t-u^2) y, t-u^2 )  Z_0(y) \diff y \diff u,
\]
it is $C^k $ for $t$ near $T$ if we have $\psi \in C^{2k,k}_{x,t}$. Further, we have the estimate
\[
\pp_t^{(i)} \int_0^{t^{\frac 1 2}}
\int\limits_{B_{2R(t-u^2) } }  w^4(y)
\psi(\mu_0(t-u^2) y, t-u^2 )  Z_0(y) \diff y \diff u \Big|_{t=T}
= o(T^{\frac12 -i}).
\]
Next, we consider
\[
 \int_0^{t^{\frac 1 2}}
\int\limits_{B_{2R(t-u^2)  }} \Big[ (1 + \Lambda(t-u^2) )^{-4}
w^4(\frac y{(1+\Lambda(t-u^2)  )^2})
-w^4(y) \Big]
\psi(\mu_0(t-u^2) y,t-u^2 )  Z_0(y) \diff y \diff u,
\]
which is $C^k_t$ if we assume $\psi \in C^{2k,k}_{x,t}$, and
\[
\begin{aligned}
& \pp_t^{(i)} \int_0^{t^{\frac 1 2}}
\int\limits_{B_{2R(t-u^2)  }} \Big[ (1 + \Lambda(t-u^2) )^{-4}
w^4(\frac y{(1+\Lambda(t-u^2)  )^2})
-w^4(y) \Big] \\
& \quad \quad \quad \quad \psi(\mu_0(t-u^2) y,t-u^2 )  Z_0(y) \diff y \diff u \Big|_{t = T}
\sim O(T^{\frac12 -i}).
\end{aligned}
\]
Next we consider
\[
\int_0^{t^{\frac 1 2}}
\int\limits_{B_{2R(t-u^2)   }} \mu_0^{\frac 1 2}(t-u^2) \mu'_0(t-u^2)
\big[ \frac 1 2 \phi(y, t-u^2) + \nn_y \phi(y,t-u^2) \cdot y \big]
Z_0(y) \diff y \diff u
\]
It is $C^k_t$ about $t$ near $T$ if we have $\phi \in C^{2k,k}_{x,t}$ and we have the following estimate
\[
\pp_t^{(i)} \int_0^{t^{\frac 1 2}}
\int\limits_{B_{2R(t-u^2)   }} \mu_0^{\frac 1 2}(t-u^2) \mu'_0(t-u^2)
\big[ \frac 1 2 \phi(y, t-u^2) + \nn_y \phi(y,t-u^2) \cdot y \big]
Z_0(y) \diff y \diff u \Big|_{t=T}
\sim o(T^{\frac 12 -i}).
\]
Next, for the term
\[
\begin{aligned}
&
\int_0^{t^{\frac 1 2}}
\int\limits_{B_{2R(t-u^2) }  }
\mu_0^2 (t-u^2) \times
\Bigg[ S(U_1)(\mu_0(t-u^2) \ y , t-u^2 ) \\
&
\qquad\qquad\qquad\qquad
- \chi(\frac{\mu_0 (t-u^2) y}{c_0(T-t + u^2)^{\frac12}})
\frac {\alpha(t-u^2)} { (\mu^2 (t-u^2) + |\mu_0 (t-u^2) y|^2 )^{\frac 1 2}  } \Bigg]
Z_0(y) \diff y \diff u,
\end{aligned}
\]
it is $C^k_t $ if we assume $\Lambda \in C^{k+1}_t $, and we have
\[
\begin{aligned}
&
\pp_t^{(i) } \int_0^{t^{\frac 1 2}}
\int\limits_{B_{2R(t-u^2) }  }
\mu_0^2 (t-u^2) \times
\Bigg[ S(U_1)(\mu_0(t-u^2) \ y , t-u^2 ) \\
&
\qquad\qquad
- \chi(\frac{\mu_0 (t-u^2) y}{c_0(T-t + u^2)^{\frac12}})
\frac {\alpha(t-u^2)} { (\mu^2 (t-u^2) + |\mu_0 (t-u^2) y|^2 )^{\frac 1 2}  } \Bigg]
Z_0(y) \diff y \diff u  \Big|_{t = T}
\sim o(T^{\frac12 -i} ).
\end{aligned}
\]
We next consider
\[
\int_0^{t^{\frac 1 2}}
O(R^{-2} (t-u^2))
\int_0^{t-u^2}  (t-u^2-s)^{-\frac 1 2} \alpha(s)
[ 1 - e^{ - \frac {c_0^2 (T-s)  }{4(t-u^2-s)} }   ] \diff s \diff u.
\]
It is smooth for $t$ near $T$ and we have the following estimate
\[
\pp_t^{(i)} \int_0^{t^{\frac 1 2}}
O(R^{-2} (t-u^2))
\int_0^{t-u^2}  (t-u^2-s)^{-\frac 1 2} \alpha(s)
[ 1 - e^{ - \frac {c_0^2 (T-s)  }{4(t-u^2-s)} }   ]
\diff s \diff u \Big|_{t=T}
\sim o(T^{\frac12 - i}).
\]
Next,
\[
\begin{aligned}
&
\int_0^{t^{\frac 1 2}}
( 1 + O(R^{-2}(t-u^2)  ))  \\
& \times
\int_0^{t-u^2} \int\limits_{\RR^3}
(\frac 1 {2\sqrt{\pi}})^3  (t-u^2-s)^{-\frac 3 2}
e^{-\frac{ |\xi|^2 } {4(t-u^2-s)}}
\chi(\frac{\xi}{c_0(T-s)^{\frac 1 2}    })\alpha(s)
\big( \frac 1 { \mu(s) + \abs{\xi} } - \frac 1 { \abs{\xi} } \big)
\diff \xi \diff s \diff u
\end{aligned}
\]
It is $C^k_t $ about $t$ near $T$ if we assume $\Lambda \in C^{k+1}_t $ and the following estimate holds
\[
\begin{aligned}
& \pp_t^{(i)}
\int_0^{t^{\frac 1 2}}
( 1 + O(R^{-2}(t-u^2)  ))
\int_0^{t-u^2} \int\limits_{\RR^3}
(\frac 1 {2\sqrt{\pi}})^3  (t-u^2-s)^{-\frac 3 2}
e^{-\frac{ |\xi|^2 } {4(t-u^2-s)}}
\chi(\frac{\xi}{c_0(T-s)^{\frac 1 2}    })\\
&\qquad\times
\alpha(s)
\big( \frac 1 { \mu(s) + \abs{\xi} } - \frac 1 { \abs{\xi} } \big)
\diff \xi \diff s \diff u \Big|_{t=T}
\sim o(T^{\frac12 - i}).
\end{aligned}
\]
Next,
\[
\begin{aligned}
&
\int_0^{t^{\frac 1 2}}
\int\limits_{B_{2R(t-u^2)  }}
\mu_0^2 (t-u^2)
\times
\Big[ \big( u_{in}(y,t-u^2) + \Phi_1(\mu_0(t-u^2) y, t-u^2) \big)^5\\
&\qquad\qquad\qquad\qquad\qquad
-  u_{in}^5(y,t-u^2) - 5(\mu_0^{ - \frac 1 2}(t-u^2) w(y))^4 \Phi_1 (0, t-u^2) \Big]  Z_0(y) \diff y \diff u
\end{aligned}
\]
It is $C^k_t$ about $t$ near $T$ if $\Lambda \in C^{k}_t$, and
\[
\begin{aligned}
&
\pp_t^{(i)}
\int_0^{t^{\frac 1 2}}
\int\limits_{B_{2R(t-u^2)  }}
\mu_0^2 (t-u^2)
\times
\Big[ \big( u_{in}(y,t-u^2) + \Phi_1(\mu_0(t-u^2) y, t-u^2) \big)^5\\
&\qquad\qquad\qquad\qquad
-  u_{in}^5(y,t-u^2) - 5(\mu_0^{ - \frac 1 2}(t-u^2) w(y))^4 \Phi_1 (0, t-u^2) \Big]  Z_0(y) \diff y \diff u \Big|_{t=T}
\sim o(T^{\frac 12 -i}).
\end{aligned}
\]

Then, we have the following equation
\begin{equation}
\alpha(t) =
\pp_t \left[
\frac{1}{\Gamma (\frac 1 2)}
\sum\limits_{j=1}^{k} c_j
\int_0^t (t-a)^{-\frac 1 2}
\mathcal B^{(j)}(0,a) \diff a
+ \frac{2}{\Gamma^2(\frac 1 2)}
\int_0^{t^{\frac 1 2} } h[\vec{c}, \Lambda, \psi, \phi](t - u^2) \diff u \right]
\end{equation}
and thus
\begin{equation}\label{Lameq}
\begin{aligned}
\Lambda (t)  = &
\mu_0^{-1}(t) \int_t^T
3^{-\frac 1 4} \mu_0^{\frac 1 2 } (b)
\pp_b \Bigg[
\frac{1}{\Gamma (\frac 1 2)}
\sum\limits_{j=1}^{k} c_j
\int_0^b (b-a)^{-\frac 1 2}
\mathcal B^{(j)}(0,a) \diff a \\
&\qquad\qquad\qquad\qquad\qquad\quad
+ \frac{2}{\Gamma^2(\frac 1 2)}
\int_0^{ b^{\frac 1 2 } } h[\vec{c}, \Lambda, \psi, \phi](b - u^2) \diff u \Bigg] \diff b.
\end{aligned}
\end{equation}
Define the space
\begin{equation}
\begin{aligned}
\mathcal{X}_{\Lambda} :=
& \Big\{ \Lambda(t) :
\| \Lambda(t) (T-t)^{-1 + \varepsilon} \|_{L^{\infty}} \le C_{0 \Lambda},\
\| \Lambda^{(1)} (t)(T-t)^{\varepsilon} \|_{L^{\infty} } \le C_{1\Lambda}, \
\cdots ,\\
&\qquad
\| \Lambda^{(k)} (t) (T-t)^{k-1+\varepsilon} \|_{L^{\infty}} \le C_{k \Lambda}, \Lambda(t) \in C^{k+1,\rho}(0,T-\delta), \foral \delta \in (0,T)
  \Big\}
  \end{aligned}
\end{equation}
where $C_{0 \Lambda}, C_{1 \Lambda},\dots, C_{k \Lambda}$ are some fixed small constants and $\varepsilon$ is a  small positive constant.

\begin{equation}
\mathcal{X}_{c} := \left\{ \vec{c}=(c_1, c_2,\dots, c_k) :
|c_j| \le C_c T^{\frac 1 2 - j - \varepsilon},\ j=1,2,\dots,  k \right\}
\end{equation}
where $C_c$ is a fixed constant.

\medskip

We aim to solve $\eqref{Lameq}$ for $\Lambda \in \mathcal{X}_{\Lambda}$, $\vec{c} \in \mathcal{X}_c$ by Schauder fixed point theorem.

For all $\Lambda_1 \in \mathcal{X}_{\Lambda}$ , $\vec{c}_1 \in \mathcal{X}_c$, we want to find the unique $\vec{c}_2 $ to get $\Lambda_2 \in \mathcal{X}_{\Lambda}$ satisfying the following equation:

\begin{equation}\label{Lam2}
\begin{aligned}
\Lambda_2 (t)  = &~
\mu_0^{-1}(t) \int_t^T
3^{-\frac 1 4} \mu_0^{\frac 1 2 } (b)
\pp_b \Bigg[
\frac{1}{\Gamma (\frac 1 2)}
\sum\limits_{j=1}^{k} c_{2j}
\int_0^b (b-a)^{-\frac 1 2}
\mathcal B^{(j)}(0,a) \diff a \\
&~\qquad\qquad\qquad\qquad\qquad\quad
+ \frac{2}{\Gamma^2(\frac 1 2)}
\int_0^{ b^{\frac 1 2 } }
h[\vec{c}_1, \Lambda_1, \psi[\Lambda_1], \phi[\Lambda_1] ] (b - u^2) \diff u \Bigg] \diff b.
\end{aligned}
\end{equation}
Since we expect $\Lambda_2 \in \mathcal{X}_{\lambda}$, we have to choose suitable $c_{2j}$ to cancel the lower power of $T-t$ on the right hand side:
\[
\begin{aligned}
c_{2j} = -\frac 2{\Gamma(\frac 1 2) j!}
\pp_b^j
\Big( \int_0^{ b^{\frac 1 2 } }
h[\vec{c}_1, \Lambda_1, \psi[\Lambda_1], \phi[\Lambda_1] ] (b - u^2) \diff u
\Big) \Big|_{b=T} = O(T^{\frac12 -j}).
%=\ & -\frac 2{\Gamma(\frac 1 2) j!}
%\sum_{i=0}^{j-1}
%h^{(i)} [\vec{c}_1, \Lambda_1, \psi[\Lambda_1], \phi[\Lambda_1] ] (0)
%\frac 1 2 (\frac 1 2 -1)\cdots (\frac 1 2 - j+1+i)
%T^{\frac 1 2 -j +i}
%\\
%& -\frac 2{\Gamma(\frac 1 2) j!}
%\int_0^{ T^{\frac 1 2 } }
%h^{(j)} [\vec{c}_1, \Lambda_1, \psi[\Lambda_1], \phi[\Lambda_1] ] (T - u^2) \diff u
\end{aligned}
\]
The higher derivatives of $h$ are well defined here since $\Lambda \in \mathcal{X}_{\Lambda}$ and $\alpha(t) \in C^{k}_t(0,T-\delta)$. We can use Schauder estimate to improve the regularity of $\phi$, $\psi $ to $C^{2k+2+2\rho, k+1+ \rho}_{x,t}$ for $t \in (0,T-\delta)$ for any small $\delta$. So $h$ also has higher derivatives up to order $k$.

Taking higher derivatives for \eqref{Lam2} and choosing $T$ small enough, we have $\Lambda_2 \in \mathcal X _{\Lambda}$.

Next, we show the Lipschitz continuity of $\Lambda_2$.
For any $0 \le t_1 < t_2 \le T-\delta $,
\begin{equation}
\begin{aligned}
& \big| \Lambda_2 (t_1) - \Lambda_2 (t_2) \big| \\
=\ & \Bigg|
\mu_0^{-1}(t_1)
\int_{t_1}^{t_2}
3^{-\frac 1 4} \mu_0^{\frac 1 2 } (b)
\pp_b \Bigg[
\frac{1}{\Gamma (\frac 1 2)}
\sum\limits_{j=1}^{k} c_{j2}
\int_0^b (b-a)^{-\frac 1 2}
\mathcal B^{(j)}(0,a) \diff a \\
&\qquad
+ \frac{2}{\Gamma^2(\frac 1 2)}
\int_0^{ b^{\frac 1 2 } }
h[\vec{c}_1, \Lambda_1, \psi[\Lambda_1], \phi[\Lambda_1] ] (b - u^2) \diff u \Bigg] \diff b
\\
& +
( \mu_0^{-1}(t_1) - \mu^{-1} (t_2) )\times
\int_{t_2}^T
3^{-\frac 1 4} \mu_0^{\frac 1 2 } (b)
\pp_b \Bigg[
\frac{1}{\Gamma (\frac 1 2)}
\sum\limits_{j=1}^{k} c_{j2}
\int_0^b (b-a)^{-\frac 1 2}
\mathcal B^{(j)}(0,a) \diff a \\
& \qquad
+ \frac{2}{\Gamma^2(\frac 1 2)}
\int_0^{ b^{\frac 1 2 } }
h[\vec{c}_1, \Lambda_1, \psi[\Lambda_1], \phi[\Lambda_1] ] (b - u^2) \diff u \Bigg] \diff b  \Bigg|
\\
\lesssim\ & (T - t_1)^{-2k} \int_{t_1}^{t_2} (T-b)^{2k} \diff b
+ (T- \theta t_1 - (1-\theta) t_2 )^{-2k-1} (t_2 - t_1)\int_{t_2}^T (T-b)^{2k} \diff b \\
\lesssim \ &
t_2 - t_1.
\end{aligned}
\end{equation}

Similarly, we can take $k$-th derivative of \eqref{Lam2} to prove $\Lambda^{(k)}_2 (t) \in {\rm Lip}(0,T-\delta)$, $\forall  \delta \in (0,T)$.

We shall solve the full inner--outer gluing system together with the reduced problem \eqref{orthogonality} in Section \ref{sec-innerouter}.

%Rearranging the constants $\tilde c_j=\frac{c_j}{2a_j}$ and $\tilde a_j=\frac{1}{4a_j}$, we have
%\begin{equation}
%\int_0^t \alpha(s) ds= \sum\limits_{j=1}^{k} \tilde c_j \frac{t^{1/2}}{t+\tilde a_j}+\sum\limits_{j=1}^{k} d_j (T-t)^j.
%\end{equation}
%Denote
%$$
%\Upsilon_j(t)=\frac{t^{1/2}}{t+\tilde a_j},\quad \Upsilon(t)=\sum\limits_{j=1}^k \frac{t^{1/2}}{t+\tilde a_j}
%$$
%Then \eqref{vanishingcondition} is equivalent to the invertibility of the following system

%Indeed, now we have $k$ free parameters $c_j$ and $k$ free initial conditions $B_{0,j}$, to prove problem \eqref{eqn-alpha} has desired solution, it suffices to show that
%\begin{equation}\label{vanishingcondition}
%\Phi(T)=\int_0^T \alpha(s) ds, \quad \Phi'(T)=\Phi''(T)=\dots=\Phi^{(k)}(T)=0, \quad \Phi^{(k+1)}(T) \neq 0,
%\end{equation}
%where we define
%\begin{equation}
%\Phi(t)=\sum\limits_{j=1}^{k} c_j \int_0^t \frac{B_j(0,s)}{(t-s)^{1/2}} ds.
%\end{equation}

\medskip

%%%%%%%%%%%%%%%%%%%%%%%%%%%%%%%%%%%%%%%%%%%%%%%%%%%%

%%%%%%%%%%%%%%%%%%%%%%%%%%%%%%%%%%%%%%%%%%%%%%%%%%%%

\section{Solving the inner--outer gluing system}\label{sec-innerouter}

\medskip

In this section, we will solve the inner outer gluing system \eqref{eqn-outer} and \eqref{eqn-inner} by using the linear theories developed in Section \ref{sec-lt}
together with the fixed point argument. We first estimate the right hand sides of problems \eqref{eqn-outer} and \eqref{eqn-inner} under the topology chosen in Section \ref{sec-lt}.

\medskip

\subsection{The outer problem: estimate of $\mathcal G$}

\medskip

Recall from \eqref{def-mathcalG} that
\begin{equation*}
\begin{aligned}
\mathcal G(\phi,\psi,\alpha)=& - \pp_t \eta_R \mu_0^{-\frac 1 2}\phi(\frac x {\mu_0},t)
  + \Delta_x \eta_R \mu_0^{-\frac 1 2}\phi(\frac x {\mu_0},t)
  + 2\nn_x \eta_R \cdot \mu_0^{-\frac 3 2} \nn_y \phi(\frac x{\mu_0},t) \\
  & +\left(5(U_1 + \Phi_1)^4 (1-\eta_R) + 5\big[ (U_1 + \Phi_1)^4
        - \big( \mu^{-\frac 1 2} w(\frac x{\mu}) \big)^4 \big] \eta_R\right)\psi\\
& + (U_1 + \Phi_1 +\Phi_2)^5 - (U_1 + \Phi_1)^5 -5(U_1 + \Phi_1)^4 \Phi_2 \\
& + \left[S(U_1) - \chi(\frac{x}{c_0(T-t)^{\frac 1 2}  })
\frac {\alpha(t)} { (\mu^2 + |x|^2 )^{\frac 1 2}  }\right](1-\eta_R)
+ \big[ (U_1 + \Phi_1)^5 - U_1^5 \big](1-\eta_R).
\end{aligned}
\end{equation*}
We write
\begin{equation}
\begin{aligned}
g_1:=&~- \pp_t \eta_R \mu_0^{-\frac 1 2}\phi(\frac x {\mu_0},t)
  + \Delta_x \eta_R \mu_0^{-\frac 1 2}\phi(\frac x {\mu_0},t)
  + 2\nn_x \eta_R \cdot \mu_0^{-\frac 3 2} \nn_y \phi(\frac x{\mu_0},t), \\
g_2:=&~\left(5(U_1 + \Phi_1)^4 (1-\eta_R) + 5\big[ (U_1 + \Phi_1)^4
        - \big( \mu^{-\frac 1 2} w(\frac x{\mu}) \big)^4 \big] \eta_R\right)\psi,\\
g_3:=&~(U_1 + \Phi_1 +\Phi_2)^5 - (U_1 + \Phi_1)^5 -5(U_1 + \Phi_1)^4 \Phi_2, \\
g_4:=&~ \left[S(U_1) - \chi(\frac{x}{c_0(T-t)^{\frac 1 2}  })
\frac {\alpha(t)} { (\mu^2 + |x|^2 )^{\frac 1 2}  }\right](1-\eta_R)
+ \big[ (U_1 + \Phi_1)^5 - U_1^5 \big](1-\eta_R).
\end{aligned}
\end{equation}

\medskip

\noindent {\bf Estimate of $g_1$}

\medskip

Since $\|\phi\|_{0,\nu,\sigma}<+\infty$, we have
\begin{equation*}
|g_1|\lesssim \mu_0^{\nu-\frac52}(t)R^{-1-\sigma}(t)\chi_{\{|x|\sim \mu_0R\}}
\end{equation*}
and thus for some $\epsilon>0$
\begin{equation}\label{est-g1}
\|g_1\|_{**}\lesssim T^{\epsilon}(1+\|\phi\|_{0,\nu,\sigma})
\end{equation}
provided
\begin{equation}\label{est-g1-cond}
1+a-\sigma<0.
\end{equation}

\medskip

\noindent {\bf Estimate of $g_2$}

\medskip

Since $\|\psi\|_*<+\infty$ and we choose the initial data such that $\psi(0,T)=0$, we have
\begin{equation*}
\begin{aligned}
|\psi(x,t)|=&~|\psi(x,t)-\psi(0,T)|\\
\leq&~|\psi(x,t)-\psi(x,T)|+|\psi(x,T)-\psi(0,T)|\\
\lesssim&~ \mu_0^{\nu-\frac12}(t)R^{-a}(t)\|\psi\|_*.
\end{aligned}
\end{equation*}
Then thanks to the cut-off $\eta_R$, we have
\begin{equation*}
\begin{aligned}
|g_2|\lesssim&~ (1-\eta_R) \frac{1}{1+|y|^4}|\psi|\\
\lesssim&~\frac{\mu_0^{\nu-\frac12}(t)R^{-a}(t)}{1+|y|^4}\|\psi\|_* \chi_{\{|x|\geq \mu_0  R\}}\\
\lesssim&~\frac{\mu_0^{\nu-\frac12+a_2}(t)R^{a_2-a-4}(t)}{|x|^{a_2}}\|\psi\|_* \chi_{\{|x|\geq \mu_0  R\}}.\\
\end{aligned}
\end{equation*}
So we have that for some $\epsilon>0$
\begin{equation}\label{est-g2}
\|g_2\|_{**} \lesssim T^{\epsilon} \|\psi\|_*
\end{equation}
provided
\begin{equation}\label{est-g2-cond}
\nu-\nu_2-\frac12+a_2-\beta(a_2-a-4)>0.
\end{equation}

\medskip

\noindent {\bf Estimate of $g_3$}

\medskip

We evaluate
\begin{equation*}
\begin{aligned}
|g_3|\lesssim&~ |(U_1 + \Phi_1)^3 \Phi_2^2|\\
\lesssim&~ \frac{\mu_0^{-\frac32}(t)}{1+|y|^3}\left(|\psi|^2\chi_{\{|x|\lesssim \sqrt{T-t}\}}+\mu_0^{-1}|\phi|^2\chi_{\{|x|\lesssim \mu_0R\}}\right)\\
\lesssim&~\frac{\mu_0^{-\frac32}(t)}{1+|y|^3}\left(\mu_0^{2\nu-1}(t)R^{-2a}(t)\|\psi\|_*^2\chi_{\{|x|\lesssim \sqrt{T-t}\}}+\frac{\mu_0^{2\nu-1}(t)R^{\frac{8-2\sigma}{3}}(t)}{1+|y|^2}\|\phi\|_{0,\nu,\sigma}^2\chi_{\{|x|\lesssim \mu_0R\}}\right)\\
\lesssim&~\mu_0^{2\nu-\frac52}(t)R^{-2a}(t)\|\psi\|_*^2\chi_{\{|x|\lesssim \mu_0R\}}+\mu_0^{2\nu-\frac52}(t)R^{\frac{8-2\sigma}3}(t)\|\phi\|_{0,\nu,\sigma}^2\chi_{\{|x|\lesssim \mu_0R\}}\\
&~+\mu_0^{2\nu-\frac52+a_2}(t)R^{a_2-2a-3}(t)\frac{1}{|x|^{a_2}}\|\psi\|_*^2\chi_{\{\mu_0R\lesssim|x|\lesssim \sqrt{T-t}\}}.
\end{aligned}
\end{equation*}
Then for some $\epsilon>0$
\begin{equation}\label{est-g3}
\|g_3\|_{**}\lesssim T^{\epsilon}(1+\|\psi\|_*^2+\|\phi\|_{0,\nu,\sigma}^2)
\end{equation}
provided
\begin{equation}\label{est-g3-cond}
\begin{cases}
\nu-\beta(2-a)>0\\
\nu-\beta(\frac{14-2\sigma}{3}+a)>0\\
2\nu-\frac52+a_2-\beta(a_2-2a-3)-\nu_2>0\\
\end{cases}
\end{equation}

\medskip

\noindent {\bf Estimate of $g_4$}

\medskip

To estimate $g_4$, we first estimate
\begin{equation}\label{est-g4-1}
\begin{aligned}
&\quad\left|\left[S(U_1) - \chi(\frac{x}{c_0(T-t)^{\frac 1 2}  })
\frac {\alpha(t)} { (\mu^2 + |x|^2 )^{\frac 1 2}  }\right](1-\eta_R)\right|\\
&\lesssim\left|\left[\eta_1 S_{in} - \chi(\frac{x}{c_0(T-t)^{\frac12}})
\frac {\alpha(t)} { (\mu^2 + |x|^2 )^{\frac 1 2}  }\right](1-\eta_R)\right|\\
&\quad +\left|(1-\eta_{o1} )\eta_{o2}
( - \pp_t u_{out} + \Delta_x u_{out} + u_{out}^5)\right| \\
&\quad +\Bigg|- \pp_t \eta_1 u_{in}
+ \Delta_x \eta_1 u_{in}
+ 2 \nn_x \eta_1 \nn_x u_{in} -\pp_t [(1-\eta_{o1} )\eta_{o2}] u_{out} \\
&\qquad \quad + \Delta_x [(1-\eta_{o1} )\eta_{o2}] u_{out}
+ 2 \nabla_x [(1-\eta_{o1} )\eta_{o2}] \nabla_x u_{out}\Bigg|
\\
& \quad + \left|\big[ \eta_1 u_{in} + (1-\eta_{o1} )\eta_{o2} u_{out}\big]^5
- \eta_1 u_{in}^5 - (1-\eta_{o1} )\eta_{o2} u_{out}^5\right|.
\end{aligned}
\end{equation}
From \eqref{est-g4main}, we have
\begin{equation}\label{est-g4-2}
\begin{aligned}
&\quad \left|\left[\eta_1 S_{in} - \chi(\frac{x}{c_0(T-t)^{\frac12}})
\frac {\alpha(t)} { (\mu^2 + |x|^2 )^{\frac 1 2}  }\right](1-\eta_R)\right|\\
&\lesssim
|\alpha(t)| \frac{ \mu^2 }{(\mu^2 + |x|^2)^{\frac 3 2}} \eta_1 (1-\eta_R)
 + \frac { | \alpha(t) | } {(\mu^2 + |x|^2 )^{\frac 1 2} }
\chi (r\le \frac{ |x| }{\sqrt{T-t}} \le c_0)(1-\eta_R)
 \\
 &\quad + \Lambda^2 \mu^{-\frac 1 2 } \mu_0'
 \left| 3^{\frac 1 4} \frac{ \mu^2 }{(\mu^2 + |x|^2)^{\frac 3 2}}
 - \frac { 3^{\frac 1 4}} {2} \frac 1 {(\mu^2 + |x|^2 )^{\frac 1 2}} \right| \eta_1 (1-\eta_R)
 \\
 &\quad + \left|2 \mu_0'' \mu^{\frac 1 2} h(\frac{x}{\mu}) \eta_1
 + \mu_0' \mu^{-\frac 1 2} \mu' h(\frac{x}{\mu}) \eta_1
 +2\mu_0' \mu^{-\frac 1 2} \mu' h(\frac{x}{\mu} ) \eta_1\right|(1-\eta_R) \\
 &\quad + \frac{\mu^{\frac 5 2}}{(\mu^2 + |x|^2)^{\frac 3 2}} (\mu_0')^2 \left|h^2(\frac x{\mu_0})\right| \eta_1 (1-\eta_R)
 \\
 & \lesssim \frac{\mu_0^{-\frac32+a_2}\mu_0' R^{a_2-3}}{|x|^{a_2}}\chi_{\{|x|\geq\mu_0R\}}\|\Lambda\|_{\infty}+\frac{\mu_0^{-\frac12}\mu_0'}{|x|}\chi_{\{|x|\geq\mu_0R\}}\|\Lambda\|_{\infty}\\
 &\quad+ \frac{\mu_0^{\frac12+\frac{a_2-3}{4k}}}{|x|^{a_2}}\chi_{\{|x|\geq\mu_0R\}}+ \frac{\mu_0^{\frac52+\frac{a_2-5}{4k}}}{|x|^{a_2}}\chi_{\{|x|\geq\mu_0R\}}
\end{aligned}
\end{equation}
where the function $h$ is defined in \eqref{def-hhh}.  Similarly, we have the following estimates for the rest terms. We evaluate the term
\begin{equation}\label{est-g4-3}
\begin{aligned}
|S_{out}|=&~|-\partial_t u_{out}+\Delta u_{out} +u_{out}^5|\\
\lesssim&~(T-t)^{5(k-\zeta_1)}.
\end{aligned}
\end{equation}
If we choose $\zeta_1=\zeta_2=\frac12$, $r=r_1$ and $r_2>3r$, then we have
\begin{equation}\label{est-g4-4}
\begin{aligned}
&~\Bigg|- \pp_t \eta_1 u_{in}
+ \Delta_x \eta_1 u_{in}
+ 2 \nn_x \eta_1 \nn_x u_{in} -\pp_t [(1-\eta_{o1} )\eta_{o2}] u_{out} \\
&~\qquad \quad + \Delta_x [(1-\eta_{o1} )\eta_{o2}] u_{out}
+ 2 \nabla_x [(1-\eta_{o1} )\eta_{o2}] \nabla_x u_{out}\Bigg|
\\
\lesssim &~ |\partial_t \eta_1(u_{in}-u_{out})|+|\nn_x \eta_1\cdot\nn_x (u_{in}-u_{out})|+|\Delta_x \eta_1(u_{in}-u_{out})|\\
\lesssim &~(T-t)^{k-\frac12},
\end{aligned}
\end{equation}
where we have used the cancellation in the matching \eqref{matching-in}--\eqref{matching-out}. Since $\zeta_1=\zeta_2=\frac12$, we have
\begin{equation}\label{est-g4-5}
\begin{aligned}
&\quad|[\eta_1 u_{in}+(1-\eta_{o1})\eta_{o2}u_{out}]^5-\eta_1 u_{in}^5-(1-\eta_{o1})\eta_{o2}u_{out}^5|\\
&\lesssim |u_{out}|^5 \chi_{\{|x|\sim \sqrt{T-t}\}}\lesssim (T-t)^{5(k-\frac12)}.
\end{aligned}
\end{equation}
We choose initial condition of $\Phi_1$ such that $\Phi_1(0,T)=0$. Then by Duhamel's formula, we can show that
$$
|\Phi_1(x,t)|\lesssim \alpha(t)(T-t)^{\frac32}.
$$
Therefore, we obtain
\begin{equation}\label{est-g4-6}
\begin{aligned}
|[(U_1+\Phi_1)^5-U_1^5](1-\eta_R)|\lesssim&~ (1-\eta_R)|U_1^4 \Phi_1|\\
\lesssim&~(1-\eta_R)\frac{\mu_0^{-2}(t)}{1+|y|^4}\alpha(t)(T-t)^{\frac32}.
\end{aligned}
\end{equation}
Collecting estimates \eqref{est-g4-1}--\eqref{est-g4-6}, we conclude that
\begin{equation}\label{est-g4}
\|g_4\|_{**}\lesssim T^{\epsilon}\left(1+\|\Lambda\|_{\infty}\right)
\end{equation}
provided
\begin{equation}\label{est-g4-cond}
\begin{cases}
a_2-\frac12-\frac{1}{2k}+\beta(3-a_2)-\nu_2>0\\
\frac12-\frac{1}{2k}-\nu_2>0\\
\frac12+\frac{a_2-3}{4k}-\nu_2>0\\
%3-\nu-\frac{1}{2k}-\beta(2+a)>0\\
5(\frac12-\frac1{4k})-\nu+\frac52-\beta(2+a)>0\\
2-\nu-\frac{1}{4k}-\beta(2+a)>0\\
a_2+\frac{1}{4k}-\frac32+\beta(4-a_2)-\nu_2>0\\
\end{cases}
\end{equation}

\medskip

In conclusion, from \eqref{est-g1}, \eqref{est-g1-cond}, \eqref{est-g2}, \eqref{est-g2-cond}, \eqref{est-g3}, \eqref{est-g3-cond}, \eqref{est-g4} and \eqref{est-g4-cond}, we obtain that for some $\epsilon>0$
\begin{equation}\label{est-mathcalG}
\|\mathcal G\|_{**}\lesssim T^{\epsilon}\left(1+\|\psi\|_*+\|\phi\|_{0,\nu,\sigma}+\|\Lambda\|_{\infty}\right)
\end{equation}
provided
\begin{equation}\label{est-mathcalG-cond}
\begin{cases}
1+a-\sigma<0\\
\nu-\nu_2-\frac12+a_2-\beta(a_2-a-4)>0\\
\nu-\beta(2-a)>0\\
\nu-\beta(\frac{14-2\sigma}{3}+a)>0\\
2\nu-\frac52+a_2-\beta(a_2-2a-3)-\nu_2>0\\
a_2-\frac12-\frac{1}{2k}+\beta(3-a_2)-\nu_2>0\\
\frac12-\frac{1}{2k}-\nu_2>0\\
\frac12+\frac{a_2-3}{4k}-\nu_2>0\\
5(\frac12-\frac1{4k})-\nu+\frac52-\beta(2+a)>0\\
2-\nu-\frac{1}{4k}-\beta(2+a)>0\\
a_2+\frac{1}{4k}-\frac32+\beta(4-a_2)-\nu_2>0\\
\end{cases}
\end{equation}

\medskip

\subsection{The inner problem: estimate of $\mathcal H$}

\medskip

Recall from \eqref{def-mathcalH} that
\begin{equation*}
\begin{aligned}
\mathcal H(\phi,\psi,\alpha)=
&~ 5\big[ (u_{in}(\mu_0 y, t) + \Phi_1(\mu_0 y, t))^4
  - \mu_0^{-2}w^4(y)\big] \mu_0^2 \phi(y,t) \\
& ~+ 5 \mu_0^{\frac 1 2}(1 + \Lambda)^{-4} w^4(\frac y{(1+\Lambda)^2})
    \psi(\mu_0 y,t)
  + \mu_0 \mu_0'
    \big[ \frac 1 2 \phi(y, t) + \nn_y \phi(y,t) \cdot y \big]\\
&~ + \mu_0^{\frac52}\left[S(U_1) - \chi(\frac{\mu_0 y}{c_0(T-t)^{\frac12}})
\frac {\alpha(t)} { (\mu^2 + |\mu_0 y|^2 )^{\frac 1 2}  }\right]\\
& ~+ \mu_0^{\frac 5 2}\big[ (u_{in}(\mu_0 y, t) + \Phi_1(\mu_0 y, t))^5
  - u_{in}^5(\mu_0 y, t) \big].
\end{aligned}
\end{equation*}
We evaluate
\begin{equation}\label{est-mathcalH1}
\begin{aligned}
&\quad\left|5\big[ (u_{in}(\mu_0 y, t) + \Phi_1(\mu_0 y, t))^4
  - \mu_0^{-2}w^4(y)\big] \mu_0^2 \phi(y,t)\right|\\
&\lesssim \frac{\mu_0'\mu_0^{\frac{3}{4k}}}{1+|y|^3}\frac{\mu_0^{\nu}R^{\frac{4-\sigma}3}}{1+|y|}\|\phi\|_{0,\nu,\sigma},
\end{aligned}
\end{equation}
\begin{equation}\label{est-mathcalH2}
\begin{aligned}
&\quad\left|5 \mu_0^{\frac 1 2}(1 + \Lambda)^{-4} w^4(\frac y{(1+\Lambda)^2})
    \psi(\mu_0 y,t)
  + \mu_0 \mu_0'
    \big[ \frac 1 2 \phi(y, t) + \nn_y \phi(y,t) \cdot y \big]\right|\\
&\lesssim \frac{\mu_0^{\nu}R^{-a}}{1+|y|^4}\|\psi\|_*+ \frac{\mu_0 \mu_0'\mu_0^{\nu}R^{\frac{4-\sigma}3}}{1+|y|}\|\phi\|_{0,\nu,\sigma}.
\end{aligned}
\end{equation}
By \eqref{est-g4main}, we have
\begin{equation}\label{est-mathcalH3}
\begin{aligned}
&\quad\left|\mu_0^{\frac52}\left[S(U_1) - \chi(\frac{\mu_0 y}{c_0(T-t)^{\frac12}})
\frac {\alpha(t)} { (\mu^2 + |\mu_0 y|^2 )^{\frac 1 2}  }\right]\right|\\
&\lesssim \frac{\mu_0\mu_0'\|\Lambda\|_{\infty}}{1+|y|^3}+\mu_0^{4-\frac1k}(1+|y|)+\frac{(\mu_0')^2\mu_0^2}{1+|y|}
\end{aligned}
\end{equation}
and
\begin{equation}\label{est-mathcalH4}
\begin{aligned}
\left|\mu_0^{\frac 5 2}\big[ (u_{in}(\mu_0 y, t) + \Phi_1(\mu_0 y, t))^5
  - u_{in}^5(\mu_0 y, t) \big]\right|\lesssim\frac{\mu_0'\mu_0^{\frac{3}{4k}}}{1+|y|^4}\|\Lambda\|_{\infty}.
\end{aligned}
\end{equation}
From estimates \eqref{est-mathcalH1}--\eqref{est-mathcalH4}, we obtain that for some $\epsilon>0$
\begin{equation}\label{est-mathcalH}
\|\mathcal{H}\|_{\nu,2+\sigma}\lesssim T^{\epsilon}\left(1+\|\phi\|_{0,\nu,\sigma}+\|\psi\|_*+\|\Lambda\|_{\infty}\right)
\end{equation}
provided
\begin{equation}\label{est-mathcalH-cond}
\begin{cases}
1+\frac{1}{4k}-\frac{\beta(4-\sigma)}3>0\\
0<\sigma<2\\
a>0\\
2-\frac{1}{2k}-\frac{\beta(2\sigma+7)}{3}>0\\
2-\frac1{2k}-\beta(\sigma-1)>0\\
4-\frac1k-\beta(3+\sigma)>0\\
1+\frac1{4k}-\nu>0\\
\end{cases}
\end{equation}

\medskip

%%%%%%%%%%%%%%%%%%%%%%%%%%%%%%%%%%%%%%%%%%%%%%%%%%%%%%%%%%%%

\subsection{The fixed point formulation}

\medskip

The inner--outer gluing system \eqref{eqn-outer} and \eqref{eqn-inner} can be formulated as a fixed point problem for operators we shall describe below.

We first define the following function spaces
\begin{equation}\label{def-fcnspaces}
\begin{aligned}
&\mathcal X_{\phi}:=\left\{\phi\in L^{\infty}(B_{2R}\times (0,T))\cap C^{2k+2\rho,k+\rho}_{y,t}(B_{2R}\times (0,T-\delta)):~\|\phi\|_{0,\nu,\sigma}<+\infty\right\}\\
&\mathcal X_{\psi}:=\big\{\psi\in L^{\infty}(\R^3\times (0,T))\cap C^{2k+2\rho,k+\rho}_{x,t}(\R^3\times (0,T-\delta)):~\|\psi\|_{*}<+\infty\big\}\\
&\mathcal X_{\Lambda}:=\Big\{ \Lambda(t) :
\| \Lambda(t) (T-t)^{-1 + \varepsilon} \|_{L^{\infty}} \le C_{0 \Lambda},\
\| \Lambda^{(1)} (t)(T-t)^{\varepsilon} \|_{L^{\infty} } \le C_{1\Lambda}, \
\cdots ,\\
&\qquad \qquad
\| \Lambda^{(k)} (t) (T-t)^{k-1+\varepsilon} \|_{L^{\infty}} \le C_{k \Lambda}, \Lambda(t) \in C^{k+1,\rho}(0,T-\delta), \foral \delta \in (0,T)
  \Big\}\\
&\mathcal X_{\vec c}:=\left\{ \vec{c}=(c_1, c_2,\dots, c_k) :
|c_j| \le C_c T^{\frac 1 2 - j - \varepsilon},\ j=1,2,\dots,  k \right\}\\
\end{aligned}
\end{equation}

Define
\begin{equation}\label{def-mX}
\mathcal X=\mathcal X_{\phi}\times \mathcal X_{\psi}\times \mathcal X_{\Lambda} \times \mathcal X_{\vec c}.
\end{equation}
We shall solve the inner--outer gluing system in a closed ball $\mathcal B$ in $(\phi,\psi,\Lambda,\vec c)\in\mathcal X$.

The inner--outer gluing system \eqref{eqn-outer} and \eqref{eqn-inner} can be formulated as a fixed point problem, where we define an operator $\mathcal F$ which returns the solution from $\mathcal B$ to $\mathcal X$
\begin{equation*}
\begin{aligned}
\mathcal F: \mathcal B\subset \mathcal X~\rightarrow& ~\mathcal X\\
v~\mapsto&~ \mathcal F(v)=(\mathcal F_{\phi}(v),\mathcal F_{\psi}(v),\mathcal F_{\Lambda}(v),\mathcal F_{\vec c}(v))
\end{aligned}
\end{equation*}
with
\begin{equation}\label{def-operators}
\begin{aligned}
\mathcal F_{\phi}(\phi,\psi,\Lambda,\vec c)=&~\mathcal T_{\phi}(\mathcal H[\phi,\psi,\Lambda])\\
\mathcal F_{\psi}(\phi,\psi,\Lambda,\vec c)=&~\mathcal T_{\psi}\left(\mathcal G(\phi,\psi,\Lambda)\right)\\
\mathcal F_{\Lambda}(\phi,\psi,\Lambda,\vec c)=&~\mathcal T_{\Lambda}(\phi,\psi,\Lambda,\vec c)\\
\mathcal F_{\vec c}(\phi,\psi,\Lambda,\vec c)=&~\mathcal T_{\vec c}\left(\phi,\psi,\Lambda,\vec c\right)\\
\end{aligned}
\end{equation}
Here $\mathcal T_{\phi}$ is the operators given from Proposition \ref{propmode} which solves the inner problem \eqref{eqn-inner}. The operator $\mathcal T_{\psi}$ defined by Proposition \ref{outer-apriori} deals with the outer problem \eqref{eqn-outer}. Operators $\mathcal T_{\Lambda}$ and $\mathcal T_{\vec c}$ handle the reduced equation \eqref{orthogonality}.

\medskip

%%%%%%%%%%%%%%%%%%%%%%%%%%%%%%%%%%%%%%%%%%%%%%%%%%%%%%%%%%%%

\medskip

\subsection{Choice of constants}\label{subsec-choices}

In this section, we list all the constraints of the parameters
%$$\beta,\alpha,a,a_1,\nu,\nu_1,\nu_2,\sigma$$
which are sufficient for the inner--outer gluing scheme to work.

We first indicate all the parameters used in different norms.
\begin{itemize}
\item $R(t)=\mu_0^{-\beta}(t)$ with $\beta\in(0,1/2)$.

\medskip

\item The norm for $\phi$ solving the inner problem \eqref{eqn-inner} is $\|\cdot\|_{0,\nu,\sigma}$ which is defined in \eqref{def-normm0}, where we require that $\nu>0$, $0<\sigma<2$.

\medskip

\item The norm for $\psi$ solving the outer problem \eqref{eqn-outer} is $\|\cdot\|_*$ which is defined in \eqref{def-norm*}, while the $\|\cdot\|_{**}$-norm for the right hand side of the outer problem \eqref{eqn-outer} is defined in \eqref{def-norm**}. Here we require that $\nu,\nu_2>0$ and $a,\gamma\in(0,1)$. Also, as mentioned in Remark \ref{remark5.2}, we require $\nu_2+\frac{2-a_2}{4k}>\nu-\frac12+a\beta$ such that the $\|\cdot\|_*$-norm is well defined.
\end{itemize}

In order to get the desired estimates for the outer problem \eqref{eqn-outer}, by \eqref{est-mathcalG-cond}, we need the following restrictions
\begin{equation*}
\begin{cases}
1+a-\sigma<0\\
\nu-\nu_2-\frac12+a_2-\beta(a_2-a-4)>0\\
\nu-\beta(2-a)>0\\
\nu-\beta(\frac{14-2\sigma}{3}+a)>0\\
2\nu-\frac52+a_2-\beta(a_2-2a-3)-\nu_2>0\\
a_2-\frac12-\frac{1}{2k}+\beta(3-a_2)-\nu_2>0\\
\frac12-\frac{1}{2k}-\nu_2>0\\
\frac12+\frac{a_2-3}{4k}-\nu_2>0\\
5(\frac12-\frac1{4k})-\nu+\frac52-\beta(2+a)>0\\
2-\nu-\frac{1}{4k}-\beta(2+a)>0\\
a_2+\frac{1}{4k}-\frac32+\beta(4-a_2)-\nu_2>0\\
\end{cases}
\end{equation*}
In order to get the desired estimates for the inner problem \eqref{eqn-inner}, by \eqref{est-mathcalH-cond}, we need
\begin{equation*}
\begin{cases}
1+\frac{1}{4k}-\frac{\beta(4-\sigma)}3>0\\
0<\sigma<2\\
a>0\\
2-\frac{1}{2k}-\frac{\beta(2\sigma+7)}{3}>0\\
2-\frac1{2k}-\beta(\sigma-1)>0\\
4-\frac1k-\beta(3+\sigma)>0\\
1+\frac1{4k}-\nu>0\\
\end{cases}
\end{equation*}
Elementary computations show that suitable choices of the parameters satisfying all the restrictions in this section can be found, which ensures the implementation of the gluing procedure.

\medskip

\subsection{Proof of Theorem \ref{teo1}}

\medskip

Consider the operator
\begin{equation}\label{def-mF}
\mathcal F=(\mathcal F_{\phi},\mathcal F_{\psi},\mathcal F_{\Lambda},\mathcal F_{\vec c})
\end{equation}
given in \eqref{def-operators}. To prove Theorem \ref{teo1}, our strategy is to show the existence of a fixed point for the operator $\mathcal F$ in $\mathcal B$ by the Schauder fixed point theorem. By collecting the estimates \eqref{est-mathcalG}, \eqref{est-mathcalH}, and using Proposition \ref{outer-apriori}, Proposition \ref{propmode} and discussions in Section \ref{sec-redu}, we conclude that for $(\phi,\psi,\Lambda,\vec c)\in\mathcal B$
\begin{equation}\label{contraction-mF}
\begin{cases}
\|\mathcal F_{\phi}(\phi,\psi,\Lambda,\vec c)\|_{0,\nu,\sigma}\leq CT^{\epsilon}\\
\|\mathcal F_{\psi}(\phi,\psi,\Lambda,\vec c)\|_{*}\leq CT^{\epsilon}\\
\|\mathcal F_{\Lambda}(\phi,\psi,\Lambda,\vec c)\|_{\Lambda}\leq CT^{\epsilon}\\
\|\mathcal F_{\vec c}(\phi,\psi,\Lambda,\vec c)\|_{\vec c}\leq CT^{\epsilon}\\
\end{cases}
\end{equation}
where $C>0$ is a constant independent of $T$, and $\epsilon>0$ is a small fixed number. On the other hand, compactness of the operator $\mathcal F$ defined in \eqref{def-mF} can be proved by proper variants of \eqref{contraction-mF}. Indeed, if we vary the parameters slightly such that all the restrictions in Section \ref{subsec-choices} are still satisfied, then we get \eqref{contraction-mF} with the norms in the left hand side defined by the new parameters, while the closed ball $\mathcal B$ remains the same. To be more specific, for fixed $\nu',a'$ which are close to $\nu,a$, one can show that if $(\phi,\psi,\Lambda)\in\mathcal B$, then
$$\|\mathcal F_{\phi}(\phi,\psi,\Lambda,\vec c)\|_{0,\nu',\sigma'}\leq CT^{\epsilon'}.$$
Furthermore, one can show that for $\nu'>\nu$ and $\nu'-\frac{\beta(4-\sigma')}3>\nu-\frac{\beta(4-\sigma)}3$, one has a compact embedding in the sense that if a sequence $\{\phi^0_n\}$ is bounded in the $\|\cdot\|_{0,\nu',\sigma'}$-norm, then there exists a subsequence which converges in the $\|\cdot\|_{0,\nu,\sigma}$-norm. Thus, the compactness follows directly from a standard diagonal argument by Arzel\`a--Ascoli's theorem. Arguing in a similar manner, the compactness for the rest operators can be proved. Therefore, the existence of the desired blow-up solution is concluded from the Schauder fixed point theorem.

%%%%%%%%%%%%%%%%%%%%%%%%%%%%%%%%%%%%%%%%%%%%%%%%%%%%%%

%The coupling
%$$
%\left|\mu_0^{-\frac12} (\phi \Delta_x \eta_2+2\nabla \phi\cdot \nabla \eta_2)\right| \sim \mu_0^{\nu-\frac52} R^{-2-a}\chi_{\{|x|\sim \mu_0 R\}}
%$$
%Then Duhamel gives
%\begin{equation}
%\begin{aligned}
%|\psi|\lesssim&~ \int_0^t \int_{|z|\sim \mu_0(s)R(s)} \frac{e^{-\frac{|x-z|^2}{4(t-s)}}}{(t-s)^{3/2}}\mu_0^{\nu-\frac52}(s) R^{-2-a}(s) dz ds\\
%\lesssim&~ \int_0^t \int_{|\tilde z|\sim \frac{\mu_0(s)R(s)}{(t-s)^{1/2}}} e^{-\frac{|\tilde x-\tilde z|^2}{4}}\mu_0^{\nu-\frac52}(s) R^{-2-a}(s) d\tilde z ds\\
%=&~\left(\int_0^{t-\delta(t)}+\int_{t-\delta(t)}^t\right)\int_{|\tilde z|\sim \frac{\mu_0(s)R(s)}{(t-s)^{1/2}}} e^{-\frac{|\tilde x-\tilde z|^2}{4}}\mu_0^{\nu-\frac52}(s) R^{-2-a}(s) d\tilde z ds\\
%\lesssim &~ \int_0^{t-\delta(t)}\mu_0^{\nu-\frac52}(s) R^{-2-a}(s) \frac{\mu_0^3(s)R^3(s)}{(t-s)^{3/2}} ds + \int_{t-\delta(t)}^t \mu_0^{\nu-\frac52}(s) R^{-2-a}(s) ds\\
%\lesssim &~ \mu_0^{\nu+\frac12}(0) R^{1-a}(0) \delta^{-1/2}(t)+\delta(t)\mu_0^{\nu-\frac52}(t) R^{-2-a}(t)
%\end{aligned}
%\end{equation}

%%%%%%%%%%%%%%%%%%%%%%%%%%%%%%%%%%%%%%%%%%%%%%%%%%%%%%%%%%

\medskip

%%%%%%%%%%%%%%%%%%%%%%%%%%%%%%%%%%%%%%%%%%%%%%%%%%%%%%%%%%

\medskip

\section{Nonradial case: blow-up at multiple points}

\medskip

As a by-product, the inner--outer gluing method carried out in this paper can be applied to construct {\em non-radial} type II blow-up at {\em multiple $N$ points} for the first blow-up rate $k=1$. To be more precise, we take the first approximation to be
$$
U_N=\sum_{j=1}^N \mu_j^{-\frac12}(t) w\left(\frac{x-\xi_j(t)}{\mu_j(t)}\right),
$$
where we expect that the scaling and translation parameters satisfy
$$\mu_j(t)\to 0,\quad \xi_j(t)\to q_j \quad \mbox{ as }~ t\to T$$
for $j=1,\dots,N$, where $q_j$ are given points in $\R^3$ with $\max\limits_{k,l=1,\dots,N}|q_k-q_l|>\delta$ for uniform $\delta>0$. Formally, the error of $U_N$ behaves like
\begin{equation*}
\begin{aligned}
S(U_N)\sim &~\sum_{j=1}^N \left(\mu_j^{-\frac32}(t)\dot\mu_j(t)Z_0(y_j)+\mu_j^{-\frac{3}{2}}(t)\dot\xi_j\cdot\nabla w(y_j)\right)\\
:=&~ \sum_{j=1}^N \left(\mathcal E_{0,j}+\mathcal E_{1,j} \right)
\end{aligned}
\end{equation*}
where $y_j=\frac{x-\xi_j(t)}{\mu_j(t)}$. To cancel out the slow decay error at mode 0 near each point $q_j$, we introduce the correction $\Phi_{(j)}$ solving
$$\partial_t \Phi_{(j)}=\Delta \Phi_{(j)} +\mathcal E_{0,j}\quad \mbox{for }~j=1\dots,N$$
so that the corrected approximation is
$$
u_*=\sum_{j=1}^N \mu_j^{-\frac12}(t) w\left(\frac{x-\xi_j(t)}{\mu_j(t)}\right)+\Phi_{(j)}.
$$
We then look for the solution
$$
u=u_*+\sum_{j=1}^N \mu_j^{-\frac12}(t)\eta_{R_j}\phi(y_j,t)+\psi(x,t).
$$
Let us emphasize that in the non-radial case, the blow-up rate for $k=1$ will be obtained by {\em orthogonality condition} instead of the matching in the general case $k\geq 2$. The orthogonality condition at scaling mode 0 is basically
$$
\int_{B_{2R}} 5w^4 \Phi_{(j)} Z_0(y) dy+ \int_{B_{2R}} 5w^4 \psi Z_0(y) dy\approx 0
$$
which turns out to be a nonlocal equation like before, and using the method in Section \ref{sec-redu}, we have
$$
\mu_j(t)\sim (T-t)^2 ~\mbox{  for  }~j=1,\cdots,N.
$$
Indeed, the orthogonality condition at mode $0$ gives a nonlocal reduced equation of the following form:
$$
\int_{0}^{t}\frac{\mu^{-1/2}_j(s)\dot\mu_j(s)}{(t-s)^{1/2}} ds=c_{*,j},
$$
where $c_{*,j}<0$ is some constant coming from the initial data. We rewrite the above integro-differential equation as
$$
\int_{0}^{t}\frac{\dot v_j(s)}{(t-s)^{1/2}} ds=c_{*,j},
$$
where $v_j(t)=2\mu_j^{1/2}(t)$. Imposing $v_j(T)=0$ and using \eqref{lemma-RL}, we obtain that for some $c>0$
$$
v_j(t)-v_j(T)=v_j(t)=c(T^{1/2}-t^{1/2})\sim T-t
$$
and thus
$$
\mu_j(t)\sim (T-t)^2
$$
which is precisely the first rate ($k=1$) predicted in \cite{fhv}.

On the other hand, the orthogonality condition at translation mode 1
$$
\int_{B_{2R}} \mathcal E_{1,j} Z_{\ell}(y) dy\approx 0
$$
simply implies
$$
\xi_j(t)\sim q_j ~\mbox{  for  }~j=1,\cdots,N,
$$
where
$$Z_{\ell}(y)=\partial_{y_{\ell}}w \quad \ell=1,2,3.$$
We will not elaborate on the details.

%%%%%%%%%%%%%%%%%%%%%%%%%%%%%%%%%%%%%%%%%%%%%%%%%%%%%%%%%%%%

\medskip

\appendix
\section{Proofs of technical Lemmas}

\medskip

In this appendix, we prove the technical lemmas in Section \ref{sec-outer}.

\begin{proof}[Proof of lemma \ref{lemma-rhs1}]
Duhamel's formula gives
\begin{equation*}
\begin{aligned}
|\psi|\lesssim~ \int_0^t \int_{|w|\leq2\mu_0(s)R(s)} \frac{e^{-\frac{|x-w|^2}{4(t-s)}}}{(4\pi(t-s))^{3/2}}\mu_0^{\nu-\frac52}(s) R^{-2-a}(s) dw ds\\
\end{aligned}
\end{equation*}
We decompose
\begin{equation*}
\begin{aligned}
&\quad \int_0^t \mu_0 ^{\nu-\frac52}(s) R^{-2-a}(s) \int_{\left|w\right|\leq 2\mu_0 (s)R(s)} \frac{e^{-\frac{|x-w|^2}{4(t-s)}}}{(4\pi(t-s))^{3/2}} dw ds\\
&=\left(\int_0^{t-(T-t)}+\int_{t-(T-t)}^{t-\mu^{\delta_1}_0(t)}+\int_{t-\mu^{\delta_1}_0(t)}^t\right) \mu_0 ^{\nu-\frac52}(s) R^{-2-a}(s) \int_{\left|w\right|\leq 2\mu_0 (s)R(s)} \frac{e^{-\frac{|x-w|^2}{4(t-s)}}}{(4\pi(t-s))^{3/2}} dw ds\\
&:=I_{11}+I_{12}+I_{13}
\end{aligned}
\end{equation*}
for some $\delta_1\geq 1$ to be found. Directly integrating, we obtain
\begin{equation*}
\begin{aligned}
I_{11}=&~ \int_0^{t-(T-t)}\mu_0 ^{\nu-\frac52}(s) R^{-2-a}(s) \int_{\left|w\right|\leq 2\mu_0 (s)R(s)} \frac{e^{-\frac{|x-w|^2}{4(t-s)}}}{(4\pi(t-s))^{3/2}} dw ds\\
\lesssim&~\int_0^{t-(T-t)}\mu_0 ^{\nu-\frac52}(s) R^{-2-a}(s) \int_{\left|\tilde w\right|\leq \frac{2\mu_0 (s)R(s)}{\sqrt{t-s}}} e^{-\frac{|\tilde x-\tilde w|^2}{4}} d\tilde w ds\\
\lesssim&~\int_0^{t-(T-t)} \mu_0 ^{\nu-\frac52}(s) R^{-2-a}(s)\frac{\left(\mu_0 (s)R(s)\right)^3}{(t-s)^{3/2}} ds\\
\lesssim&~ \int_0^{t-(T-t)} \frac{\mu_0 ^{\nu+\frac12}(s) R^{1-a}(s)}{(T-s)^{3/2}} ~ds\\
\lesssim&~ \mu_0 ^{\nu+\frac12-\frac{1}{4k}}(0)R^{1-a}(0),
\end{aligned}
\end{equation*}
Similarly we compute

\begin{equation}\label{est-I12}
\begin{aligned}
I_{12}\lesssim&~ \int_{t-(T-t)}^{t-\mu^{\delta_1}_0(t)} \mu_0 ^{\nu-\frac52}(s) R^{-2-a}(s)\frac{\left(\mu_0 (s)R(s)\right)^3}{(t-s)^{3/2}} ds\\
\lesssim&~ \mu_0 ^{\nu+\frac{1-\delta_1}{2}}(t)R^{1-a}(t)
\end{aligned}
\end{equation}
and
\begin{equation}\label{est-I13}
\begin{aligned}
I_{13}\lesssim&~ \int_{t-\mu^{\delta_1}_0(t)}^t \mu_0 ^{\nu-\frac52}(s) R^{-2-a}(s) ds\\
\lesssim&~ \mu_0 ^{\nu-\frac52+\delta_1}(t) R^{-2-a}(t).
\end{aligned}
\end{equation}
Since $\beta\in(0,1/2)$, we can choose $\delta_1=2-2\beta$. Therefore, we get
$$I_{11}+I_{12}+I_{13}\lesssim \mu_0 ^{\nu-\frac12}(0) R^{-a}(0)$$
as desired.

Similarly, to prove \eqref{outerT-rhs1}, we decompose
$$|\psi(x,t)-\psi(x,T)|\leq I_{21}+I_{22}+I_{23}$$
with
$$I_{21}=\int_0^{t-(T-t)} \int_{\R^3}|G(x-w,t-s)-G(x-w,T-s)||f(w,s)|dwds$$
$$I_{22}=\int_{t-(T-t)}^t \int_{\R^3}|G(x-w,t-s)-G(x-w,T-s)||f(w,s)|dwds$$
$$I_{23}=\int_t^T \int_{\R^3}|G(x-w,T-s)||f(w,s)|dwds,$$
where $G(x,t)$ is the heat kernel
\begin{equation}\label{def-heatkernel}
G(x,t)=\frac{e^{-\frac{|x|^2}{4t}}}{(4\pi t)^{3/2}}.
\end{equation}
For the first integral $I_{21}$, we have
$$I_{21}\leq (T-t)\int_0^1\int_0^{t-(T-t)}\int_{|w|\leq 2\mu_0 (s)R(s)} |\partial_t G(x-w,t_v-s)|\mu_0 ^{\nu-\frac52}(s)R^{-2-a}(s) dwdsdv,$$
where $t_v=vT+(1-v)t$. Changing variables, we evaluate
\begin{equation*}
\begin{aligned}
&\quad \int_{|w|\leq 2\mu_0 (s)R(s)} |\partial_t G(x-w,t_v-s)| dw\\
&\lesssim \int_{|w|\leq 2\mu_0 (s)R(s)} e^{-\frac{|x-w|^2}{4(t_v-s)}}\left(\frac{|x-w|^2}{(t_v-s)^{\frac{7}{2}}}+\frac{1}{(t_v-s)^{\frac{5}{2}}}\right)dw\\
&=\int_{\left|w_v\right|\leq \frac{2 \mu_0 (s)R(s)}{\sqrt{t_v-s}}} e^{-\frac{|x_v-w_v|^2}{4}}\left(1+|x_v-w_v|^2\right)\frac{1}{t_v-s}dw_v\\
\end{aligned}
\end{equation*}
and thus
\begin{equation*}
\begin{aligned}
&\quad\int_0^{t-(T-t)}\int_{|w|\leq 2\mu_0 (s)R(s)} |\partial_t G(x-w,t_v-s)|\mu_0 ^{\nu-\frac52}(s)R^{-2-a}(s) dwds\\
&\lesssim \int_0^{t-(T-t)} \mu_0 ^{\nu-\frac52}(s)R^{-2-a}(s) \frac{(\mu_0 (s)R(s))^3}{(t_v-s)^{5/2}}ds\\
&\lesssim \int_0^{t-(T-t)} \mu_0 ^{\nu+\frac12}(s)R^{1-a}(s) (T-s)^{-5/2} ds\\
&\lesssim \mu_0 ^{\nu+\frac12-\frac{3}{4k}}(t)R^{1-a}(t),
\end{aligned}
\end{equation*}
from which we conclude that
\begin{equation}\label{est-I21}
I_{21}\lesssim \mu_0 ^{\nu+\frac12-\frac{1}{4k}}(t)R^{1-a}(t).
\end{equation}
For $I_{22}$, we have
\begin{equation*}
\begin{aligned}
I_{22}\leq&~ \int_{t-(T-t)}^t\int_{|w|\leq 2\mu_0 (s)R(s)} |G(x-w,t-s)|\mu_0 ^{\nu-\frac52}(s)R^{-2-a}(s)dwds\\
&~+ \int_{t-(T-t)}^t\int_{|w|\leq 2\mu_0 (s)R(s)} |G(x-w,T-s)|\mu_0 ^{\nu-\frac52}(s)R^{-2-a}(s)dwds.
\end{aligned}
\end{equation*}
The first integral above can be estimated as
\begin{equation*}
\begin{aligned}
&\quad \int_{t-(T-t)}^t \int_{|w|\leq 2\mu_0 (s)R(s)} |G(x-w,t-s)|\mu_0 ^{\nu-\frac52}(s)R^{-2-a}(s)dwds\\
&=\left(\int_{t-(T-t)}^{t-\mu_0 ^{\delta_1}(t)}+\int_{t-\mu_0 ^{\delta_1}(t)}^t\right) |G(x-w,t-s)|\mu_0 ^{\nu-\frac52}(s)R^{-2-a}(s)dwds.\\
\end{aligned}
\end{equation*}
Notice that we already estimate the above integral in \eqref{est-I12} and \eqref{est-I13}. So the choice $\delta_1=2-2\beta$, one has
\begin{equation*}
\begin{aligned}
&\quad \int_{t-(T-t)}^t \int_{|w|\leq 2\mu_0 (s)R(s)} |G(x-w,t-s)|\mu_0 ^{\nu-\frac52}(s)R^{-2-a}(s)dwds\\
&\lesssim  \mu_0 ^{\nu-\frac12}(t) R^{-a}(t).
\end{aligned}
\end{equation*}
Similarly, it holds that
\begin{equation*}
\begin{aligned}
&\quad \int_{t-(T-t)}^t \int_{|(w_r,w_z)-\xi(s)|\leq 2\mu_0 (s)R(s)} |G(x-w,T-s)|\mu_0 ^{\nu-3}(s)R^{-2-\alpha}(s)dwds\\
&\lesssim  \mu_0 ^{\nu-\frac12}(t) R^{-a}(t).
\end{aligned}
\end{equation*}
Therefore, we obtain
\begin{equation}\label{est-I22}
I_{22}\lesssim \mu_0 ^{\nu-\frac12}(t) R^{-a}(t).
\end{equation}
For $I_{23}$, changing variables, one has
\begin{equation}\label{est-I23}
\begin{aligned}
I_{23}\lesssim&~ \int_t^T \int_{\left|w\right|\leq 2\mu_0 (s)R(s)}\frac{e^{-\frac{|x-w|^2}{4(T-s)}}}{(T-s)^{3/2}} \mu_0 ^{\nu-\frac52}(s)R^{-2-a}(s)dwds\\
\lesssim&~ \int_t^T \int_{\left|\tilde w\right|\leq \frac{2\mu_0 (s)R(s)}{\sqrt{T-s}}} e^{-\frac{|\tilde x-\tilde w|^2}{4}} \mu_0 ^{\nu-\frac52}(s)R^{-2-a}(s) d\tilde w ds\\
\lesssim&~ \int_t^T  \mu_0 ^{\nu-\frac52}(s)R^{-2-a}(s)\frac{(\mu_0 (s)R(s))^3}{(T-s)^{3/2}} ds\\
\lesssim&~ \mu_0 ^{\nu+\frac12-\frac{1}{2k}}(t)R^{1-a}(t)
\end{aligned}
\end{equation}
%provided $\nu-\beta(2-\alpha)>0$.
Collecting \eqref{est-I21}, \eqref{est-I22} and \eqref{est-I23}, we conclude the validity of \eqref{outerT-rhs1}.

Then we prove the gradient estimate \eqref{outergradient-rhs1}. By the heat kernel, we get
\begin{equation*}
\begin{aligned}
|\nabla \psi(x,t)|\lesssim&~ \int_0^t \frac{\mu_0 ^{\nu-\frac52}(s)R^{-2-a}(s)}{(t-s)^{\frac{5}{2}}}\int_{\left|w\right|\leq2 \mu_0 (s)R(s)} e^{-\frac{|x-w|^2}{4(t-s)}}|x-w| dwds\\
%\lesssim&~ \int_0^t \frac{\mu_0 ^{\nu-3}(s)R^{-2-\alpha}(s)}{(t-s)^{1/2}}\int_{\left|(\tilde w_r,\tilde w_z)-\frac{\xi(s)}{\sqrt{t-s}}\right|\leq \frac{2 \mu_0 (s)R(s)}{\sqrt{t-s}}} e^{-\frac{|\tilde x-\tilde w|^2}{4}}|\tilde x-\tilde w| d\tilde wds\\
\lesssim&~ \int_0^t \frac{\mu_0 ^{\nu-\frac52}(s)R^{-2-a}(s)}{(t-s)^{1/2}}\int_{\left|\tilde w\right|\leq \frac{2 \mu_0 (s)R(s)}{\sqrt{t-s}}}  e^{-\frac{|\tilde w|^2}{4}}(1+|\tilde w|) d\tilde wds,\\
\end{aligned}
\end{equation*}
where $\tilde x=x(t-s)^{-1/2}$.
First, we compute
\begin{equation}\label{grad-rhs11}
\begin{aligned}
&\quad\int_0^{t-(T-t)} \frac{\mu_0 ^{\nu-\frac52}(s)R^{-2-a}(s)}{(t-s)^{1/2}}\int_{\left|w\right|\leq \frac{2 \mu_0 (s)R(s)}{\sqrt{t-s}}} e^{-\frac{|\tilde w|^2}{4}}(1+|\tilde w|) d\tilde wds\\
&\lesssim \int_0^{t-(T-t)} \frac{\mu_0 ^{\nu-\frac52}(s)R^{-2-a}(s)}{(t-s)^{1/2}}\frac{\left(\mu_0 (s)R(s)\right)^3}{(t-s)^{3/2}} ds\\
&\lesssim \int_0^{t-(T-t)} \frac{\mu_0 ^{\nu+\frac12}(s)R^{1-a}(s)}{(t-s)^{2}} ds\\
&\lesssim \int_0^{t-(T-t)} \mu_0 ^{\nu+\frac12}(s)R^{1-a}(s)(T-s)^{-2} ds\\
&\lesssim  \mu_0 ^{\nu+\frac12-\frac{1}{2k}}(0)R^{1-a}(0).
\end{aligned}
\end{equation}
Then we compute
\begin{equation}\label{grad-rhs12}
\begin{aligned}
&\quad\int_{t-(T-t)}^{t-\mu_0 ^{\delta_2}(t)} \frac{\mu_0 ^{\nu-\frac52}(s)R^{-2-a}(s)}{(t-s)^{1/2}}\int_{\left|w\right|\leq \frac{2 \mu_0 (s)R(s)}{\sqrt{t-s}}} e^{-\frac{|\tilde w|^2}{4}}(1+|\tilde w|) d\tilde wds\\
&\lesssim \int_{t-(T-t)}^{t-\mu_0 ^{\delta_2}(t)} \frac{\mu_0 ^{\nu+\frac12}(s)R^{1-a}(s)}{(t-s)^{2}} ds\\
&\lesssim \mu_0 ^{\nu+\frac{1}{2}-\delta_2}(t)R^{1-a}(t),
\end{aligned}
\end{equation}
where $\delta_2\geq 1$ is a constant to be determined. On the other hand, we have
\begin{equation}\label{grad-rhs13}
\begin{aligned}
&\quad\int_{t-\mu_0 ^{\delta_2}(t)}^t \frac{\mu_0 ^{\nu-\frac52}(s)R^{-2-a}(s)}{(t-s)^{\frac{5}{2}}}\int_{\left|w\right|\leq2 \mu_0 (s)R(s)} e^{-\frac{|x-w|^2}{4(t-s)}}|x-w| dwds\\
&\lesssim \int_{t-\mu_0 ^{\delta_2}(t)}^t \frac{\mu_0 ^{\nu-\frac52}(s)R^{-2-a}(s)}{(t-s)^{1/2}} ds\\
&\lesssim \mu_0 ^{\nu-\frac52+\frac{\delta_2}{2}}(t)R^{-2-a}(t).
\end{aligned}
\end{equation}
By choosing $\delta_2=2-2\beta$ and combining \eqref{grad-rhs11}--\eqref{grad-rhs13}, we prove the validity of the gradient estimate \eqref{outergradient-rhs1}. The proof of \eqref{outergradientT-rhs1} is similar to that of \eqref{outerT-rhs1}. We omit the details.

To prove the H\"older estimate \eqref{outerholder-rhs1}, we decompose
$$|\psi(x,t_2)-\psi(x,t_1)|\leq J_{11}+J_{12}+J_{13}$$
with
$$J_{11}=\int_0^{t_1-(t_2-t_1)}\int_{\R^3} |G(x-w,t_2-s)-G(x-w,t_1-s)|f(w,s)dwds,$$
$$J_{12}=\int_{t_1-(t_2-t_1)}^{t_1}\int_{\R^3} |G(x-w,t_2-s)-G(x-w,t_1-s)|f(w,s)dwds,$$
and
$$J_{13}=\int_{t_1}^{t_2}\int_{\R^3} G(x-w,t_2-s)f(w,s)dwds,$$
where $G(x,t)$ is the heat kernel \eqref{def-heatkernel}. Here we assume that $0<t_1<t_2<T$ with $t_2<2t_1$.
For $J_{11}$, by letting $t_v=vt_2+(1-v)t_1$, we have
\begin{equation*}
\begin{aligned}
J_{11}\leq &~(t_2-t_1)\int_0^1 \int_0^{t_1-(t_2-t_1)}\int_{\R^3} |\partial_t G(x-w,t_v-s)|f(w,s)dwdsdv\\
\lesssim &~ (t_2-t_1)\int_0^1 \int_0^{t_1-(t_2-t_1)}\int_{\left|w\right|\leq2 \mu_0 (s)R(s)} e^{-\frac{|x-w|^2}{4(t_v-s)}}\bigg(\frac{|x-w|^2}{(t_v-s)^{\frac{7}{2}}}\\
&\qquad\qquad\qquad\qquad\qquad\qquad+\frac{1}{(t_v-s)^{\frac{5}{2}}}\bigg) \mu_0 ^{\nu-\frac52}(s)R^{-2-a}(s) dwdsdv,\\
\end{aligned}
\end{equation*}
and
\begin{equation*}
\begin{aligned}
&~\int_{\left|w\right|\leq2 \mu_0 (s)R(s)} e^{-\frac{|x-w|^2}{4(t_v-s)}}\left(\frac{|x-w|^2}{(t_v-s)^{\frac{7}{2}}}+\frac{1}{(t_v-s)^{\frac{5}{2}}}\right)\mu_0 ^{\nu-\frac52}(s)R^{-2-a}(s) dw\\
=&~\int_{\left|w_v\right|\leq \frac{2 \mu_0 (s)R(s)}{\sqrt{t_v-s}}} e^{-\frac{|x_v-w_v|^2}{4}}\left(1+|x_v-w_v|^2\right)\frac{\mu_0 ^{\nu-\frac52}(s)R^{-2-a}(s)}{t_v-s}dw_v.\\
\end{aligned}
\end{equation*}
Observing that for any $\gamma_1\in(0,1)$, we have
$$\int_{\left|w_v\right|\leq \frac{2 \mu_0 (s)R(s)}{\sqrt{t_v-s}}} e^{-\frac{|x_v-w_v|^2}{4}}\left(1+|x_v-w_v|^2\right) dw_v \lesssim \left(\frac{2 \mu_0 (s)R(s)}{\sqrt{t_v-s}}\right)^{\gamma_1}.$$
Thus, one has
\begin{equation*}
\begin{aligned}
J_{11}\lesssim (t_2-t_1) \int_0^{t_1-(t_2-t_1)} \frac{\mu_0 ^{\nu-\frac52+\gamma_1}(s)R^{-2-a+\gamma_1}(s)}{(t_2-s)^{1+\frac{\gamma_1}{2}}}ds.
\end{aligned}
\end{equation*}
Recalling that $R(t)=\mu_0 ^{-\beta}(t)$ for $\beta\in(0,1/2)$, we have the following two cases
\begin{itemize}
\item If $\nu-\frac52+\gamma_1+\beta(2+a-\gamma_1)< 0$, then we have
\begin{equation*}
\begin{aligned}
&\quad \int_0^{t_1-(t_2-t_1)} \frac{\mu_0 ^{\nu-\frac52+\gamma_1}(s)R^{-2-a+\gamma_1}(s)}{(t_2-s)^{1+\frac{\gamma_1}{2}}}ds\\
&\lesssim \mu_0 ^{\nu-\frac52+\gamma_1}(t_1)R^{-2-a+\gamma_1}(t_1)\int_0^{t_1-(t_2-t_1)}\frac{1}{(t_2-s)^{1+\frac{\gamma_1}{2}}}ds\\
&\lesssim \mu_0 ^{\nu-\frac52+\gamma_1}(t_1)R^{-2-a+\gamma_1}(t_1)(t_2-t_1)^{-\gamma_1/2}.
\end{aligned}
\end{equation*}

\item If $\nu-\frac52+\gamma_1+\beta(2+a-\gamma_1)\geq 0$, then we decompose
\begin{equation*}
\begin{aligned}
&\quad\int_0^{t_1-(t_2-t_1)} \frac{\mu_0 ^{\nu-\frac52+\gamma_1}(s)R^{-2-a+\gamma_1}(s)}{(t_2-s)^{1+\frac{\gamma_1}{2}}}ds\\
&= \left(\int_0^{t_1-(T-t_1)}+\int_{t_1-(T-t_1)}^{t_1-(t_2-t_1)}\right) \frac{\mu_0 ^{\nu-\frac52+\gamma_1}(s)R^{-2-a+\gamma_1}(s)}{(t_2-s)^{1+\frac{\gamma_1}{2}}}ds.
\end{aligned}
\end{equation*}
Assuming
$$2k[\nu-\frac52+\gamma_1+\beta(2+a-\gamma_1)]-\frac{\gamma_1}{2}<0,$$
we obtain that
\begin{equation*}
\begin{aligned}
&\quad\int_0^{t_1-(T-t_1)} \frac{\mu_0 ^{\nu-\frac52+\gamma_1}(s)R^{-2-a+\gamma_1}(s)}{(t_2-s)^{1+\frac{\gamma_1}{2}}}ds\\
&\lesssim \int_0^{t_1-(T-t_1)} \frac{\mu_0 ^{\nu-\frac52+\gamma_1}(s)R^{-2-a+\gamma_1}(s)}{(T-s)^{1+\frac{\gamma_1}{2}}}ds\\
&=\int_0^{t_1-(T-t_1)}(T-s)^{2k[\nu-\frac52+\gamma_1+\beta(2+a-\gamma_1)]-1-\frac{\gamma_1}{2}} ds\\
%&\lesssim\frac{|\log T|^{\nu-3+\beta(2+\alpha)}(T-t_1)^{\nu-3+\beta(2+\alpha)}}{|\log(T-t_1)|^{2(\nu-3+\beta(2+\alpha))}} ds\\
%&\lesssim\mu_0 ^{\nu-3+\beta(2+\alpha)}(t_2)\\
&\lesssim\mu_0 ^{\nu-\frac{5}{2}+\gamma_1}(t_2)R^{-2-a+\gamma_1}(t_2)(t_2-t_1)^{-\gamma_1/2}
\end{aligned}
\end{equation*}
and similarly
\begin{equation*}
\begin{aligned}
\int_{t_1-(T-t_1)}^{t_1-(t_2-t_1)} \frac{\mu_0 ^{\nu-\frac52+\gamma_1}(s)R^{-2-a+\gamma_1}(s)}{(t_2-s)^{1+\frac{\gamma_1}{2}}} ds\lesssim \mu_0 ^{\nu-\frac{5}{2}+\gamma_1}(t_2)R^{-2-a+\gamma_1}(t_2)(t_2-t_1)^{-\gamma_1/2}.
\end{aligned}
\end{equation*}
\end{itemize}
In both cases, we have
$$J_{11}\lesssim \mu_0 ^{\nu-\frac{5}{2}+\gamma_1}(t_2)R^{-2-a+\gamma_1}(t_2)(t_2-t_1)^{1-\gamma_1/2}.$$

For $J_{12}$, we evaluate
\begin{equation*}
\begin{aligned}
&~\int_{t_1-(t_2-t_1)}^{t_1}\int_{\R^3} |G(x-w,t_1-s)|f(w,s)dwds\\
\lesssim&~ \int_{t_1-(t_2-t_1)}^{t_1} \mu_0 ^{\nu-\frac52}(s)R^{-2-a}(s)\int_{|\tilde w|\leq \frac{2 \mu_0 (s)R(s)}{\sqrt{t_1-s}}} e^{-\frac{|\tilde x-\tilde w|^2}{4}}d \tilde w ds\\
\lesssim&~ \int_{t_1-(t_2-t_1)}^{t_1} \mu_0 ^{\nu-\frac52}(s)R^{-2-a}(s)\left(\frac{2\mu_0(s)R(s)}{\sqrt{t_1-s}}\right)^{\gamma_1} ds\\
\lesssim&~ \mu_0 ^{\nu-\frac{5}{2}+\gamma_1}(t_2)R^{-2-a+\gamma_1}(t_2)(t_2-t_1)^{1-\gamma_1/2},
\end{aligned}
\end{equation*}
where $\gamma_1\in(0,1)$. Similarly, we have
$$\int_{t_1-(t_2-t_1)}^{t_1}\int_{\R^3} |G(x-w,t_2-s)|f(w,s)dwds\lesssim \mu_0 ^{\nu-\frac{5}{2}+\gamma_1}(t_2)R^{-2-a+\gamma_1}(t_2)(t_2-t_1)^{1-\gamma_1/2}.$$
Thus we conclude that
$$J_{12}\lesssim \mu_0 ^{\nu-\frac{5}{2}+\gamma_1}(t_2)R^{-2-a+\gamma_1}(t_2)(t_2-t_1)^{1-\gamma_1/2}.$$
Finally, for $J_{13}$,
\begin{equation*}
\begin{aligned}
J_{13}=&~\int_{t_1}^{t_2}\int_{\R^3} G(x-w,t_2-s)f(w,s)dwds\\
\lesssim&~ \int_{t_1}^{t_2} \mu_0 ^{\nu-\frac52}(s)R^{-2-a}(s) \int_{\tilde w\leq \frac{2 \mu_0 (s)R(s)}{\sqrt{t_2-s}}} e^{-\frac{|\tilde x-\tilde w|^2}{4}}d \tilde w ds\\
\lesssim&~\mu_0 ^{\nu-\frac{5}{2}+\gamma_1}(t_2)R^{-2-a+\gamma_1}(t_2)(t_2-t_1)^{1-\gamma_1/2}
\end{aligned}
\end{equation*}
follows from the same argument as before. This completes the proof of \eqref{outerholder-rhs1}.

\end{proof}

\medskip

\begin{proof}[Proof of Lemma \ref{lemma-rhs2}]
We first prove \eqref{outer-rhs2}. Similar to the proof of Lemma \ref{lemma-rhs1}, Duhamel's formula gives
\begin{equation*}%\label{duhamel-rhs2}
\begin{aligned}
|\psi(x,t)|\lesssim& \int_0^t \frac{\mu^{\nu_2}_0(s)}{(t-s)^{3/2}}\int_{\mu_0 (s)R(s)\leq \left|w\right|} \frac{e^{-\frac{|x-w|^2}{4(t-s)}}}{\left|w\right|^a} dwds\\
\lesssim& \int_0^t \frac{\mu^{\nu_2}_0(s)}{(t-s)^{a/2}}\int_{\frac{\mu_0 (s)R(s)}{\sqrt{t-s}}\leq\left|\tilde w\right|} \frac{e^{-\frac{|\tilde x-\tilde w|^2}{4}}}{\left|\tilde w\right|^a} d\tilde wds,\\
\end{aligned}
\end{equation*}
where $\tilde x=x(t-s)^{-1/2}$. Notice that for $A>0$ we have
$$
\int_{A\leq\left|\tilde w\right|} \frac{e^{-\frac{|\tilde x-\tilde w|^2}{4}}}{\left|\tilde w\right|^a} d\tilde w\lesssim \frac{1}{A^a}~\mbox{ for }~0<a\leq 2.
$$ So
\begin{equation}\label{lemma-rhs21}
\begin{aligned}
&~ \int_0^{t-\mu_0^2(t)R^2(t)} \frac{\mu^{\nu_2}_0(s)}{(t-s)^{a/2}}\int_{\frac{\mu_0 (s)R(s)}{\sqrt{t-s}}\leq\left|\tilde w\right|} \frac{e^{-\frac{|\tilde x-\tilde w|^2}{4}}}{\left|\tilde w\right|^a} d\tilde wds\\
\lesssim&~ \int_0^{t-\mu_0^2(t)R^2(t)} \frac{\mu^{\nu_2}_0(s)}{(t-s)^{a/2}}ds\\
\lesssim&~\mu_0^{\nu_2+\frac{2-a}{4k}}(0)
\end{aligned}
\end{equation}
and
\begin{equation}\label{lemma-rhs22}
\begin{aligned}
&~ \int^t_{t-\mu_0^2(t)R^2(t)} \frac{\mu^{\nu_2}_0(s)}{(t-s)^{a/2}}\int_{\frac{\mu_0 (s)R(s)}{\sqrt{t-s}}\leq\left|\tilde w\right|} \frac{e^{-\frac{|\tilde x-\tilde w|^2}{4}}}{\left|\tilde w\right|^a} d\tilde wds\\
\lesssim&~ \int^t_{t-\mu_0^2(t)R^2(t)} \frac{\mu^{\nu_2}_0(s)}{(\mu_0(s)R(s))^{a}}ds\\
\lesssim&~ \mu_0^{\nu_2+2-a}(0) R^{2-a}(0).
\end{aligned}
\end{equation}
By \eqref{lemma-rhs21}--\eqref{lemma-rhs22}, we conclude the validity of \eqref{outer-rhs2}.

To prove \eqref{outergradient-rhs2}, we have
\begin{equation*}%\label{duhamel-rhs2}
\begin{aligned}
|\nabla\psi(x,t)|\lesssim& \int_0^t \frac{\mu^{\nu_2}_0(s)}{(t-s)^{5/2}}\int_{\mu_0 (s)R(s)\leq \left|w\right|} \frac{e^{-\frac{|x-w|^2}{4(t-s)}}|x-w|}{\left|w\right|^a} dwds\\
\lesssim& \int_0^t \frac{\mu^{\nu_2}_0(s)}{(t-s)^{(a+1)/2}}\int_{\frac{\mu_0 (s)R(s)}{\sqrt{t-s}}\leq\left|\tilde w\right|} \frac{e^{-\frac{|\tilde x-\tilde w|^2}{4}}|\tilde x-\tilde w|}{\left|\tilde w\right|^a} d\tilde wds\\
\lesssim& \int_0^{t-\mu^2_0(t)R^2(t)} \frac{\mu^{\nu_2}_0(s)}{(t-s)^{(a+1)/2}}ds + \int^t_{t-\mu^2_0(t)R^2(t)} \frac{\mu^{\nu_2}_0(s)}{(t-s)^{(a+1)/2}}\frac{(t-s)^{a/2}}{(\mu_0(s)R(s))^a}ds\\
\lesssim& ~\mu_0^{\nu_2+\frac{1-a_2}{4k}}(0).
\end{aligned}
\end{equation*}
All the rest estimates can be proved similarly.

%Now we prove the gradient estimate \eqref{outergradient-rhs2}. By the heat kernel, we have
%\begin{equation*}
%\begin{aligned}
%|\nabla \psi(x,t)|\lesssim& ~\int_0^t \frac{\mu_0 ^{\nu_2}(s)}{(t-s)^{\frac{n+2}{2}}}\int_{\mu_0 (s)R(s)\leq \left|(w_r,w_z)-\xi(s)\right|\leq 2\delta\sqrt{T-s}} \frac{e^{-\frac{|x-w|^2}{4(t-s)}}|x-w|}{|(w_r,w_z)-\xi(s)|^2} dw ds\\
%=&~\int_0^t \frac{\mu_0 ^{\nu_2}(s)}{(t-s)^{3/2}}\int_{\frac{\mu_0 (s)R(s)}{\sqrt{t-s}}\leq\left|(\tilde w_r,\tilde w_z)-\frac{\xi(s)}{\sqrt{t-s}}\right|\leq \frac{2\delta\sqrt{T-s}}{\sqrt{t-s}}} \frac{e^{-\frac{|\tilde x-\tilde w|^2}{4}}|\tilde x-\tilde w|}{\left|(\tilde w_r,\tilde w_z)-\frac{\xi(s)}{\sqrt{t-s}}\right|^2} d\tilde w ds\\
%\lesssim&~\int_0^t \frac{\mu_0 ^{\nu_2}(s)}{(t-s)^{3/2}}\frac{t-s}{\mu_0 ^2(s)R^2(s)} ds\\
%\lesssim&~\mu_0 ^{\nu_2-2}(t)R^{-2}(t)\sqrt{t}.
%\end{aligned}
%\end{equation*}
\end{proof}

\medskip

%%%%%%%%%%%%%%%%%%%%%%%%%%%%%%%%%%%%%%%%%%%%%%%%%%%%

\bigskip

\noindent
{\bf Acknowledgements:} M. del Pino has been supported by a UK Royal Society
Research Professorship. M. Musso is
partly supported by EPSRC of UK.  The  research  of J.~Wei is partially supported by NSERC of Canada.

\end{document}